\documentclass{amsart}

\usepackage{amssymb}

\usepackage{amsmath,amscd,graphicx,amsthm}

\newtheorem{theorem}{Theorem}
\newtheorem{lemma}[theorem]{Lemma}
\newtheorem{proposition}[theorem]{Proposition}

\theoremstyle{definition}

\theoremstyle{remark}
\newtheorem{remark}[theorem]{Remark}

\numberwithin{equation}{section}
\numberwithin{theorem}{section}

\def\A{{\mathcal A}}
\def\AA{{\mathbb A}}
\def\B{{\mathcal B}}
\def\C{{\mathbb C}}
\def\CC{{\mathcal C}}
\def\D{{\mathfrak d}}

\def\FF{{\mathcal F}}
\def\FFF{{\mathcal F}}
\def\G{{\mathcal G}}
\def\GCC{{\G\CC}}

\def\LL{{\Lambda}}

\def\N{\mathcal N}
\def\O{{\mathcal O}}
\def\P{{\mathcal P}}

\def\TE{{\mathcal S}}
\def\T{{\mathbb T}}
\def\TA{{\mathcal T}}

\def\UU{\overline{\A}}

\def\Z{{\mathbb Z}}

\def\b{\mathfrak b}

\def\fy{\varphi}
\def\g{\mathfrak g}

\def\h{\mathfrak h}

\def\n{\mathfrak n}
\def\one{\mathbf 1}

\def\q{{\bf q}}

\def\wB{{\widetilde{B}}}

\def\wx{{\widetilde{\bf x}}}
\def\x{{\bf x}}

\def\Ad{\operatorname{Ad}}

\def\End{\operatorname{End}}

\def\Id{{\operatorname {Id}}}

\def\Mat{\operatorname{Mat}}

\def\Poi{{\{\cdot,\cdot\}}}

\def\Van{\operatorname{Van}}

\def\ad{\operatorname{ad}}

\def\deg{{\operatorname{deg}}}
\def\diag{\operatorname{diag}}

\def\rank{\operatorname{rank}}

\def\dnabla{{\raisebox{2pt}{$\bigtriangledown$}}\negthinspace}

\begin{document}
\title[Drinfeld double of $GL_n$ and generalized cluster structures]{Drinfeld double of $GL_n$ and generalized cluster structures}

\author{M. Gekhtman}

\address{Department of Mathematics, University of Notre Dame, Notre Dame,
IN 46556}
\email{mgekhtma@nd.edu}

\author{M. Shapiro}
\address{Department of Mathematics, Michigan State University, East Lansing,
MI 48823}
\email{mshapiro@math.msu.edu}

\author{A. Vainshtein}
\address{Department of Mathematics \& Department of Computer Science, University of Haifa, Haifa,
Mount Carmel 31905, Israel}
\email{alek@cs.haifa.ac.il}

\begin{abstract}
We construct a generalized cluster structure compatible with  the Poisson bracket on the Drinfeld double of the standard
Poisson-Lie group $GL_n$ and derive from it a generalized cluster structure on $GL_n$ compatible with the push-forward of the Poisson bracket  on the  
dual Poisson--Lie group.
\end{abstract}


\maketitle

\tableofcontents

\section{Introduction}
\label{intro}

The connection between cluster algebras and Poisson structures is documented in \cite{GSVb}. Among the most important examples in which this connection has been utilized are coordinate rings of double Bruhat cells in semisimple Lie groups equipped with (the restriction of) the standard Poisson-Lie structure. In \cite{GSVb}, we applied our technique of constructing a cluster structure compatible with a given Poisson structure in this situation and recovered the cluster structure built in \cite{CAIII}. The standard Poisson-Lie structure is a particular case of a family of Poisson-Lie structures corresponding to quasi-triangular Lie bialgebras. Such structures are associated with solutions to the classical Yang-Baxter equation. Their complete classification was obtained by Belavin and Drinfeld in \cite{BD} using certain combinatorial data defined in terms of the corresponding root system. In \cite{GSVMMJ} we conjectured that any such solution gives rise to a compatible cluster structure on the Lie group and provided several examples supporting this conjecture. 
Recently \cite{GSVPNAS,GSVMem}, we constructed the cluster structure corresponding to the Cremmer--Gervais Poisson structure in $GL_n$ for any $n$.

As we established in \cite{GSVMem}, the construction of 
cluster structures on a simple Poisson-Lie group $\G$  relies on properties of the Drinfeld double $D(\G)$. Moreover, in the Cremmer-Gervais case generalized determinantal identities on which cluster transformations are modeled can be extended to identities valid in the double. It is not too far-fetched then to suspect that 
there exists a cluster structure on $D(\G)$ compatible with the Poisson-Lie bracket induced by the Poisson-Lie bracket on $\G$.
However, an interesting phenomenon was observed even in the first nontrivial example of $D(GL_2)$: although we were able to construct a log-canonical regular coordinate chart in terms of which all standard coordinate functions are expressed as (subtraction free) Laurent polynomials, it is not possible to define cluster transformations in such a way that all cluster variables which one expects to be mutable are transformed into regular functions. This problem is resolved, however, if one is allowed to use {\em generalized cluster transformations} previously considered in \cite{GSV1,GSVb} and, more recently, axiomatized in  \cite{CheSha}. 

In this paper, we describe
such a generalized cluster structure on the Drinfeld double in the case of the standard Poisson-Lie group $GL_n$.
Using this structure, one can recover the standard cluster structure on $GL_n$.
Furthermore, there is a well-known map from the dual Poisson--Lie group $GL_n^*\subset D(GL_n)$ to an open dense subset $GL_n^\dag$ of $GL_n$ (see \cite{r-sts}). The push-forward of the Poisson structure on $GL_n^*$ extends from $GL_n^\dag$ to the whole $GL_n$ (the resulting Poisson structure is not 
Poisson--Lie). We define a generalized cluster structure on $GL_n$ compatible with the 
latter Poisson structure, using a special seed of the generalized cluster structure on $D(GL_n)$.
 Note that the log-canonical basis suggested in \cite{Bra} is different from the one
constructed here and does not lead to a regular cluster structure. 

Section~\ref{section2} contains the definition of a generalized cluster structure of geometric type, as well as 
 properties of such structures in the rings of regular functions of algebraic varieties. This includes
several basic results on compatibility  with Poisson brackets and toric actions. 
These statements are natural extensions of the corresponding results on ordinary cluster structures, and their proofs are obtained via minor modifications.
 Section~\ref{section2} also contains basic information on Poisson--Lie groups and the corresponding Drinfeld doubles.

Section~\ref{section3} contain the main results of the paper. The initial log-canonical basis is described in Section~\ref{logcan}, the corresponding quiver 
is presented in Section~\ref{init}, and the generalized exchange relation is given in Section~\ref{exrelsec}. The main results are stated in Section~\ref{main}
and include two theorems: Theorem~\ref{structure} treats the generalized cluster structure on the Drinfeld double, while Theorem~\ref{dualstructure} deals
with the generalized cluster structure on $GL_n$ compatible with the push-forward of the bracket on the dual Poisson--Lie group. 
Section~\ref{outline} explains the main steps
in the proof of Theorem~\ref{structure}, and shows how it yields Theorem~\ref{dualstructure}.

Section~\ref{novoeslovo} has an independent value. Recall that originally upper cluster algebras were defined over the ring of Laurent polynomials in stable variables.
In this section we prove that upper cluster algebras over subrings of this ring retain all properties of usual upper cluster algebras, and under certain
coprimality conditions coincide with the intersection of rings of Laurent polynomials in a finite collection of clusters. 

The  log-canonicity of the suggested initial basis is proved in Section~\ref{prlogcan}. The proofs rely on invariance properties of the elements
of the basis.

The first part of Theorem~\ref{structure} is proved in Section~\ref{prgcs}. We start with studying a left--right toric action by diagonal matrices, see Section~\ref{torac}.
In Section~\ref{cmpty} we prove the compatibility of our generalized cluster structure with the standard Poisson--Lie structure on the Drinfeld double. To prove the
regularity of the generalized cluster structure, we verify the coprimality conditions in the initial cluster and its neighbors.

The second part of Theorem~\ref{structure} is proved in Section~\ref{upeqreg}. 

Section~\ref{aux} contains proofs of various auxiliary statements in matrix theory that are used in other sections.

A short preliminary version of this paper was published in~\cite{GSVcr}.

\section{Preliminaries}\label{section2}

\subsection{Generalized cluster structures of geometric type and compatible Poisson brackets}
\label{SecPrel}

Let $\widetilde{B}=(b_{ij})$ be an $N\times(N+M)$ integer matrix
whose principal part $B$ is skew-symmetrizable (recall that the principal part of a rectangular matrix  
is its maximal leading square submatrix). 
Let $\FFF$ be the field of rational functions in $N+M$ independent variables
with rational coefficients. There are $M$  distinguished variables; 
they are denoted
$x_{N+1},\dots,x_{N+M}$ and called {\em stable\/}. A stable variable $x_{j}$ is called {\em isolated\/} if $b_{ij}=0$ 
for $1\le i\le N$. The {\it coefficient group\/} is a free multiplicative abelian group of Laurent monomials in stable variables, 
and its integer group ring is $\bar\AA=\Z[x_{N+1}^{\pm1},\dots,x_{N+M}^{\pm1}]$ (we write
$x^{\pm1}$ instead of $x,x^{-1}$).

For each $i$, $1\le i\le N$, fix a factor $d_i$ of $\gcd\{b_{ij}: 1\le j\le N\}$. 
A {\em seed\/} (of {\em geometric type\/}) in $\FFF$ is a triple
$\Sigma=(\x,\widetilde{B},\P)$,
where $\x=(x_1,\dots,x_N)$ is a transcendence basis of $\FFF$ over the field of
fractions of  $\bar\AA$ and $\P$ is a set of $n$ {\em strings}. The $i$th string is a collection of 
monomials  
$p_{ir}\in\AA=\Z[x_{N+1},\dots,x_{N+M}]$, $0\le r\le d_i$, such that  
$p_{i0}=p_{id_i}=1$; 
it is called {\em trivial\/} if $d_i=1$, and hence both elements of the string are equal to one.

Matrices $B$ and $\wB$ are called the
{\it exchange matrix\/} and the {\it extended exchange matrix}, respectively. The $N$-tuple  $\x$ is called a {\em cluster\/}, and its elements
$x_1,\dots,x_N$ are called {\em cluster variables\/}. The monomials $p_{ir}$ are called {\em exchange coefficients}.
We say that
$\widetilde{\x}=(x_1,\dots,x_{N+M})$ is an {\em extended
cluster\/}, and $\widetilde\Sigma=(\wx,\widetilde{B},\P)$ is an {\em extended seed}.

Given a seed as above, the {\em adjacent cluster\/} in direction $k$, $1\le k\le N$,
is defined by
$\x'=(\x\setminus\{x_k\})\cup\{x'_k\}$,
where the new cluster variable $x'_k$ is given by the {\em generalized exchange relation}
\begin{equation}\label{exchange}
x_kx'_k=\sum_{r=0}^{d_k}p_{kr}u_{k;>}^r v_{k;>}^{[r]}u_{k;<}^{d_k-r}v_{k;<}^{[d_k-r]}
\end{equation}
with {\em cluster $\tau$-monomials\/} $u_{k;>}$ and $u_{k;<}$, $1\le k\le N$, defined by
\begin{equation*}
\begin{aligned}
u_{k;>}&=\prod\{x_i^{b_{ki}/d_k}: 1\le i\le N, b_{ki}>0\}, \\
u_{k;<}&=\prod\{x_i^{-b_{ki}/d_k}: 1\le i\le N, b_{ki}<0\},
\end{aligned}
\end{equation*}
and {\em stable $\tau$-monomials\/}
$v_{k;>}^{[r]}$ and $v_{k;<}^{[r]}$, $1\le k\le N$, $0\le r\le d_k$, defined by
 \begin{equation*}
\begin{aligned}
v_{k;>}^{[r]}&=\prod\{x_i^{\lfloor rb_{ki}/d_k\rfloor}: N+1\le i\le N+M,  b_{ki}>0\},\\ 
v_{k;<}^{[r]}&=\prod\{x_i^{\lfloor -rb_{ki}/d_k\rfloor}: N+1\le i\le N+M,  b_{ki}<0\};
 \end{aligned}
\end{equation*}
here, as usual, the product over the empty set is assumed to be
equal to~$1$. In what follows we write $v_{k;>}$ instead of $v_{k;>}^{[d_k]}$ and $v_{k;<}$ instead of $v_{k;<}^{[d_k]}$.
The right hand side of~\eqref{exchange} is called a {\it generalized exchange polynomial}.

We say that $\wB'$ is
obtained from $\wB$ by a {\em matrix mutation\/} in direction $k$ if
\begin{equation}
\label{MatrixMutation}
b'_{ij}=\begin{cases}
         -b_{ij}, & \text{if $i=k$ or $j=k$;}\\
                 b_{ij}+\displaystyle\frac{|b_{ik}|b_{kj}+b_{ik}|b_{kj}|}2,
                                                  &\text{otherwise.}
        \end{cases}
\end{equation}
Note that $b_{ij}=0$ for $1\le i\le N$ implies
$b'_{ij}=0$ for $1\le i\le N$; in other words, the set of isolated variables does not depend on the cluster. 
Moreover, $\gcd\{b_{ij}:1\le j\le N\}=\gcd\{b'_{ij}:1\le j\le N\}$, and hence the $N$-tuple $(d_1,\dots,d_N)$
retains its property.

The {\em exchange coefficient mutation\/} in direction $k$ is given by
\begin{equation}
\label{CoefMutation}
 p'_{ir}=\begin{cases}
          p_{i,d_i-r}, & \text{if $i=k$;}\\
           p_{ir}, &\text{otherwise.}
        \end{cases}
\end{equation}

\begin{remark}
The definition above is adjusted from an earlier definition of generalized
cluster structures given in~\cite{CheSha}. In that case a somewhat
less involved construction was modeled on examples appearing in a study
of triangulations of surfaces with orbifold points, and coefficients of
exchange polynomials were assumed to be  elements of an arbitrary
tropical semifield. In contrast, our main example forces us to consider a situation in which
the coefficients have to be realized as regular functions on an
underlying variety which results in a complicated definition above.
More exactly, if one defines coefficients $p_{i;l}$ in~\cite{CheSha} as $p_{il}v_{i;>}^{[l​]}v_{i;<}^{[d_i-l]}$, then our 
exchange coefficient mutation rule~\eqref{CoefMutation} becomes a specialization of the general rule (2.5) in~\cite{CheSha}.
As we will explain in future publications, many examples of this sort
arise in an investigation of exotic cluster structures on Poisson--Lie groups.
\end{remark}

Consider Laurent monomials in stable variables
\begin{equation*}
q_{ir}=\frac{v_{i;>}^rv_{i;<}^{d_i-r}}{\left(v_{i;>}^{[r]}v_{i;<}^{[d_i-r]}\right)^{d_i}}, \qquad 0\le r\le d_i, 1\le i\le N. 
\end{equation*}
By~\eqref{MatrixMutation}, $b_{ij}=b'_{ij}\bmod d_i$
for $1\le j\le N+M$. Consequently, the mutation rule for $q_{ir}$ is the same as~\eqref{CoefMutation}. In what follows we will 
use Laurent monomials $\hat p_{ir}=p_{ir}^{d_i}/q_{ir}$. One can rewrite the generalized exchange relation~\eqref{exchange} in
terms of the $\hat p_{ir}$ as follows:
\begin{equation}\label{modexchange}
x_kx'_k=\sum_{r=0}^{d_k}\left(\hat p_{kr}v_{k;>}^rv_{k;<}^{d_k-r}\right)^{1/d_k}  u_{k;>}^r u_{k;<}^{d_k-r};
\end{equation}
note that $\left(\hat p_{kr}v_{k;>}^rv_{k;<}^{d_k-r}\right)^{1/d_k}$ is a monomial in $\AA$ for $0\le r\le d_k$.

In certain cases, it is convenient to represent the data $(\wB, d_1,\dots,d_N)$ by a quiver.
Assume that the {\it modified extended exchange matrix\/} $\widehat B$ obtained from 
$\wB$ by replacing
each $b_{ij}$ by $b_{ij}/d_i$ for $1\le j\le N$ and retaining it for $N+1\le j\le N+M$ has a skew-symmetric 
principal part;
we say that the corresponding quiver $Q$ with vertex multiplicities $d_1,\dots,d_N$
represents $(\wB, d_1,\dots,d_N)$ and write
$\Sigma=(\x,Q,\P)$. Vertices that correspond to cluster variables are called {\it mutable}, 
those that correspond to stable variables are called {\it frozen}.
A mutable vertex with $d_i\ne 1$ is called {\em special\/}, and $d_i$ is said to be its {\em multiplicity}. 
A frozen vertex corresponding to an isolated variable is called {\em isolated}. 

 Given a seed $\Sigma=(\x,\widetilde{B},\P)$, we say that a seed
$\Sigma'=(\x',\widetilde{B}',\P')$ is {\em adjacent\/} to $\Sigma$ (in direction
$k$) if $\x'$, $\widetilde{B}'$ and $\P'$ are as above. 
Two seeds are {\em mutation equivalent\/} if they can
be connected by a sequence of pairwise adjacent seeds. 
The set of all seeds mutation equivalent to $\Sigma$ is called the {\it generalized cluster structure\/} 
(of geometric type) in $\FFF$ associated with $\Sigma$ and denoted by $\GCC(\Sigma)$; in what follows, 
we usually write $\GCC(\wB,\P)$, or even just $\GCC$ instead. Clearly, by taking $d_i=1$ for $1\le i\le N$, and hence making all strings trivial, we get an ordinary cluster structure. Indeed, in this case the right hand side of the generalized exchange relation \eqref{exchange} contains two terms; furthermore, 
$u_{k;>}^0=u_{k;<}^0=v_{k;>}^{[0]}=v_{k;<}^{[0]}=1$ and
\begin{equation*}
\begin{aligned}
u_{k;>}^1 v_{k;>}^{[1]}&=\prod \{x_i^{b_{ki}} : 1\le i\le N+M, b_{ki}>0\},\\
u_{k;<}^1 v_{k;<}^{[1]}&=\prod \{x_i^{-b_{ki}} : 1\le i\le N+M, b_{ki}<0\}.
\end{aligned}
\end{equation*}
Consequently, in this case \eqref{exchange} coincides with the ordinary exchange relation, while the exchange coefficient mutation \eqref{CoefMutation} becomes trivial.
Similarly to the case of
ordinary cluster structures, we will associate to $\GCC(\wB,\P)$ a labeled $N$-regular tree $\T_N$ whose vertices correspond to seeds, and edges 
correspond to the adjacency of seeds.

Fix a ground ring $\widehat{\AA}$ such that $\AA\subseteq\widehat\AA\subseteq\bar\AA$. 
We associate with $\GCC(\wB,\P)$ two algebras of rank $N$ over $\widehat{\AA}$:
the {\em generalized cluster algebra\/} $\A=\A(\GCC)=\A(\wB,\P)$, which 
is the $\widehat{\AA}$-subalgebra of $\FF$ generated by all cluster
variables from all seeds in $\GCC(\wB,\P)$, and the {\it generalized upper cluster algebra\/}
$\UU=\UU(\GCC)=\UU(\wB,\P)$, which is the intersection of the rings of Laurent polynomials over $\widehat{\AA}$ in cluster variables
taken over all seeds in $\GCC(\wB,\P)$. The generalized  {\it Laurent phenomenon\/} \cite{CheSha}
claims the inclusion $\A(\GCC)\subseteq\UU(\GCC)$. 

\begin{remark} Note that our definition of the generalized cluster algebra is slightly more general than the one used in \cite{CheSha}. However, the proof in \cite{CheSha} utilizes the Caterpillar Lemma
of Fomin and Zelevinsky (see \cite{FoZe}) and follows their standard pattern of reasoning; it can be repeated {\it verbatim\/} in our case as well.
\end{remark}

Let $V$ be a quasi-affine variety over $\C$, $\C(V)$ be the field of rational functions on $V$, and
$\O(V)$ be the ring of regular functions on $V$. Let $\GCC$ be a generalized cluster structure in $\FF$ as above.
Assume that $\{f_1,\dots,f_{N+M}\}$ is a transcendence basis of $\C(V)$. Then the map $\theta: x_i\mapsto f_i$,
$1\le i\le N+M$, can be extended to a field isomorphism $\theta: \FF_\C\to \C(V)$,  
where $\FF_\C=\FF\otimes\C$ is obtained from $\FF$ by extension of scalars.
The pair $(\GCC,\theta)$ is called a generalized cluster structure {\it in\/}
$\C(V)$, $\{f_1,\dots,f_{N+M}\}$ is called an extended cluster in
 $(\GCC,\theta)$.
Sometimes we omit direct indication of $\theta$ and say that $\GCC$ is a generalized cluster structure {\em on\/} $V$. 
A generalized cluster structure $(\GCC,\theta)$ is called {\it regular\/}
if $\theta(x)$ is a regular function for any cluster variable $x$. 
The two algebras defined above have their counterparts in $\FF_\C$ obtained by extension of scalars; they are
denoted $\A_\C$ and $\UU_\C$.  As it is explained in~\cite[Section 3.4]{GSVb}, the natural choice of 
the ground ring for $\A_\C$ and $\UU_\C$
is 
\begin{equation}\label{hata}
\widehat\AA=\C[x_{N+1}^{\pm1},\dots,x_{N+M'}^{\pm1}, x_{N+M'+1},\dots,x_{N+M}],
\end{equation} 
where $\theta(x_{N+i})$ does not vanish on $V$ if and only if $1\le i\le M'$.
If, moreover, the field isomorphism $\theta$ can be restricted to an isomorphism of 
$\A_\C$ (or $\UU_\C$) and $\O(V)$, we say that 
$\A_\C$ (or $\UU_\C$) is {\it naturally isomorphic\/} to $\O(V)$.

 The following statement is a direct corollary of the natural extension of the Starfish Lemma (Proposition~3.6 in \cite{FoPy}) 
to the case of generalized cluster structures.
The proof of this extension literally follows the proof of the Starfish Lemma.

\begin{proposition}\label{regcoin}
Let $V$ be a Zariski open subset in $\C^{N+M}$ and $\GCC=(\GCC(\wB,\P),\theta)$ be a generalized cluster structure in $\C(V)$  
with $N$ cluster and $M$ stable variables such that

{\rm(i)} there exists an extended cluster $\wx=(x_1,\dots,x_{N+M})$ in $\GCC$ such that $\theta(x_i)$ is
regular on $V$ for $1\le i\le N+M$, and $\theta(x_i)$ and $\theta(x_j)$ are coprime in $\O(V)$ for $1\le i\ne j\le N+M$;

{\rm(ii)} for any cluster variable $x_k'$, $1\le k\le N$, obtained via the generalized exchange relation~\eqref{exchange} 
applied to $\wx$, $\theta(x_k')$ is regular on $V$ and coprime in $\O(V)$ with $\theta(x_k)$.

\noindent Then $\GCC$ is a regular generalized cluster structure. If additionally

{\rm(iii)} each regular function on $V$ belongs to $\theta(\UU_\C(\GCC))$,

\noindent then $\UU_\C(\GCC)$ is naturally isomorphic to $\O(V)$.
\end{proposition}

\begin{remark} Conditions of the Starfish Lemma in our case are satisfied since $\O(V)$ is a unique factorization domain.
\end{remark}

Let $\Poi$ be a Poisson bracket on the ambient field $\FFF$, and $\GCC$ be a generalized cluster structure in $\FFF$. 
We say that the bracket and the generalized cluster structure are {\em compatible\/} if any extended
cluster $\widetilde{\x}=(x_1,\dots,x_{N+M})$ is {\em log-canonical\/} with respect to $\Poi$, that is,
$\{x_i,x_j\}=\omega_{ij} x_ix_j$,
where $\omega_{ij}\in\Z$ are constants for all $i,j$, $1\le i,j\le N+M$.

Let $\Omega=(\omega_{ij})_{i,j=1}^{N+M}$.  The following proposition can be considered as a natural extension of Proposition~2.3 in \cite{GSVMem}. 
The proof is similar to the proof of Theorem~4.5 in \cite{GSVb}.

\begin{proposition}\label{compatchar}
Assume that $\wB\Omega=[\Delta\;\; 0]$ for a non-degenerate diagonal matrix $\Delta$, and all Laurent polynomials $\hat p_{ir}$ are Casimirs of the bracket $\Poi$.
Then $\rank\wB=N$, and the bracket $\Poi$ is compatible with $\GCC(\wB,\P)$.
\end{proposition}

The notion of compatibility  extends to Poisson brackets on $\FF_\C$ without any changes.

 Fix an arbitrary extended cluster
$\wx=(x_1,\dots,x_{N+M})$ and define a {\it local toric action\/} of rank $s$ as a map 
$\TA^W_{\q}:\FF_\C\to
\FF_\C$ given on the generators of $\FF_\C=\C(x_1,\dots,x_{N+M})$ by the formula 
\begin{equation}
\TA^W_{\q}(\wx)=\left ( x_i \prod_{\alpha=1}^s q_\alpha^{w_{i\alpha}}\right )_{i=1}^{N+M},\qquad
\q=(q_1,\dots,q_s)\in (\C^*)^s,
\label{toricact}
\end{equation}
where $W=(w_{i\alpha})$ is an integer $(N+M)\times s$ {\it weight matrix\/} of full rank, and extended naturally to the whole $\FF_\C$. 

Let $\wx'$ be another extended cluster in $\GCC$, then the corresponding local toric action defined by the weight matrix $W'$
is {\it compatible\/} with the local toric action \eqref{toricact} if it commutes with the sequence of (generalized) cluster transformations that takes $\wx$ to $\wx'$.
 If local toric actions at all clusters are compatible, they define a {\it global toric action\/} $\TA_{\q}$ on $\FF_\C$ called a 
$\GCC$-extension of the local toric action \eqref{toricact}. 
The following proposition can be viewed as a natural extension of Lemma~2.3 in \cite{GSV1} and is proved in a similar way.

\begin{proposition}\label{globact}
The local toric action~\eqref{toricact} is uniquely $\GCC$-extendable
to a global action of $(\C^*)^s$  if  $\wB W = 0$ and all Casimirs $\hat p_{ir}$
are invariant under~\eqref{toricact}.
 \end{proposition}

\subsection{Standard Poisson-Lie group $\G$ and its Drinfeld double}
\label{double}

A reductive complex Lie group $\G$ equipped with a Poisson bracket $\Poi$ is called a {\em Poisson--Lie group\/}
if the multiplication map
$\G\times \G \ni (X,Y) \mapsto XY \in \G$
is Poisson. Denote by 
$\langle \ , \ \rangle$ an invariant nondegenerate form on
$\g$, and by $\nabla^R$, $\nabla^L$ the right and
left gradients of functions on $\G$ with respect to this form defined by
\begin{equation*}
\left\langle \nabla^R f(X),\xi\right\rangle=\left.\frac d{dt}\right|_{t=0}f(Xe^{t\xi}),  \quad
\left\langle \nabla^L f(X),\xi\right\rangle=\left.\frac d{dt}\right|_{t=0}f(e^{t\xi}X)
\end{equation*}
for any $\xi\in\g$, $X\in\G$.

Let $\pi_{>0}, \pi_{<0}$ be projections of  
$\g$ onto subalgebras spanned by positive and negative roots, $\pi_0$ be the projection onto the Cartan 
subalgebra $\h$, and let $R=\pi_{>0} - \pi_{<0}$. 
The {\em standard Poisson-Lie bracket\/} $\Poi_r$ on $\G$  can be written as  
\begin{equation}
\{f_1,f_2\}_r = \frac 1 2  \left( \left\langle R(\nabla^L f_1), \nabla^L f_2 \rangle - \langle R(\nabla^R f_1), \nabla^R f_2 \right\rangle \right).
\label{sklyabra}
\end{equation}

The standard Poisson--Lie structure is a particular case of Poisson--Lie structures corresponding to
quasitriangular Lie bialgebras. For a detailed exposition of these structures see, e.~g., 
\cite[Ch.~1]{CP}, \cite{r-sts} and \cite{Ya}.

Following \cite{r-sts}, let us recall the construction of {\em the Drinfeld double}. The double of $\g$ is 
$D(\g)=\g  \oplus \g$ equipped with an invariant nondegenerate bilinear form
$\langle\langle (\xi,\eta), (\xi',\eta')\rangle\rangle = \langle \xi, \xi'\rangle - \langle \eta, \eta'\rangle$. 
Define subalgebras $\D_\pm$ of $D(\g)$ by
$\D_+=\{( \xi,\xi) : \xi \in\g\}$ and $\D_-=\{ (R_+(\xi),R_-(\xi)) : \xi \in\g\}$,
where $R_\pm\in \End\g$ is given by $R_\pm=\frac{1}{2} ( R \pm \Id)$. 
The operator $R_D= \pi_{\D_+} - \pi_{\D_-}$ can be used to define 
a Poisson--Lie structure on $D(\G)=\G\times \G$, the double of the group $\G$, via
\begin{equation}
\{f_1,f_2\}_D = \frac{1}{2}\left (\left\langle\left\langle R_D(\dnabla^L f_1), \dnabla{^L} f_2 \right\rangle\right\rangle 
- \left\langle\left\langle R_D(\dnabla^R f_1), \dnabla^R f_2 \right\rangle\right\rangle \right),
\label{sklyadouble}
\end{equation}
where $\dnabla^R$ and $\dnabla^L$ are right and left gradients with respect to $\langle\langle \cdot ,\cdot \rangle\rangle$.
The diagonal subgroup $\{ (X,X)\ : \ X\in \G\}$ is a Poisson--Lie subgroup of $D(\G)$ (whose Lie algebra is $\D_+$) naturally isomorphic
to $(\G,\Poi_r)$.

The group $\G^*$  whose Lie algebra is $\D_-$ is a Poisson-Lie subgroup of $D(\G)$ called {\em the dual Poisson-Lie group of $\G$}.
The map $D(\G) \to \G$ given by $(X,Y) \mapsto U=X^{-1} Y$ induces another Poisson bracket on $\G$, see~\cite{r-sts}; we denote this bracket $\Poi_*$. 
The image of the restriction of this map to $\G^*$ is denoted $\G^\dag$. Symplectic leaves on 
$(\G,\Poi_*)$ were studied in \cite{EvLu}.

In this paper we only deal with the case of $\G=GL_n$. In that case $\G^\dag$ is the non-vanishing locus of trailing principal minors $\det U_{[i,n]}^{[i,n]}$
(here and in what follows we write $[a,b]$ to denote the set $\{i\in \Z : a\le i\le b\}$).
The bracket \eqref{sklyadouble} takes the form
\begin{equation}\label{sklyadoubleGL}
\begin{split}
\{f_1,f_2\}_D = &\langle R_+(E_L f_1), E_L f_2\rangle -  \langle R_+(E_R f_1), E_R f_2\rangle\\
&+  \langle X\nabla_X  f_1, Y\nabla_Y f_2\rangle - \langle\nabla_X  f_1\cdot X, \nabla_Y f_2 \cdot Y\rangle,
\end{split}
\end{equation}
where $\nabla_X f=\left(\frac{\partial f}{\partial x_{ji}}\right)_{i,j=1}^n$,
$\nabla_Y f=\left(\frac{\partial f}{\partial y_{ji}}\right)_{i,j=1}^n$,
and
\begin{equation}\label{erel}
E_R = X \nabla_X + Y\nabla_Y, \quad E_L =  \nabla_X X+ \nabla_Y Y.
\end{equation}
So, \eqref{sklyadoubleGL} can be rewritten as
\begin{equation}\label{sklyadoubleGL1}
\begin{split}
\{f_1,f_2\}_D = &\left\langle R_+(E_L f_1), E_L f_2\right\rangle -  \left\langle R_+(E_R f_1), E_R f_2\right\rangle\\
&+  \left\langle E_R  f_1, Y\nabla_Y f_2\right\rangle - \left\langle E_L f_1, \nabla_Y f_2 \cdot Y\right\rangle.
\end{split}
\end{equation}
Further, if $\fy$ is a function of $U=X^{-1}Y$ then  $E_L\fy=[\nabla\fy,U]$ and $E_R\fy=0$, and hence for an arbitrary function $f$ on $D(GL_n)$ one has
\begin{equation}\label{sklyadoubleGL2}
\{\fy,f\}_D = \left\langle R_+([\nabla\fy,U]), E_L f\right\rangle - \left\langle [\nabla\fy,U], \nabla_Y f \cdot Y\right\rangle.
\end{equation}

\section{Main results}\label{section3}

\subsection{Log-canonical basis}
\label{logcan}

Let $(X,Y)$ be a point in the double $D(GL_n)$. For $k,l\ge 1$, $k+l\le n-1$ define a $(k+l)\times(k+l)$ matrix 
$$
F_{kl}=F_{kl}(X,Y)=\left[\begin{array}{cc}X^{[n-k+1,n]} & Y^{[n-l+1,n]}\end{array}\right]_{[n-k-l+1,n]}.
$$ 
 For $1\le j\le i\le n$ define an $(n-i+1)\times (n-i+1)$ matrix
$$
G_{ij}=G_{ij}(X)=X_{[i,n]}^{[j,j+n-i]}.
$$
For $1\le i\le j\le n$ define an $(n-j+1)\times (n-j+1)$ matrix
$$
H_{ij}=H_{ij}(Y)=Y_{[i,i+n-j]}^{[j,n]}.
$$
For  $k,l\ge 1$, $k+l\le n$ define an $n\times n$ matrix
\begin{equation*}\label{phidef}
\Phi_{kl}=\Phi_{kl}(X,Y)=
\left[\begin{array}{ccccc}(U^0)^{[n-k+1,n]}& U^{[n-l+1,n]} & (U^2)^{[n]} & \dots & (U^{n-k-l+1})^{[n]}\end{array}\right] 
\end{equation*}
where $U=X^{-1}Y$. 

\begin{remark}
\label{identify}
Note that the definition of $F_{kl}$ can be extended to the case $k+l=n$ yielding $F_{n-l,l}=X\Phi_{n-l,l}$.
One can also identify $F_{0 l}$ with $H_{n-l+1,n-l+1}$ and $F_{k 0}$ with $G_{n-k+1,n-k+1}$. 
Finally, it will be convenient, for technical reasons, to identify $G_{i,i+1}$ with $F_{n-i,1}$. 
\end{remark}

Denote 
\begin{gather*}
f_{kl}=\det F_{kl},\quad g_{ij}=\det G_{ij},\quad h_{ij}=\det H_{ij},\\  
\fy_{kl}=s_{kl}(\det X)^{n-k-l+1}\det\Phi_{kl}, 
\end{gather*}
$2n^2-n+1$ functions in total. 
 Here $s_{kl}$ is a sign defined as follows: 
\[
s_{kl}=\begin{cases}
                    (-1)^{k(l+1)}\qquad\qquad\qquad\qquad \text{for $n$  even},\\
										(-1)^{(n-1)/2+k(k-1)/2+l(l-1)/2} \quad \text{for $n$ odd}.
\end{cases}
\]										
It is periodic in $k+l$ with period~4 for 
$n$ odd and period~2 for $n$ even; $s_{n-l,l}=1$; $s_{n-l-1,l}=(-1)^l$ for $n$ odd and $s_{n-l-1,l}=(-1)^{l+1}$ for $n$ even;
$s_{n-l-2,l}=-1$ for $n$ odd; $s_{n-l-3,l}=(-1)^{l+1}$ for $n$ odd.
Note that the pre-factor in the definition of $\fy_{kl}$ is needed to obtain a regular function in matrix entries of $X$ and $Y$.

Consider the polynomial 
\[
\det(X+\lambda Y)=\sum_{i=0}^n \lambda^i s_ic_i(X,Y), 
\]
where $s_i=(-1)^{i(n-1)}$. 
Note that $c_0(X,Y)=\det X=g_{11}$ and  $c_n(X,Y)=\det Y=h_{11}$.

\begin{theorem}
\label{basis}
The family of functions $F_n=\{g_{ij}, h_{ij}, f_{kl}, \fy_{kl}, c_1,\ldots, c_{n-1} \}$ forms a log-canonical coordinate system with respect to the Poisson-Lie bracket  \eqref{sklyadoubleGL1} on $D(GL_n)$.
\end{theorem}

\subsection{Initial quiver}
\label{init}
The modified extended exchange matrix $\widehat{B}$ has a  skew-symmet\-ric principal part, and hence 
can be represented by a quiver. 
The quiver $Q_n$ contains $2n^2 - n +1$  vertices  labeled by the functions $g_{ij}, h_{ij}, f_{kl}, \fy_{kl}$ in the log-canonical basis $F_n$. 
The functions $c_1, \ldots, c_{n-1}$ correspond to isolated vertices. They are not connected to any of the other vertices, and will be not shown on figures.
The vertex $\fy_{11}$ is the only special vertex, and its multiplicity equals $n$.
The vertices $g_{i1}$, $1\le i\le n$, and $h_{1j}$, $1\le j\le n$, are frozen, so $N=2n^2-3n+1$ and $M=3n-1$. Below we describe $Q_n$ assuming that $n>2$.

Vertex $\fy_{kl}$, $k,l\ne1$, $k+l<n$, has degree~6. The edges pointing from $\fy_{kl}$ are $\fy_{kl}\to\fy_{k+1,l}$,
$\fy_{kl}\to \fy_{k,l-1}$, and $\fy_{kl}\to \fy_{k-1,l+1}$; the edges pointing towards $\fy_{kl}$ are $\fy_{k,l+1}\to\fy_{kl}$,
$\fy_{k+1,l-1}\to\fy_{kl}$, and $\fy_{k-1,l}\to\fy_{kl}$. Vertex $\fy_{kl}$, $k,l\ne1$, $k+l=n$, has degree~4. 
The edges pointing from $\fy_{kl}$ in this case are $\fy_{kl}\to \fy_{k,l-1}$ and $\fy_{kl}\to f_{k-1,l}$; the edges pointing towards $\fy_{kl}$ are 
 $\fy_{k-1,l}\to\fy_{kl}$ and $f_{k,l-1}\to \fy_{kl}$. 

Vertex $\fy_{k1}$, $k\in [2,n-2]$, has degree~4. The edges pointing from $\fy_{k1}$ are $\fy_{k1}\to\fy_{k-1,2}$ and
$\fy_{k1}\to \fy_{1k}$; the edges pointing towards $\fy_{k1}$ are $\fy_{k2}\to\fy_{k1}$ and 
$\fy_{1,k-1}\to\fy_{k1}$. Note that for $k=2$ the vertices $\fy_{k-1,2}$ and $\fy_{1k}$ coincide, hence for $n>3$
there are two edges pointing from $\fy_{21}$ to $\fy_{12}$. Vertex $\fy_{n-1,1}$ has degree 5. The edges pointing from $\fy_{n-1,1}$
are $\fy_{n-1,1}\to \fy_{1,n-1}$, $\fy_{n-1,1}\to f_{n-2,1}$, and $\fy_{n-1,1}\to g_{11}$; the edges pointing towards
$\fy_{n-1,1}$ are $\fy_{1,n-2}\to \fy_{n-1,1}$ and $g_{22}\to \fy_{n-1,1}$.

Vertex $\fy_{1l}$, $l\in [2,n-2]$, has degree~6. The edges pointing from $\fy_{1l}$ are $\fy_{1l}\to\fy_{2l}$,
$\fy_{1l}\to \fy_{1,l-1}$, and $\fy_{1l}\to \fy_{l+1,1}$; the edges pointing towards $\fy_{1l}$ are $\fy_{1,l+1}\to\fy_{1l}$,
$\fy_{2,l-1}\to\fy_{1l}$, and $\fy_{l1}\to\fy_{1l}$. Vertex $\fy_{1,n-1}$ has degree 5. The edges pointing from $\fy_{1,n-1}$
are $\fy_{1,n-1}\to \fy_{1,n-2}$ and $\fy_{1,n-1}\to h_{22}$; the edges pointing towards
$\fy_{1,n-1}$ are $\fy_{n-1,1}\to \fy_{1,n-1}$, $f_{1,n-2}\to \fy_{1,n-1}$, and $h_{11}\to \fy_{1,n-1}$.

Finally, $\fy_{11}$ is the special vertex. It has degree~4, and the corresponding edges are $\fy_{12}\to\fy_{11}$, $g_{11}\to\fy_{11}$ 
and $\fy_{11}\to\fy_{21}$, $\fy_{11}\to h_{11}$.

Vertex $f_{kl}$, $k+l<n$, has degree~6. The edges pointing from $f_{kl}$ are $f_{kl}\to f_{k+1,l-1}$,
$f_{kl}\to f_{k,l+1}$, and $f_{kl}\to f_{k-1,l}$; the edges pointing towards $f_{kl}$ are $f_{k-1,l+1}\to f_{kl}$,
$f_{k,l-1}\to f_{kl}$, and $f_{k+1,l}\to f_{kl}$. To justify this description in the extreme cases (such that $k+l=n-1$, $k=1$,
and $l=1$), we use the identification of Remark~\ref{identify}.

\begin{figure}[ht]
\begin{center}
\includegraphics[width=12cm]{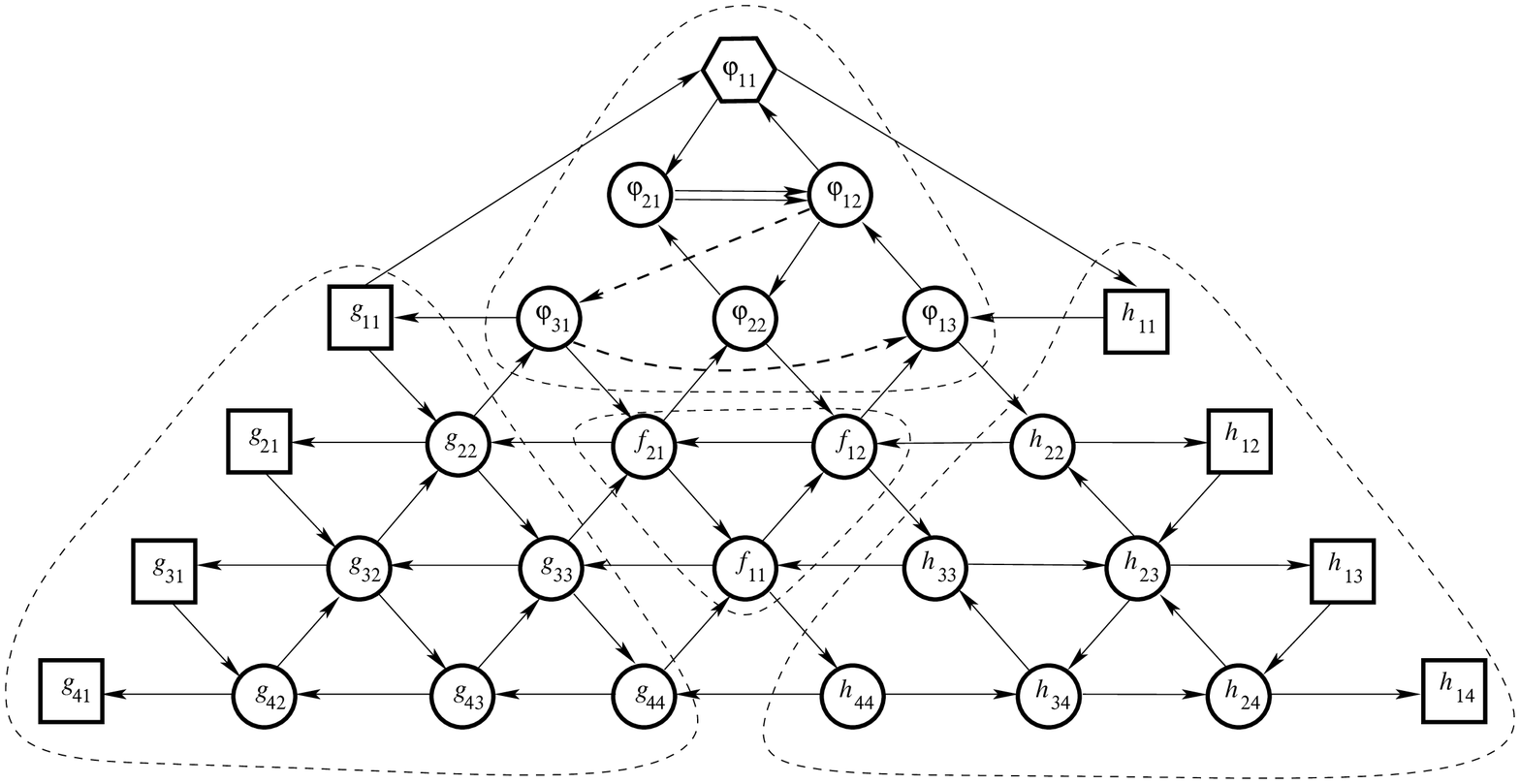}
\caption{Quiver $Q_4$}
\label{D4}
\end{center}
\end{figure}

Vertex $g_{ij}$, $i\ne n$, $j\ne 1$, has degree~6. The edges pointing from $g_{ij}$ are $g_{ij}\to g_{i+1,j+1}$,
$g_{ij}\to g_{i,j-1}$, and $g_{ij}\to g_{i-1,j}$; the edges pointing towards $g_{ij}$ are $g_{i,j+1}\to g_{ij}$,
$g_{i-1,j-1}\to g_{ij}$, and $g_{i+1,j}\to g_{ij}$ (for $i=j$ we use the identification of Remark~\ref{identify}).
Vertex $g_{nj}$, $j\ne 1$, has degree~4. The edges pointing from $g_{nj}$ 
are $g_{nj}\to g_{n-1,j}$ and $g_{nj}\to g_{n,j-1}$; the edges pointing towards $g_{nj}$ are 
$g_{n-1,j-1}\to g_{nj}$ and $g_{n,j+1}\to g_{nj}$  (for $j=n$ we use the identification of Remark~\ref{identify}).
Vertex $g_{i1}$, $i\ne 1$, $i\ne n$, has degree~2, and the corresponding edges 
are $g_{i1}\to g_{i+1,2}$ and $g_{i2}\to g_{i1}$.  
Vertex $g_{11}$ has degree~3, and the corresponding edges are $\fy_{n-1,1}\to g_{11}$ and $g_{11}\to g_{21}$, $g_{11}\to\fy_{11}$.
Finally, $g_{n1}$ has degree~1, and the only edge is $g_{n2}\to g_{n1}$.

Vertex $h_{ij}$, $i\ne 1$, $j\ne n$, $i\ne j$, has degree~6. The edges pointing from $h_{ij}$ are 
$h_{ij} \to h_{i,j-1}$, $h_{ij} \to h_{i-1,j}$, and $h_{ij} \to h_{i+1,j+1}$;
 the edges pointing towards $h_{ij}$ are $h_{i+1,j}\to h_{ij}$,
$h_{i,j+1}\to h_{ij}$, and $h_{i-1,j-1}\to h_{ij} $. 
Vertex $h_{ii}$, $i\ne 1$, $i\ne n$, has degree~4. The edges pointing from $h_{ii}$ are $h_{ii}\to f_{1,n-i}$ and $h_{ii}\to h_{i-1,i}$; the edges pointing to $h_{ii}$ are $f_{1,n-i+1}\to h_{ii}$ and $h_{i,i+1}\to h_{ii}$.
Vertex $h_{in}$, $i\ne 1$, $i\ne n$, has degree~4. The edges pointing from $h_{in}$ 
are $h_{in}\to h_{i,n-1}$ and $h_{in}\to h_{i-1,n}$; the edges pointing towards $h_{in}$ are 
$h_{i+1,n}\to h_{in}$ and $h_{i-1,n-1}\to h_{in}$. 
Vertex $h_{1j}$, $j\ne 1$, $j\ne n$, has degree~2, and the corresponding edges 
are $h_{1j}\to h_{2,j+1}$ and $h_{2j}\to h_{1j}$. 
Vertex $h_{nn}$ has degree~3. The edges pointing from $h_{nn}$ are $h_{nn}\to g_{nn}$ and $h_{nn}\to h_{n-1,n}$; the edge pointing
to $h_{nn}$ is $f_{11}\to h_{nn}$. Vertex $h_{11}$ has degree~2, and the corresponding edges are $h_{11}\to\fy_{1,n-1}$ and $\fy_{11}\to h_{11}$.
Finally, $h_{1n}$ has degree~1, and the only edge is $h_{2n}\to h_{1n}$.

The quiver $Q_4$ is shown in Fig.~\ref{D4}. The frozen vertices are shown as squares, the special vertex is shown as a hexagon, isolated vertices are not shown. 
 Certain arrows are dashed; this does not have any particular meaning, and just makes the picture more comprehensible.
One can identify in $Q_n$ four 
``triangular regions" associated with four families
$\{ g_{ij}\}$, $\{h_{ij}\}$, $\{f_{kl}\}$, $\{\fy_{kl}\}$. We will call vertices in these regions $g$-, $h$-, $f$- and $\fy$-vertices, respectively. It is easy to see that $Q_4$, as well as $Q_n$ for any $n$, can be embedded into a torus. 

The case $n=2$ is special. In this case there are only three types of vertices: $g$, $h$, and $\fy$. The quiver $Q_2$ is shown in Fig.~\ref{D2}.

\begin{figure}[ht]
\begin{center}
\includegraphics[width=6cm]{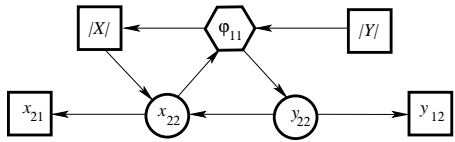}
\caption{Quiver $Q_2$}
\label{D2}
\end{center}
\end{figure}

\begin{remark}
\label{diagonal} On the diagonal subgroup $\{ (X,X) : X\in GL_n\}$ of $D(GL_n)$, $g_{ii} = h_{ii}$ for $1\le i\le n$, and functions $f_{kl}$ and $\fy_{kl}$ vanish identically.
Accordingly, vertices in $Q_n$ that correspond to $f_{kl}$ and $\fy_{kl}$ are erased and, for $1\le i\le n$, vertices corresponding to $g_{ii}$ and  $h_{ii}$ are identified.
As a result, one recovers a seed of the cluster structure compatible with the standard Poisson-Lie structure on $GL_n$,  see~\cite[Chap.~4.3]{GSVb}.
\end{remark}
 
\begin{remark}
\label{tribfz} 
At this point, we should emphasize a connection between the data $(F_n, Q_n)$ and particular seeds that give rise to the standard cluster structures on double Bruhat 
cells $\G^{e,w_0}$, $\G^{w_0,e}$ for $\G=GL_n$ and $w_0$ the longest element in the corresponding Weyl group. 
We will frequently explore this connection in what follows. Consider, in particular, the subquiver $Q_n^h$ of $Q_n$ associated with functions $h_{ij}$ in which, 
in addition to vertices $h_{1i}$, we also
view vertices $h_{ii}$ as frozen. Restricted to upper triangular matrices, the family $\{h_{ij}\}$ together with the quiver $Q_n^h$ defines an initial seed for the 
cluster structure on $\G^{e,w_0}$. This can be seen, for example, by applying the construction of Section~2.4 in \cite{CAIII} using  a reduced word 
$1 2 1 3 2 1\ldots (n-1) (n-2) \ldots 2 1$ for $w_0$. This leads to the cluster formed by
all dense minors that involve the first row. The seed we are interested in is then obtained via the transformation $B \mapsto W_0 B^T W_0$ 
applied to upper triangular matrices. Similarly,
the family $\{g_{ij}\}$ restricted to lower-triangular matrices together with the quiver $Q_n^g$ obtained from $Q_n$ in the same way as $Q_n^h$ (and isomorphic to it) defines an initial seed for $\G^{w_0,e}$. 

As explained in Remark~2.20 in \cite{CAIII}, in the case of  the standard cluster structures on $\G^{e,w_0}$ or $\G^{w_0,e}$, the cluster algebra and the upper cluster algebra coincide. This implies, in particular, that in every cluster for this cluster structure, each matrix entry of an upper/lower triangular matrix is expressed as a Laurent polynomial in cluster variables which is polynomial in stable variables. Furthermore, using similar considerations and the invariance under right multiplication by unipotent lower triangular matrices of column-dense minors that involve the first column, it is easy to conclude that each such minor has a Laurent polynomial expression  in terms of dense minors involving the first column and, moreover, each leading dense principal minor enters this expression polynomially. Both  these properties will be utilized below.

\end{remark}

\subsection{Exchange relations}\label{exrelsec} 
We define the set 
 $\P_n$ of strings for $Q_n$ that contains only one nontrivial string $\{p_{1r}\}$,
$0\le r\le n$. It corresponds to the vertex $\fy_{11}$ of multiplicity $n$, and $p_{1r}=c_r$,  $1\le r\le n-1$. 
The strings corresponding to all other vertices are trivial. Consequently, the generalized exchange relation at the vertex $\fy_{11}$ 
for $n>2$ is expected to look as follows:
\begin{equation*}
\fy_{11}\fy'_{11}=\sum_{r=0}^nc_r\fy_{21}^r\fy_{12}^{n-r}.
\end{equation*}
Indeed, such a relation exists in the ring of regular functions on $D(GL_n)$, and is given by the following proposition.

\begin{proposition}\label{polyrel}
For any $n>2$,
\begin{equation}\label{general}
 \det ((-1)^{n-1}\fy_{12}X+\fy_{21}Y)=\fy_{11}P^*_n,
\end{equation}
where $P^*_n$ is a polynomial in the entries of $X$ and $Y$.
\end{proposition}

For $n=2$, relation~\eqref{general} is replaced by
$$
\det(-y_{22}X+x_{22}Y)=\fy_{11}
\left|\begin{matrix}
y_{21} & y_{22}\cr
x_{21} & x_{22}
\end{matrix}\right|.
$$

 \subsection{Statement of main results}
\label{main}
Let $n\ge 2$.

\begin{theorem}
\label{structure}
{\rm (i)}  
The extended seed $\widetilde\Sigma_n=(F_n,Q_n,\P_n)$ defines 
a generalized cluster structure  $\GCC_n^D$
in the ring of regular functions on $D(GL_n)$ compatible with the standard Poisson--Lie structure on $D(GL_n)$. 

{\rm (ii)} The corresponding generalized upper cluster algebra 
 $\UU(\GCC_n^D)$ 
over 
\[
\widehat{\AA}=\C[g_{11}^{\pm1},g_{21},\dots, g_{n1},h_{11}^{\pm1}, h_{12},\dots,h_{1n}, c_1,\dots, c_{n-1}]
\] 
is naturally isomorphic to
the ring of regular functions on  $D(GL_n)$.
\end{theorem}

\begin{remark}\label{lowdim}
1. Since the only stable variables that do not vanish on $D(GL_n)$ are $g_{11}=\det X$ and $h_{11}=\det Y$, 
the ground ring in (ii) above is a particular case of~\eqref{hata}.
In fact, it follows from the proof that a stronger statement holds:

(i) $\GCC_n^D$ extends to a regular generalized cluster structure on $\Mat_n\times \Mat_n$; 

(ii) the generalized upper cluster algebra over 
\[
\widehat{\AA}=\C[g_{11}^{\pm1},g_{21},\dots, g_{n1},h_{11}, h_{12},\dots,h_{1n}, c_1,\dots, c_{n-1}]
\]
 is naturally isomorphic to
the ring of regular function on $GL_n\times \Mat_n$.

2. For $n=2$ the obtained generalized cluster structure has a finite type.
Indeed, the principal part of the exchange matrix for the cluster shown
in Fig.~\ref{D2} has a form
\[
\left(\begin{array}{rrr}
0 & 2  & -2 \\
-1 & 0 & 1 \\
1 & -1 & 0\\
\end{array}\right).
\]
The mutation of this matrix in direction $2$ transforms it into
\[
\left(\begin{array}{rrr}
0 & -2  & 0 \\
1 & 0 &-1 \\
0 & 1 & 0\\
\end{array}\right), 
\]
and its Cartan companion is a Cartan matrix of type $B_3$. Therefore, by~\cite[Thm.~2.7]{CheSha},  the generalized cluster structure has type $B_3$. This
implies, in particular,  that its exchange graph is the 1-skeleton of
the 3-dimensional cyclohedron (also known as the Bott--Taubes polytope), and its cluster complex is the 3-dimensional
polytope polar dual to the cyclohedron (see~\cite[Sec.~5.2]{FoRe} for further details).

3. It follows immediately from Theorem~\ref{structure}(i) that the extended seed obtained from $\widetilde\Sigma_n$ by deleting functions $\det X$
and $\det Y$ from $F_n$, deleting the corresponding vertices from $Q_n$
and restricting relation~\eqref{general} to $\det X=\det Y=1$ defines a generalized cluster structure in the ring of regular functions on $D(SL_n)$
compatible with the standard Poisson--Lie structure on $D(SL_n)$. Moreover,
by Theorem~\ref{structure}(ii), 
the corresponding generalized upper cluster algebra is naturally isomorphic to the ring of regular functions on $D(SL_n)$.
\end{remark}

Using  Theorem \ref{structure}, we can construct a  generalized cluster 
structure 
on $GL_n^\dag$. For $U\in GL_n^\dag$, denote
$\psi_{kl}(U) = s_{kl}\det \Phi_{kl}(U)$, where $s_{kl}$ are the signs defined in Section \ref{logcan}.
The initial extended cluster $F^\dag_n$  for $GL_n^\dag$ consists of
functions  $\psi_{kl}(U)$, $k,l\ge 1$, $k+l\le n$, $(-1)^{(n-i)(j-i)}h_{ij}(U)$, $2\le i\le j\le n$, $h_{11}(U)=\det U$, and $c_i(\one, U)$, 
$1\le i\le n-1$.
To obtain the initial seed for $GL_n^\dag$, we apply a certain
sequence $\TE$ of cluster transformations to the initial seed
for $D(GL_n)$. This sequence does not involve vertices associated with
functions $\psi_{kl}$. The resulting cluster $\TE(F_n)$ contains a
subset $\{ \left (\det X\right )^{\nu(f)}f  : f \in F^\dag_n\}$ with 
$\nu(\psi_{kl})=n-k-l+1$ and $\nu(h_{ij})=1$.  
These functions are attached to a subquiver $Q^\dag_n$ in the resulting quiver $\TE(Q_n)$, which 
is isomorphic to the subquiver of $Q_n$ formed by vertices
associated with functions $\fy_{kl}, f_{ij}$ and $h_{ii}$, see Fig.~\ref{D*4}.
Functions $h_{ii}(U)$ are declared stable variables, $c_i(\one,U)$ remain isolated.  
See Theorem~\ref{clusterU} below for more details.

\begin{figure}[ht]
\begin{center}
\includegraphics[width=6cm]{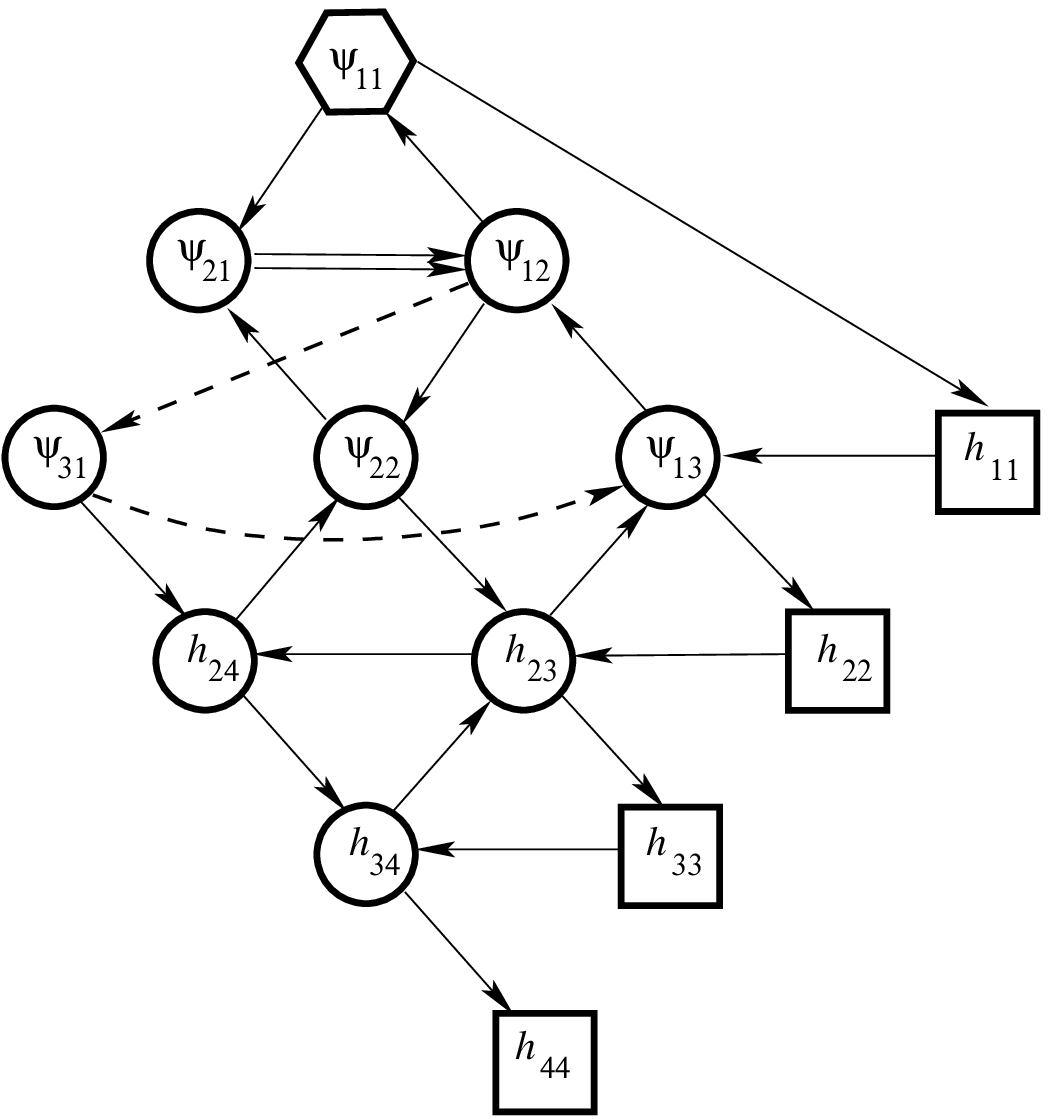}
\caption{Quiver $Q^\dag_4$}
\label{D*4}
\end{center}
\end{figure}

All exchange relations  defined by mutable
vertices of $Q^\dag_n$ are homogeneous in $\det X$. This allows us to use $(F^\dag_n, Q_n^\dag,\P_n)$ as an initial seed for $GL_n^\dag$.
The generalized exchange relation associated with the cluster variable
$\psi_{11}$ now takes the form
$\det ((-1)^{n-1}\psi_{12}\one+\psi_{21}U)=\psi_{11}\Pi^*_n$,
where $\Pi^*_n$ is a polynomial in the entries of $U$.

\begin{theorem}
\label{dualstructure}
{\rm (i)} The extended seed $(F_n^\dag,Q_n^\dag,\P_n)$ defines a generalized cluster structure 
$\GCC_n^{\dag}$
in the ring of regular functions on $GL_n^\dag$ compatible with $\Poi_*$.
 
{\rm (ii)} The corresponding generalized  upper cluster algebra over 
\[
\widehat{\AA}=\C[h_{11}(U)^{\pm1}, \dots, h_{nn}(U)^{\pm1}, c_1(\one,U),\dots, c_{n-1}(\one,U)]
\]
is naturally isomorphic to the ring of regular functions on $GL_n^\dag$.
\end{theorem}

\begin{remark}
1. It follows from Remark~\ref{ultinohii} that a stronger statement holds, similarly to Remark~\ref{lowdim}.1:
 
 (i) $\GCC_n^\dag$ extends to a regular generalized cluster structure on $\Mat_n$;

(ii)  the generalized upper cluster algebra over 
\[
\widehat{\AA}=\C[h_{11}(U),\dots, h_{nn}(U), c_1(\one,U),\dots, c_{n-1}(\one,U) ]
\]
 is naturally isomorphic to
the ring of regular function on $\Mat_n$.

2. Let $\mathcal V_n$ be the intersection of $SL_n^\dag$ with a generic conjugation orbit in $SL_n$. This variety plays a role in a rigorous mathematical
description of Coulomb branches in  4D gauge theories. The generalized cluster structure $\GCC_n^{\dag}$ descends 
 to $\mathcal V_n$ if one fixes the values of $c_1(\one,U),\dots,c_{n-1}(\one,U)$. 
The existence of a cluster structure on $\mathcal V_n$ was suggested by D.~Gaiotto (A.~Braverman, private communication).
\end{remark}

\subsection{The outline of the proof}\label{outline} 
We start with defining a local toric action of right and left multiplication by diagonal 
matrices, and use Proposition~\ref{globact} to check that this action can be extended to a global one. This fact is then used in the proof of the compatibility
assertion in Theorem~\ref{structure}(i), which is based on Proposition~\ref{compatchar}. As a byproduct, we get that the extended exchange matrix
of $\GCC_n^D$ is of full rank.

Next, we have to check conditions (i)--(iii) of Proposition~\ref{regcoin}. The regularity condition in (i) follows from Theorem~\ref{basis} and the explicit
description of the basis. The coprimality condition in (i) is a corollary of the following stronger statement.

\begin{theorem}\label{irredinbas}
All functions in $F_n$ are irreducible polynomials in matrix entries.
\end{theorem}

We then establish the regularity and coprimality conditions in (ii), which completes the proof of Theorem~\ref{structure}(i).

To prove Theorem~\ref{structure}(ii), it is left  to check condition~(iii) of Proposition~\ref{regcoin}. 
The usual way to do that consists in applying Theorem~3.21 from~\cite{GSVb} which claims that for cluster structures of
geometric type with an exchange matrix of full rank, the upper cluster algebra coincides with the upper bound at any cluster. It remains to choose an
appropriate set of generators in $\O(V)$ and to check that each element of this set can be represented as a Laurent polynomial in some fixed cluster and in all its neighbors.
We will have to extend the above result in three directions:

1) to upper cluster algebras over $\widehat\AA$, 
as opposed to upper cluster algebras over $\bar\AA$;

2) to more general neighborhoods of a vertex in $\T_N$, as opposed to the stars of vertices;

3) to generalized cluster structures of geometric type, as opposed to ordinary cluster structures.

Let $\GCC=\GCC(\wB,\P)$ be a generalized cluster structure as defined in Section~\ref{SecPrel}, and let $L$ be the number of isolated variables in $\GCC$. For the $i$th
nontrivial string in $\P$, define a $(d_i-1)\times L$ integer matrix $\wB(i)$: the $r$th row of $\wB(i)$ contains the exponents of isolated variables in the exchange coefficient 
$p_{ir}$ (recall that $p_{ir}$ are monomials). 

Following~\cite{Fra}, we call a {\it nerve\/} an arbitrary subtree of $\T_N$ on $N+1$ vertices such that all its edges have different labels. A star
of a vertex in $\T_N$ is an example of a nerve. Given a nerve $\N$, we define an {\it upper bound\/} $\UU(\N)$ as the intersection of the rings of Laurent polynomials
$\widehat\AA[\x^{\pm 1}]$ taken over all seeds in $\N$. We prove the following theorem that seems to be interesting in its own right.

\begin{theorem}\label{nervub}
Let $\wB$ be a skew-symmetrizable matrix of full rank, and let 
\[
\rank \wB(i)=d_i-1 
\]
for any nontrivial string in $\P$.
Then the upper bounds $\UU(\N)$ do not depend on the choice of $\N$, and hence coincide with the
generalized upper cluster algebra $\UU(\wB,\P)$ over $\widehat\AA$.
\end{theorem}

We then proceed as follows. First, we choose the  $2n^2$ matrix entries of $X$ and $Y$ as the generating set of the ring of regular functions on $D(GL_n)$. Then we
prove the following result.

\begin{theorem}\label{allxcluster}
Each matrix entry of $X$ is either a stable variable or a cluster variable in $\GCC_n^D$.
\end{theorem}

To treat the remaining part of the generating set we consider a special nerve $\N_0$ in the tree $\T_{(n-1)(2n-1)}$. First of all, we design a sequence $\TE$ of 
cluster transformations that takes the initial extended seed $\widetilde\Sigma_n$ to a new extended seed $\widetilde\Sigma'_n=\TE(\widetilde\Sigma_n)=
(\TE(F_n),\TE(Q_n),\TE(\P_n))$ having the following properties. 
 Let $Q_n^\dag$ and $F_n^\dag$ be as defined in Section~\ref{main}, and $U=X^{-1}Y$.

\begin{theorem}\label{clusterU}
There exists a sequence $\TE$ of cluster transformations such that 

{\rm (i)} $\TE(\P_n)=\P_n$; 

{\rm (ii)} $\TE(Q_n)$ contains a subquiver $Q_n'$ isomorphic to $Q_n^\dag$;  

{\rm (iii)} the functions in $\TE(F_n)$ assigned to the vertices of $Q_n'$ constitute the set  
$\left\{\left(\det X^{n-k-l+1}\psi_{kl}(U)\right)_{k,l\ge 1, k+l\le n}, \left(\det X\cdot h_{ij}(U)\right)_{2\le i\le j\le n}, \det X\cdot h_{11}(U)\right\}$;

{\rm (iv)} the only vertices in $Q'_n$ connected with the rest of vertices in $\TE(Q_n)$ are those associated with 
$\det X\cdot h_{ii}$, $2\le i\le n$,  and $\fy_{11}$.
\end{theorem}

 As an immediate corollary we get Theorem~\ref{dualstructure}(i).

The nerve $\N_0$ contains the seed $\widetilde\Sigma'_n$, a seed
$\widetilde\Sigma''_n$ adjacent to $\widetilde\Sigma'_n$, and a seed $\widetilde\Sigma'''_n$ adjacent to $\widetilde\Sigma''_n$. 
Besides, it contains $2(n-1)^2$ seeds adjacent to $\widetilde\Sigma'_n$ and distinct from $\widetilde\Sigma''_n$,
and $n-3$ seeds adjacent to $\widetilde\Sigma'''_n$ and distinct from $\widetilde\Sigma''_n$. A more detailed description of $\N_0$ is given in 
Section~\ref{thenerve} below. We then prove

\begin{theorem}\label{Urestore}
Each matrix entry of $U=X^{-1}Y$ multiplied by an appropriate power of $\det X$ belongs to the upper bound $\UU(\N_0)$.
\end{theorem}

Consequently each matrix entry of $Y=XU$ belongs to $\UU(\N_0)$. 

It remains to note that $\wB(1)$ is the $(n-1)\times (n-1)$ identity matrix.
Therefore, all conditions in Theorem~\ref{nervub} are satisfied, and we get the proofs of Theorems~\ref{structure}(ii) 
 and~\ref{dualstructure}(ii).

\section{Generalized upper cluster algebras of geometric type over $\hat\AA$}\label{novoeslovo}

Let $\GCC=\GCC(\wB,\P)$ be a generalized cluster structure as defined in Section~\ref{SecPrel}, and let $\AA\subseteq\widehat\AA\subseteq\bar\AA$ be the corresponding rings.
The goal of this section is to prove Theorem~\ref{nervub}. We start with the following statement, which is an extension of the standard result on
the coincidence of upper bounds (see e.g. \cite[Corollary~3.22]{GSVb}).

\begin{theorem} \label{nervubcop}
If the generalized exchange polynomials are coprime in $\AA[x_1,\dots,x_{N}]$ for any seed in 
$\GCC$, then the upper 
bounds $\UU(\N)$ do not depend on the choice of the nerve $\N$, and hence coincide with the upper cluster algebra $\UU(\wB,\P)$ over $\widehat\AA$.
\end{theorem}

\begin{proof} Let us consider first the case $N=1$. In this case everything is exactly the same as in the standard situation. Namely, we consider two adjacent
clusters $\x=\{x_1\}$ and $\x_1=\{x_1'\}$ and the exchange relation $x_1x_1'=P_1$, where $P_1\in \AA$. The same reasoning as in 
Lemma~3.15 from \cite{GSVb} yields
\[
\widehat\AA[x_1^{\pm1}]\cap\widehat\AA[(x_1')^{\pm1}]=\widehat\AA[x_1,x_1'].
\]
As a corollary, for general $N$ one gets
\begin{equation}\label{twolaur}
\widehat\AA[x_1^{\pm1},x_2^{\pm1},\dots,x_N^{\pm1}]\cap\widehat\AA[(x_1')^{\pm1},x_2^{\pm1},\dots,x_N^{\pm1}]=\widehat\AA[x_1,x_1',x_2^{\pm1},\dots,x_N^{\pm1}].
\end{equation}
The latter relation is obtained from the one for $N=1$ via replacing $\AA$ with $\AA[x_2,\dots,x_N]$ and 
the ground ring $\widehat\AA$ with $\widehat\AA[x_2^{\pm1},\dots,x_N^{\pm1}]$.

Let now $N=2$. Note that $\T_2$ is an infinite path, and hence all nerves are just two-pointed stars. Let $\x=\{x_1,x_2\}$ be an arbitrary cluster, 
$\x_1=\{x_1',x_2\}$ and $\x_2=\{x_1,x_2'\}$ be the two adjacent clusters obtained via generalized exchange relations $x_1x_1'=P_1$ and $x_2x_2'=P_2$
with $P_1\in\AA[x_2]$ and $P_2\in\AA[x_1]$. Besides, let $\x_3=\{x_1',x_2''\}$ be the cluster obtained from $\x_1$ via the generalized
exchange relation $x_2x_2''=\bar P_2$ with $\bar P_2\in\AA[x_1']$. Let $\N$ be the nerve 
$\x_1$\textemdash$\x$\textemdash$\x_2$, and $\N_1$ be the nerve 
consisting of the clusters 
$\x_1$\textemdash$\x$\textemdash$\x_3$. The following statement is an analog of Lemma~3.19 in \cite{GSVb}.

\begin{lemma}\label{2pstar}
Assume that $P_1$ and $P_2$ are coprime in $\AA[x_1,x_2]$ and $P_1$ and $\bar P_2$ are coprime in $\AA[x_1',x_2]$. Then $\UU(\N)=\UU(\N_1)$.
\end{lemma}

\begin{proof} The proof differs substantially from the proof of Lemma 3.19, since we are not allowed to invert monomials in $\widehat\AA$.

It is enough to prove the inclusion $\UU(\N)\subseteq\UU(\N_1)$, since the opposite inclusion is obtained by switching roles between $\x$ and $\x_1$.
By~\eqref{twolaur}, we have
\[
\UU(\N)=\widehat\AA[x_1,x_1',x_2^{\pm1}]\cap\widehat\AA[x_1^{\pm1},(x_2')^{\pm1}],\quad
\UU(\N_1)=\widehat\AA[x_1,x_1',x_2^{\pm1}]\cap\widehat\AA[(x_1')^{\pm1},(x_2'')^{\pm1}].
\]

Let $y\in \widehat\AA[x_1,x_1',x_2^{\pm1}]$; expand $y$ as a Laurent polynomial in $x_2$. Each term of this expansion containing a non-negative power of $x_2$ belongs
to $\widehat\AA[(x_1')^{\pm1},(x_2'')^{\pm1}]$, so we have to consider only $y$ of the form
\[
y=\sum_{k=1}^a \frac{Q'_k+Q_k}{x_2^k}, \qquad a\ge 1,
\]
with $Q_k\in\widehat\AA[x_1]$, $Q'_k\in\widehat\AA[x_1']$.

We can treat $y$ as above in two different ways. On the one hand, by substituting $x_1'=P_1/x_1$, 
it can be considered as an element in $\widehat\AA[x_1^{\pm1},x_2^{\pm1}]$ and written as
\begin{equation}\label{firsty}
y=\sum_{k\le a}\frac{R_k}{x_1^{\delta_k}x_2^k}
\end{equation}
with $R_k\in\widehat\AA[x_1]$ and $\delta_k\ge 0$. Imposing the condition $y\in\widehat\AA[x_1^{\pm1},(x_2')^{\pm1}]$,
we get $R_k=P_2^kS_k$  for $k> 0$ and some $S_k\in\widehat\AA[x_1]$. Note that each summand in~\eqref{firsty} with $k\le 0$
belongs to $\widehat\AA[x_1^{\pm1},(x_2')^{\pm1}]$ automatically.

On the other hand, by substituting $x_1=P_1/x_1'$, 
$y$ can be considered as an element in $\widehat\AA[(x_1')^{\pm1},x_2^{\pm1}]$ and written as
\begin{equation}\label{secy}
y=\sum_{k\le a}\frac{R_k'}{(x_1')^{\delta_k'}x_2^k}
\end{equation}
with $R_k'\in\widehat\AA[x_1']$ and $\delta_k'\ge 0$. Note that $R_k'$ can be restored via $R_l$ and $\delta_l$, $k\le l\le a$.
We will prove that $\bar P_2^k$ divides $R_k'$ in $\widehat\AA[x_1']$ for any $k>0$. This would mean that each summand in~\eqref{secy}
belongs to $\widehat\AA[(x_1')^{\pm1},(x_2'')^{\pm1}]$, and hence $\UU(\N)\subseteq\UU(\N_1)$ as claimed above.

Assume first that $\hat b_{12}=\hat b_{21}=0$ in the modified exchange matrix $\hat B$, which means that $P_1, P_2\in\AA$ and $\bar P_2=P_2$. 
Rewrite an arbitrary term $T_k=R_k/(x_1^{\delta_k}x_2^k)$, 
$k> 0$, in~\eqref{firsty} as an element in $\widehat\AA[(x_1')^{\pm1},x_2^{\pm1}]$ via substituting $x_1=P_1/x_1'$. Recall that $R_k$ is divisible by $P_2^k$, 
hence  
\[
T_k=\frac{(x_1')^{\delta_k}P_2^k S_k\rvert_{x_1\leftarrow P_1/x_1'}}{P_1^{\delta_k}x_2^k}=\frac{\bar P_2^k \bar S_k}{(x_1')^{\gamma_k}{P_1^{\delta_k}x_2^k}}
\]
for some $\gamma_k\ge 0$ and $\bar S_k\in \widehat\AA[x_1']$. Comparing the latter expression with~\eqref{secy}, we see that $\gamma_k=\delta_k'$ and
$R_k'=\bar P_2^k \bar S_k/P_1^{\delta_k}$. Since $\bar P_2$ and $P_1$ are coprime, this means  that $\bar P_2^k$ divides 
$R_k'$ in $\widehat\AA[x_1']$.

Assume now that $\hat b_{12}=b> 0$ (otherwise $\hat b_{21}>0$, and we proceed in the same way with $P_2$ instead of $P_1$). 
Then one can rewrite $P_1$ as $P_1=P_{10}+x_2^bP_{11}$ with $P_{10}$ is a monomial in $\AA$ and 
$P_{11}\in\AA[x_2]$ is not divisible by $x_2$.
Consider an arbitrary term $T_k=R_k/(x_1^{\delta_k}x_2^k)$, $k> 0$, in~\eqref{firsty} as an element in $\widehat\AA[(x_1')^{\pm1}]((x_2))$ via substituting 
$x_1=(P_{10}+x_2^bP_{11})/x_1'$ and expanding the result in the Taylor series in $x_2$. Similarly to the previous case, we get
\[
T_k=\sum_{j=0}^k\frac{P_2^{k-j}\rvert_{x_1\leftarrow P_{10}/x_1'}\hat S_j}{j!P_{10}^{\delta_k+j}x_2^{k-j}}+
\sum_{j>k}\frac{\hat S_jx_2^{j-k}}{j!P_{10}^{\delta_k+j}}
\]
for some $\hat S_j\in\widehat\AA[x_1']$. 
Since $y\in\hat\A[(x_1')^{\pm 1},x_2^{\pm 1}]$, we conclude
that the infinite sums above contribute only finitely many terms to
$y=\sum_{k\le a} T_k$. By~\eqref{secy}, any term of these finitely many automatically
belongs to $\hat\A[(x'_1)^{\pm 1},(x''_2)^{\pm 1}]$. To treat the finite
sum in $T_k$ we note that 
\[
\frac{(x_1')^{\deg_{x_1}P_2}P_2\rvert_{x_1\leftarrow P_{10}/x_1'}}{\bar P_2}
\]
is a monomial in $\AA$. So, the finite sum can be rewritten as
\[
\sum_{j=0}^k\frac{\bar P_2^{k-j}\bar S_j}{j!(x_1')^{\gamma_j}P_{10}^{\delta_k+j}x_2^{k-j}}
\]
for some $\gamma_j\ge 0$ and $\bar S_j\in \widehat\AA[x_1']$. Comparing the latter expression with~\eqref{secy}, we get
\[
\frac{R_k'}{(x_1')^{\delta'_k}}=\bar P_2^k\sum_{j=0}^{a-k}\frac{\bar S_j}{j!(x_1')^{\gamma_{j+k}}P_{10}^{\delta_{j+k}+j}}.
\]
Note that $\bar P_2$ and $P_{10}$ are coprime, since  $P_{10}$ is a monomial and $\bar P_2$ does not have monomial factors,
which means that $\bar P_2^k$ divides $R_k'$ in $\widehat\AA[x_1']$.
\end{proof}

In the case of  an arbitrary $N$ one can use Lemma~\ref{2pstar} to reshape nerves while preserving  the upper bounds. 
Namely, let $\N$ be a nerve, and $v_1,v_2,v_3\in\N$ be three vertices such that $v_1$ is adjacent to $v_2$ and $v_2$ is the unique vertex adjacent 
to $v_3$. Consider the nerve $\N'$ that does not contain $v_3$,  contains a new
vertex $v_3'$ adjacent only to $v_1$, and otherwise is identical to $\N$; the edge between $v_1$ and $v_3'$ in $\N'$ bears the same label as the edge
between $v_2$ and $v_3$ in $\N$. A single application of Lemma~\ref{2pstar} with $\widehat\AA$ replaced by  $\widehat\AA[x_3^{\pm1},\dots,x_N^{\pm1}]$
shows that $\UU(\N)=\UU(\N')$. Clearly, any two nerves can be connected via a sequence of such transformations, and the result follows.
\end{proof}
 
To complete the proof of Theorem~\ref{nervub}, it remains to establish the following result.

\begin{lemma}\label{rankcond}
Let $\wB$ be a skew-symmetrizable matrix of full rank, and let $\rank \wB(i)=d_i-1$ for any nontrivial string in $\P$.
Then the generalized exchange polynomials are coprime in $\AA[x_1,\dots,x_{N}]$ for any seed in $\GCC$.
\end{lemma}

\begin{proof} We follow the proof of Lemma~3.24  from \cite{GSVb} with minor modifications. Fix an arbitrary seed $\Sigma=(\wx,\wB,\P)$, and let 
$P_i$ be the generalized exchange polynomial corresponding to the $i$th cluster variable. 

Assume first that there exist $j$ and $j'$ such that $b_{ij}>0$ and $b_{ij'}<0$. We want to define the weights of the variables that make $P_i$ into 
a quasihomogeneous polynomial. Put $w(x_j)=1/b_{ij}$, $w(x_{j'})=-1/b_{ij'}$. 
If $j,j'\le N$, put $w(x_k)=0$ for $k\ne j,j'$. Otherwise, put $w(x_k)=0$ for all remaining cluster variables and all remaining stable non-isolated
variables. Finally, define the weights of isolated variables from the equations $w(\hat p_{ir})=0$, $1\le r\le d_i-1$. The condition on the rank of
$\wB(i)$ guarantees that these equations possess a unique solution. Now~\eqref{modexchange} shows that this weight assignment turns $P_i$ into a quasihomogeneous
polynomial of weight one.

Let $P_i=P'P''$ for some nontrivial polynomials $P'$ and $P''$, then they both are quasihomogeneous with respect to the weights defined above, and
each one of them contains exactly one monomial
in variables entering $u_{i;>}, v_{i;>}$, and exactly one monomial in variables entering $u_{i;<}, v_{i;<}$. Consider these two monomials in $P'$.
Let $\delta_j$ and $\delta_{j'}$ be the degrees of $x_j$ and $x_{j'}$ in these two monomials, respectively. Then the quasihomogeneity condition implies
$\delta_j/b_{ij}=-\delta_{j'}/b_{ij'}$. Moreover, for any $j''\ne j, j'$ such that $b_{ij''}>0$ (or $b_{ij''}<0$) a similar procedure gives
$\delta_{j''}/b_{ij''}=-\delta_{j'}/b_{ij'}$ (or $\delta_j/b_{ij}=-\delta_{j''}/b_{ij''}$, respectively. This means that the $i$th row of $\wB$ can be restored
from the exponents of variables entering the above two monomials by dividing them by a constant. Consequently, if $P_i$ and $P_j$ possess a nontrivial
common factor, the corresponding rows of $\wB$ are proportional, which contradicts the full rank assumption.

If all nonzero entries in the $i$th row have the same sign, we proceed in a similar way. Namely, if there exist $j,j'$ such that $b_{ij}, b_{ij'}\ne0$, we
put $w(x_j)=1/|b_{ij}|$, $w(x_{j'})=-1/|b_{ij'}|$. The weights of other variables are defined in the same way as above. This makes $P_i$ into a quasihomogeneous
polynomial of weight zero, and the result follows. The case when there exists a unique $j$ such that $b_{ij}\ne 0$ is trivial.
\end{proof}

\section{Proof of Theorem~\ref{basis}}\label{prlogcan}

The proof exploits various invariance properties of functions in $F_n$. First, we need some preliminary lemmas. 

Let a bilinear form $\langle\cdot , \cdot\rangle_0$ on $\mathfrak{gl}_n$ be defined 
as $\langle A ,B \rangle_0 = \langle \pi_0(A) , \pi_0(B) \rangle$.

\begin{lemma}
\label{cross-inv}
Let $g(X), h(Y), f(X,Y), \fy(X,Y)$ be functions with the following invariance properties: 
\begin{equation}\label{inv_prop}
\begin{aligned}
 g(X)&=g(N_+X), \quad   h(Y)=h(YN_-), \\  
f(X,Y) &= f(N_+X N_- ,N_+Y N'_-), \quad
\fy(X,Y)= \fy(A X N_-, A Y N_-),
\end{aligned}
\end{equation}
where $A$ is an arbitrary element of $GL_n$, $N_+$ is an arbitrary  unipotent upper-triangular element and $N_-, N'_-$ are arbitrary  unipotent lower-triangular elements. Then
\begin{align*}
&\{g, h\}_D = \frac{1}{2}  \langle X\nabla_X  g, Y\nabla_Y h\rangle_0 - \frac{1}{2} \langle\nabla_X  g\cdot X, \nabla_Y h \cdot Y\rangle_0, \\
&\{f , g\}_D =  \frac{1}{2}  \langle E_L f  , \nabla_X g\cdot X\rangle_0 - \frac{1}{2} \langle E_R f  , X\nabla_X g\rangle_0, \\
&\{h, f\}_D =  \frac{1}{2} \langle \nabla_Y h \cdot Y  , E_L f\rangle_0 - \frac{1}{2} \langle Y\nabla_Y h, E_R f \rangle_0, \\
&\{\fy, f\}_D =  \frac{1}{2} \langle E_L \fy  , \nabla_X f\cdot X\rangle_0 -  \frac12\langle E_L \fy, \nabla_Y f\cdot Y \rangle_0, \\
&\{\fy, g\}_D =  \frac{1}{2} \langle E_L \fy  , \nabla_X g\cdot X \rangle_0, \\
&\{\fy, h\}_D =  -\frac{1}{2} \langle E_L \fy  , \nabla_Y h\cdot Y \rangle_0, 
\end{align*}
where $E_L$ and $E_R$ are given by~\eqref{erel}.
\end{lemma}

\begin{proof} From \eqref{inv_prop}, we obtain 
$ X\nabla_X g, E_R f \in \b_+$, $ \nabla_Y h\cdot Y, \nabla_X f\cdot X, \nabla_Y f\cdot Y, E_L \fy  \in \b_-$,
$E_R \fy=0$. Taking into account that $R_+(\xi)=\frac{1}{2}\pi_0(\xi)$ for $\xi\in\b_-$ and that
$\b_\pm \perp \n_\pm$ with respect to $\langle\cdot ,\cdot \rangle$, the result follows from \eqref{sklyadoubleGL},\eqref{sklyadoubleGL1}.
\end{proof}

\begin{lemma}
\label{homogen}
Let $g(X), h(Y), f(X,Y), \fy(X,Y)$ be functions as in Lemma \ref{cross-inv}. Assume, in addition,
that $g$ and $h$ are homogeneous with respect to right and left multiplication of their arguments by arbitrary diagonal matrices and that $f$ and $\fy$ are homogeneous with respect to right and left multiplication of  $X, Y$ by the same pair of diagonal matrices.
Then all Poisson brackets $\Poi_D$ among functions $\log g$, $\log h$, $\log f$, $\log \fy$ are constant.
\end{lemma}

\begin{proof} The homogeneity of $g(X)$ with respect to the left multiplication by diagonal matrices 
implies that there exists a diagonal element $\xi$ such that for any diagonal $h$ and any $X$,
$g(\exp(h) X) = \exp \langle h, \xi\rangle g(X)$. The infinitesimal version of this property reads
$\pi_0(X  \nabla \log g(X)) = \xi$. A similar argument shows that diagonal projections of all 
elements needed to compute Poisson brackets between $\log g$, $\log h$, $\log f$, $\log \fy$  using
formulas of Lemma~\ref{cross-inv} are constant diagonal matrices, and the claim follows.
\end{proof}

Lemmas \ref{cross-inv}, \ref{homogen} show that any four functions $f_{ij}, g_{kl}, h_{\alpha\beta}, \fy_{\mu\nu}$ 
are log-canonical. Indeed, it is clear from definitions in Section~\ref{logcan} that
\begin{equation}\label{ourinv_prop}
\begin{aligned}
 g_{ij}(X)&=g_{ij}(N_+X), \quad g_{ii}(X)=g_{ii}(N_+XN_-), \\  
h_{ij}(Y)&=h_{ij}(YN_-), \quad  h_{ii}(Y)=h_{ii}(N_+YN_-),\\  
f_{kl}(X,Y) &= f_{kl}(N_+X N_- ,N_+Y N'_-),\\
\tilde \fy_{kl}(X,Y) &= \tilde\fy_{kl}(A X N_-, A Y N_-),
\end{aligned}
\end{equation}
where $\tilde \fy_{kl} = \det\Phi_{kl}$, and so the corresponding invariance 
properties in~\eqref{inv_prop} are satisfied for any function
taken in any of these four families. Besides, all these 
functions possess the homogeneity property as in Lemma~\ref{homogen} as well.

For a generic element $X \in GL_n$, consider its Gauss factorization 
\begin{equation}
\label{gauss}
X=X_{>0} X_0 X_{<0}
\end{equation}
with $X_{<0}$ unipotent lower-triangular, $X_0$ diagonal and $X_{>0}$ unipotent upper-trian\-gular elements. Sometimes it will be convenient to use notations $X_{\leq 0} = X_0 X_{<0}$ and $X_{\geq 0} = X_{>0} X_0$.
Taking
$N_+=(X_{>0})^{-1}$ in the first relation in~\eqref{ourinv_prop}, $N_-=(Y_{<0})^{-1}$ in the second relation, and
$N_+=(Y_{>0})^{-1}$, $N_-=(X_{<0})^{-1}$, $N'_-=(Y_{<0})^{-1}$ in the third relation, one gets
\begin{equation}\label{canform}
\begin{aligned}
 g_{ij}(X)&=g_{ij}(X_{\leq 0}),\\  
h_{ij}(Y)&=h_{ij}(Y_{\geq 0}), \\  
f_{kl}(X,Y) &=  f_{kl}\left ((Y_{>0})^{-1}X_{\geq 0},Y_0\right )\\
  &= h_{n-l+1,n-l+1}(Y) h_{n-k-l+1,n-k+1}\left ((Y_{>0})^{-1}X_{\geq 0}\right ).
\end{aligned}
\end{equation}

 Next, we need to prove log-canonicity within each of the four families.
The following lemma is motivated by the third formula in~\eqref{canform}.

\begin{lemma}
\label{Zmap}
The almost everywhere defined map 
\[
Z : (D(GL_n), \Poi_D)\to (GL_n, \Poi_r)
\]
 given by
$(X,Y) \mapsto Z= Z(X,Y)=(Y_{>0})^{-1}X_{\geq 0}$ is Poisson. 
%w.r.t. $\Poi_D$ and $\Poi_r$.
\end{lemma}

\begin{proof} Denote $\pi_{\geq0}=\pi_{>0}+\pi_0$ and $\pi_{\leq0}=\pi_{<0}+\pi_0$. We start by computing the variation 
\begin{equation*}
\begin{aligned}
\delta Z &= (Y_{>0})^{-1}\left (\delta X_{\geq 0} - \delta Y_{>0} (Y_{>0})^{-1}X_{\geq 0}\right )=
Z (X_{\geq 0})^{-1} \delta X_{\geq 0} - (Y_{>0})^{-1}  \delta Y_{>0} Z\\
&=Z \pi_{\geq 0}\left ( (X_{\geq 0})^{-1} \delta X (X_{<0})^{-1}\right ) -
\pi_{>0}\left ( (Y_{>0})^{-1} \delta Y (Y_{\leq 0})^{-1}\right ) Z.
\end{aligned}
\end{equation*} 
Then for a smooth function  $f$ on $GL_n$ we have 
\begin{multline*}
\delta f(Z(X,Y))= \left\langle\nabla f,\delta Z\right\rangle\\
=\left\langle  (X_{<0})^{-1} \pi_{\leq 0} ( \nabla f \cdot Z) (X_{\geq 0})^{-1}, \delta X \right\rangle 
- \left\langle  (Y_{\leq 0})^{-1} \pi_{<0} ( Z \nabla f ) (Y_{>0})^{-1}, \delta Y \right\rangle.
\end{multline*}

Therefore, if we denote $\tilde f(X,Y) = f\circ Z(X,Y)$ then
\begin{align*}
&X \nabla_X \tilde f = \Ad_{X_{\geq 0}} \pi_{\leq 0} (\nabla f \cdot Z),\\
&Y \nabla_Y \tilde f = - \Ad_{Y_{>0}} \pi_{<0} (Z \nabla f ),\\
&\nabla_X \tilde f\cdot X = \Ad_{(X_{<0})^{-1}} \pi_{\leq 0} (\nabla f \cdot Z) \in \b_-,\\
&\nabla_Y \tilde f \cdot Y= - \Ad_{(Y_{\leq 0})^{-1}} \pi_{<0} (Z \nabla f ) \in \n_-,
\end{align*}
and
\begin{align*}
E_R \tilde f &= \Ad_{Y_{>0}} \left ( \Ad_Z \pi_{\leq 0} (\nabla f \cdot Z) - \pi_{<0} (Z \nabla f )\right )\\
&= \Ad_{Y_{>0}} \left ( Z \nabla f - \Ad_Z \pi_{>0} (\nabla f \cdot Z) - \pi_{<0} (Z \nabla f )\right )\\
&=\Ad_{Y_{>0}} \left (\pi_{\geq 0} (Z \nabla f ) - \Ad_Z \pi_{>0} (\nabla f \cdot Z) \right )\in \b_+.
\end{align*}
Plugging into \eqref{sklyadouble} we obtain
\begin{align*}
\{ f_1\circ Z, f_1\circ Z\}_D&= \frac{1}{2} \langle\nabla  f_1\cdot Z, \nabla  f_2\cdot Z\rangle_0
- \frac{1}{2} \langle Z\nabla  f_1, Z\nabla  f_2\rangle_0 + \langle X\nabla_X  \tilde f_1, Y\nabla_Y \tilde f_2\rangle \\
&= \frac{1}{2} \langle\nabla  f_1\cdot Z, \nabla  f_2\cdot Z\rangle_0
- \frac{1}{2} \langle Z\nabla  f_1, Z\nabla  f_2\rangle_0\\
&\qquad - \langle \Ad_Z\pi_{\leq 0}( \nabla f_1 \cdot Z), \pi_{<0}(Z \nabla f_2)\rangle.
\end{align*}
The last term can be rewritten as
$
\langle Z \nabla f_1 , \pi_{<0}(Z \nabla f_2)\rangle - \langle \pi_{>0} ( \nabla f_1 \cdot Z) , Z \nabla f_2)\rangle 
$. 
Comparing with 
\eqref{sklyabra},
we obtain
$\{ f_1\circ Z, f_1\circ Z\}_D = \{f_1, f_2\}_r \circ Z$.
\end{proof}

We are now ready to deal with the three families out of four.

\begin{lemma}
\label{fgh}
Families of functions $\{ f_{ij}\}, \{ g_{ij}\}, \{ h_{ij}\}$ are log-canonical with respect to $\Poi_D$.
\end{lemma}

\begin{proof} If $\fy_1(X,Y) = g_{ij}(X)$ and $\fy_2(X,Y) = g_{\alpha\beta}(X)$ (or
$\fy_1(X,Y) = h_{ij}(Y)$ and $\fy_2(X,Y) = h_{\alpha\beta}(Y)$) then
$\{ \fy_1,\fy_2\}_D = \{ \fy_1,\fy_2\}_r$. Furthermore, Proposition~4.19 in \cite{GSVb} specialized to the $GL_n$ case implies that in both cases
\begin{equation}
\label{aux_psi}
 \{ \log\fy_1,\log\fy_2\}_r = \frac{1}{2} \langle \xi_{1,L} , \xi_{2,L}\rangle_0 - \frac{1}{2} 
 \langle \xi_{1,R} , \xi_{2,R}\rangle_0,
 \end{equation}
provided $i-j\geq \alpha-\beta$ ($i-j\leq \alpha-\beta$, respectively),
 where $\xi_{\bullet,L}$, $\xi_{\bullet,R}$ are projections of the left and right gradients
 of $\log\fy_\bullet$
to the diagonal subalgebra. These projections are constant due to the homogeneity of all functions involved with respect to both left and right multiplication by diagonal matrices. Thus, families $\{ g_{ij}\}$, $\{ h_{ij}\}$ are log-canonical. The claim about the family $\{ f_{ij}\}$ now follows from Lemma~\ref{Zmap} and the third
 equation  in~\eqref{canform}. 
\end{proof}

The remaining family $\{\fy_{ij}\}$ is treated separately.

\begin{lemma}
\label{fylc}
The family  $\{ \fy_{kl}\}$ is log-canonical with respect to $\Poi_D$.
\end{lemma}

\begin{proof} Since $\det X$ is a Casimir function for $\Poi_D$, we only need to show that functions $\tilde \fy_{kl}=\det \Phi_{kl}$ are log-canonical with respect to the Poisson bracket
\begin{equation}
\label{PoissonU}
\{\fy_1,\fy_2\}_* = \left\langle R_+([\nabla\fy_1,U]), [\nabla\fy_2,U]\right\rangle - \left\langle [\nabla\fy_1,U], \nabla \fy_2 \cdot U\right\rangle,
\end{equation}
which one obtains from \eqref{sklyadoubleGL2} by assuming that $f(X,Y)=\fy(X^{-1}Y)$. In other words, $\Poi_*$ is the push-forward of $\Poi_D$ under the map
$(X,Y) \mapsto U=X^{-1}Y$.

Let $C=e_{21} + \cdots + e_{n, n-1} + e_{1n}$ be the cyclic permutation matrix. By Lemma~\ref{BC}, we write $U$ as 
\begin{equation}
\label{normalU}
U = N_- B_+ C N_-^{-1},
\end{equation}
where $N_-$ is unipotent lower triangular and $B_+$ is upper triangular.
Since functions $\tilde \fy_{kl}$ are invariant under the conjugation by unipotent lower triangular matrices, we have 
 $\tilde \fy_{kl}(U)=\tilde\fy_{kl}(B_+C)$.
Furthermore, 
\[
\left((B_+C)^i\right)^{[n]} = b_{ii} \cdots b_{11} e_i + \sum_{s < i} \alpha_{is} e_s
\]
 for $i\le n$, where $b_{ij}$, $1\le i\le j\le n$, are the entries
of $B_+$. It follows that 
\begin{equation}
\begin{split}
\label{tildefy}
\tilde \fy_{kl}(U)&= \pm \left (\prod_{s=1}^{n-k-l+1}b_{ss}^{n-k-l-s+2}\right )\det (B_+)_{[n-k-l+2, n-k]}^{[n-l+2, n]}\\
&= \pm \det U^{n-k-l+1}\frac{h_{n-k-l+2,n-l+2}(B_+)}{\prod_{s=2}^{n-k-l+2}h_{ss}(B_+)}.
\end{split}
\end{equation}

\begin{remark}\label{tildefysign} It is easy to check that the sign in the first line of~\eqref{tildefy} equals $(-1)^{k(n-k)+(l-1)(n-k-l+1)}s_{kl}$. 
We will use this fact below 
in the proof of Theorem~\ref{structure}(ii).
\end{remark}

Note that $\det U=\det B_+C= \pm\prod_{s=1}^{n}b_{ss} $ is a Casimir function for~\eqref{PoissonU}. Therefore to prove Lemma~\ref{fylc} it suffices to show that
functions $\det (B_+)_{[i, i+ n-j ]}^{[j, n]}$, $2\leq i \leq j \leq n$, are log-canonical with respect to $\Poi_*$ as functions of $U$. To this end, we first will compute the push-forward of $\Poi_*$ under the 
map $U \mapsto B_+'=B_+'(U)=(B_+)_{[2,n]}^{[2,n]}$ of $GL_n$ to the space $\B_{n-1}$ of $(n-1)\times (n-1)$ invertible upper triangular matrices. 

Let $S= e_{12} + \cdots + e_{m-1, m}$ be the $m\times m$ upper shift matrix. For an $m\times m$ matrix $A$, define 
$$
\tau(A)= S A S^T.
$$

\begin{lemma}
\label{PoissonB}

Let $f_1, f_2$ be two differentiable functions on $\B_{n-1}$. Then
\[
\{f_1\circ B_+', f_2\circ B_+'\}_* = \{f_1, f_2\}_b\circ B_+',
\]
where $\Poi_b$ is defined by
\begin{equation}
\label{brackB}
\{f_1, f_2\}_b(A) = \{f_1, f_2\}_r(A) + \frac{1}{2} \left (  \left\langle A \nabla f_1, \tau( (\nabla f_2) A)\right\rangle_0 -  
\left\langle \tau( (\nabla f_1) A), A \nabla f_2 \right\rangle_0   \right )
\end{equation}
for any $A\in \B_{n-1}$.
\end{lemma}

\begin{proof} We start by computing an infinitesimal variation of $B_+$ as a function of $U$. From~\eqref{normalU}, we obtain 
$\left (\Ad_{N_-^{-1}}\delta U\right )C^{-1} = 
[ N_-^{-1}\delta N_-, B_+C] C^{-1} + \delta B_+$.
Then 
\begin{equation}
\label{aux1}
\pi_{< 0}\left( \left (\Ad_{N_-^{-1}}\delta U\right )C^{-1} \right ) = \pi_{< 0}\left( [ N_-^{-1}\delta N_-, B_+C] C^{-1}\right ).
\end{equation}
If we define $\LL : \n_- \to \n_-$  via $\LL(\nu_-) = - \pi_{< 0}\left( \ad_{B_+C}(\nu_-) C^{-1}\right)$ for $\nu_-\in\n_-$ then~\eqref{aux1} above implies
that $N_-^{-1}\delta N_- = \LL^{-1} \left( \pi_{< 0}\left( \left (\Ad_{N_-^{-1}}\delta U\right )C^{-1} \right ) \right)$. 
Invertibility of $\LL$ is easy to establish
by observing that~\eqref{aux1} can be written as a triangular linear system for matrix entries of $N_-^{-1}\delta N_-$.
The operator $\LL^* : \n_+ \to \n_+$ dual to $\LL$ with respect to 
$\langle\cdot,\cdot \rangle$ acts on $\nu_+\in \n_+$ as
\begin{equation}
\label{Lstar}
\begin{split}
\LL^*(\nu_+) &= \pi_{>0} \left ( \ad_{B_+C} (C^{-1} \nu_+)  \right )
=\pi_{>0} \left ( B_+ \nu_+ - C^{-1} \nu_+ B_+ C  \right )\\
& = B_+ \nu_+ - S \nu_+ B_+ S^T 
= \left (\one - \tau\circ \Ad_{B_+^{-1}}\right ) (B_+ \nu_+).
\end{split}
\end{equation}
Note that $\LL^*$ extends to an operator on $\mathfrak{gl}_n$ given by the same formula and invertible due to the fact that
$\tau\circ \Ad_{B_+^{-1}}$ is nilpotent.

Let $f$ be a differentiable function on $\B_n$. Denote $\tilde f (U)=(f\circ B_+)(U)$. 
Then
\begin{equation}
\label{nablafu}
\begin{split}
\langle \nabla\tilde f, & \delta U\rangle = \langle \nabla f, \delta B_+\rangle = 
\langle C^{-1}\nabla f, \left (\Ad_{N_-^{-1}}\delta U\right ) - [ N_-^{-1}\delta N_-, B_+C]  \rangle\\
& =\langle \Ad_{N_-}(C^{-1} \nabla f ), \delta U\rangle + \left\langle [C^{-1}\nabla f, B_+C ], \LL^{-1} \left(\pi_{< 0}\left(\left (\Ad_{N_-^{-1}}\delta U\right )C^{-1}\right)\right)\right\rangle\\
&=\langle \Ad_{N_-}\left (C^{-1} \left ( \nabla f + (\LL^*)^{-1}\left(\pi_{> 0} (   [  C^{-1}\nabla f, B_+C ]   )\right)\right ) \right ), \delta U\rangle;
\end{split}
\end{equation}
in the last equality we have used the identities $\LL^{-1}=\pi_{<0}\circ\LL^{-1}$
and $(\LL^*)^{-1}=\pi_{>0}\circ(\LL^*)^{-1}$.

From now on we assume that $f$ depends only on $B_+'$, that is, does not depend on the first row of 
$B_+$. Thus, the first column of $\nabla f$ is zero.
Define $\zeta^f=\pi_{> 0} (   [  C^{-1}\nabla f, B_+C ]   )$ and $\xi^f = \nabla f+ (\LL^*)^{-1}(\zeta^f)$. Clearly,
$$
\zeta^f= \pi_{> 0} ( C^{-1}\nabla f \cdot B_+C -  B_+ \nabla f )= \pi_{> 0}( \tau (\nabla f\cdot B_+)) - \pi_{> 0} ( B_+\nabla f )
$$
and
\begin{align*}
[C^{-1}\xi^f, B_+C] &= C^{-1}\xi^f B_+C - B_+ \xi^f = C^{-1}\xi^f B_+ S^T - B_+ \xi^f \\
& = S\xi^f B_+ S^T - B_+ \xi^f + e_{n1} \xi^f B_+C \\
&= -  \left (\one - \tau\circ \Ad_{B_+^{-1}}\right ) (B_+ \xi^f) + e_{n1} \xi^f B_+C,
\end{align*}
which is equivalent to
\begin{equation}\label{aux2}
[C^{-1}\xi^f, B_+C]= -\LL^*( \xi^f)  + e_{n1} \xi^f B_+C 
\end{equation}
by~\eqref{Lstar}.  Furthermore,
$$
\LL^*( \xi^f) = \LL^*(\nabla f) + \zeta^f 
= \pi_{\leq 0} \left ( B_+ \nabla f- \tau (\nabla f \cdot B_+)\right ).
$$
Consequently, $[C^{-1}\xi^f, B_+C]\in \b_-$. 
Using this fact and the invariance of the trace form, for any $f_1, f_2$ that depend only on $B_+'$ we can now compute $\{ f_1\circ B_+, f_2\circ B_+\}_*$ from~\eqref{PoissonU} and~\eqref{nablafu} as 
\begin{equation}
\label{PoissonBraw}
\begin{split}
\{ f_1\circ B_+, f_2\circ B_+\}_* (U) &= \langle R_+ \left ([C^{-1}\xi^1, B_+C] \right ), [C^{-1}\xi^2, B_+C]  \rangle\\
& - \langle [C^{-1}\xi^1, B_+C] , C^{-1}\xi^2 B_+C\rangle,
\end{split}
\end{equation}
where $\xi^i = \xi^{f_i}$, $i=1,2$.
Thus, $\nabla^{i}=\nabla f_i$ are lower triangular matrices with zero first column, and so $\nabla^i B_+$, $B_+ \nabla^i$, $\xi^i$, $\xi^i B_+$, $B_+ \xi^i$ have zero first column as well, and $C^{-1}\xi^2 B_+C$ has zero last column. 
 We conclude that the second term in~\eqref{aux2} does not affect either term in the right hand side of~\eqref{PoissonBraw}.
In particular,  the first term in~\eqref{PoissonBraw} becomes
$$
\frac{1}{2} \left\langle B_+ \nabla^1- \tau (\nabla^1 B_+), B_+ \nabla^2- \tau (\nabla^2 B_+)\right\rangle_0,
$$
while the second can be re-written as 
\begin{equation}
\nonumber
\begin{split}
\big\langle \pi_{\leq 0} \left ( \tau (\nabla^1 B_+) - B_+ \nabla^1\right ),& C^{-1}\xi^2 B_+C\big\rangle\ = 
\left\langle S^T\pi_{\leq 0} \left ( \tau (\nabla^1 B_+) - B_+ \nabla^1 \right )S, \xi^2 B_+\right\rangle \\
&= \left\langle \pi_{\leq 0} \left ( \nabla^1 B_+\right ) - S^T\pi_{\leq 0} \left ( B_+ \nabla^1 \right ) S, \xi^2 B_+\right\rangle\\ 
&= \left\langle \pi_{\leq 0} \left ( \nabla^1 B_+\right ) - S^T\pi_{\leq 0} \left ( B_+ \nabla^1 \right ) S, \nabla^2 B_+\right\rangle\\
 & +  \left\langle  \pi_{\leq 0} \left ( \nabla^1 B_+\right ) - S^T\pi_{\leq 0} \left ( B_+ \nabla^1 \right ) S, (\LL^*)^{-1} (\zeta^2)B_+\right\rangle,
\end{split}
\end{equation}
where $\zeta^2=\zeta^{f_2}$.

The last term can be transformed as
\begin{equation}
\nonumber
\begin{split}
\langle  \Ad_{B_+}\left(\pi_{\leq 0}( \nabla^1 B_+) \right.&\left.- S^T\pi_{\leq 0} ( B_+ \nabla^1) S\right), B_+(\LL^*)^{-1} (\zeta^2)\rangle\\
& = \langle \left (\one - \Ad_{B_+}\circ \tau^* \right ) \pi_{\leq 0} ( B_+ \nabla^1 ), B_+(\LL^*)^{-1} (\zeta^2)\rangle\\
& = \langle  \pi_{\leq 0} \left ( B_+ \nabla^1 \right ), \left (\one - \tau\circ \Ad_{B_+^{-1}}\right ) B_+ (\LL^*)^{-1} (\zeta^2)\rangle\\ 
&= \langle  \pi_{\leq 0} \left ( B_+ \nabla^1 \right ), \LL^* ((\LL^*)^{-1} (\zeta^2))\rangle\\
& = \langle  \pi_{\leq 0} \left ( B_+ \nabla^1 \right ), \pi_{> 0}( \tau (\nabla^2 B_+)) - \pi_{> 0} ( B_+\nabla^2  )
\rangle.
\end{split}
\end{equation}
Here, in the first equality we used the fact that $(\LL^*)^{-1} (\zeta^2) \in \n_+$, and that
$$
\langle  \Ad_{B_+}\pi_{\leq 0}( \nabla^1 B_+),A\rangle=\langle \pi_{\le0}\left( \Ad_{B_+}\pi_{\leq 0}( \nabla^1 B_+)\right),A\rangle
=\langle \pi_{\le0}\left( \Ad_{B_+} (\nabla^1 B_+)\right),A\rangle
$$
for $A\in \b_+$.

Combining our simplified expressions for two terms in the right hand side of \eqref{PoissonBraw} and taking into account that
$\langle\tau(\nabla^1B_+),\tau(\nabla^2B_+)\rangle_0=\langle\nabla^1B_+,\nabla^2B_+\rangle_0$
we obtain
\begin{equation*}
\begin{split}
\{f_1\circ B_+, f_2\circ &B_+\}_*(U) \\
&= \frac{1}{2} \left (  \left\langle B_+ \nabla^1, \tau( \nabla^2 B_+)\right\rangle_0 -  \left\langle \tau( \nabla^1 B_+), B_+ \nabla^2 \right\rangle_0   
 \right)  + \{f_1, f_2\}_r(B_+)
\end{split}
\end{equation*}
for functions $f_1, f_2$ on $\B_n$ that depend only on $B_+'$. To complete the proof of Lemma \ref{PoissonB}, it remains to observe that for such functions, the right hand side does not depend on the first row of $B_+$ and is equal to a similar expression in which $B_+$ is replaced with $B_+'$ and the bracket $\Poi_r$ and the forms 
$\langle\cdot ,\cdot \rangle$, $\langle\cdot , \cdot\rangle_0$ are replaced with their counterparts
for $GL_{n-1}$.
\end{proof}

Now we can finish the proof of Lemma~\ref{fylc}. Let functions $f_1, f_2$ belong to the family 
$\{ \log\det (B_+)_{[i, i+ n-j ]}^{[j, n]}= \log\det (B_+')_{[i-1, i+ n-j ]}^{[j-1, n]},\ 2\leq i \leq j \leq n\}$. 
Then the second term in our expression~\eqref{brackB} for $\{f_1, f_2\}_b(B_+')$ is constant
because of the homogeneity of minors of $B_+'$ under right and left diagonal multiplication, and the first term is constant because, as we discussed earlier,  
functions $\det (B_+')_{[i-1, i+ n-j ]}^{[j-1, n]}$ are log-canonical with respect to $\Poi_r$ (see, e.g.,~\eqref{aux_psi}).
\end{proof}

This ends the proof of Theorem~\ref{basis}.

\begin{remark} 
\label{MMJcompat} 
The bracket \eqref{brackB}  can be extended to the entire $GL_{n-1}$. In fact, the right hand side makes sense for $GL_m$ for any  $m\in \mathbb{N}$. 
It can be induced via the map 
$\T_{m}\times \T_{m}\times GL_{m} \ni (H=\diag (h_1,\ldots, h_{m}), \tilde H=\diag (\tilde h_1,\ldots, \tilde h_{m}), X) \mapsto H X\tilde H\in GL_{m}$ if one equips $\T_{m}\times \T_{m}$ with a Poisson bracket $\{h_i, h_j\} = \{\tilde h_i, \tilde h_j\} =0, \{h_i, \tilde h_j\}=\delta_{i,j-1}$. It follows from~\cite[Prop.~2.2]{LMP} that right and left diagonal multiplication generates a global toric action for the
standard cluster structure on $GL_m$
(and on double Bruhat cells in $GL_m$), for which $\Poi_r$ is a compatible Poisson structure. 
Therefore, the above extension of~\eqref{brackB} to the entire group is compatible with this cluster structure as well.
\end{remark}

\section{Proof of Theorem~\ref{structure}(i)}\label{prgcs} 

\subsection{Toric action}\label{torac}
Let us start from the following important statement.

\begin{theorem}\label{gtaind}
The action 
\begin{equation}\label{lraction}
(X,Y)\mapsto (T_1 X T_2, T_1 Y T_2)
\end{equation}
 of right and left multiplication by diagonal matrices 
is $\GCC_n^D$-extendable to a global toric action on $\FF_{\C}$.
\end{theorem}

\begin{proof}
For an arbitrary vertex $v$ in $Q_n$
denote by $x_v$ the cluster variable attached to $v$. If $v$ is a mutable vertex, then the {\em $y$-variable} ($\tau$-variable in the terminology of \cite{GSV1}) corresponding to $v$ is defined as 
$$
y_v = \frac{\prod_{(v\to u) \in Q} x_u}{\prod_{(w\to v) \in Q} x_w}.
$$
Note that the product in the above formula is taken over all arrows,
so, for example, $\fy_{21}^2$ enters the numerator of the $y$-variable
corresponding to $\fy_{12}$. 
By Proposition~\ref{globact}, to prove the theorem it suffices to check that $y_v$ is a homogeneous function of degree zero with respect to the
action~\eqref{lraction} (see~\cite[Remark 3.3]{GSVMMJ} for details), and that the Casimirs $\hat p_{1r}$ are invariant under~\eqref{lraction}.

Let us start with verifying the latter condition. According to Section~\ref{logcan}, $\hat p_{1r}=c_r^ng_{11}^{r-n}h_{11}^{-r}$, $1\le r\le n-1$.
It is well-known that functions $c_i$ are Casimirs for the Poisson-Lie bracket~\eqref{sklyadoubleGL1} on $D(GL_n)$, as well as $g_{11}$ and $h_{11}$.
Therefore, $\hat p_{1r}$ are indeed Casimirs. Their invariance under~\eqref{lraction} is an easy calculation.

Next, for a function $\psi(X,Y)$ on $D(GL_n)$ homogeneous with respect 
to~\eqref{lraction}, define
%$(X,Y)\to (D_1 X D_2, D_1 X D_2)$, where $D_1, D_2$  are diagonal matrices, let 
the {\em left\/} and {\em right weights} $\xi_L(\psi)$, $\xi_R(\psi)$ 
of $\psi$ as the {\em constant\/} diagonal matrices 
$\pi_0(E_L \log\psi)$ and $\pi_0(E_R \log\psi)$. 
Recall that all functions $g_{ij}$, $h_{ij}$, $f_{kl}$, $\fy_{kl}$
possess this homogeneity property.

For $1\leq i \leq j \leq n$, let $\Delta(i,j)$ denote a diagonal matrix with $1$'s in the entries $(i,i), \ldots, (j,j)$ and $0$'s everywhere else. 
It follows directly from the definitions in Section~\ref{logcan} that
\begin{equation}\label{weights}
\begin{aligned}
& \xi_L(g_{ij})= \Delta(j, n+j -i),\quad\xi_R(g_{ij})= \Delta(i,n),  \\
& \xi_L(h_{ij})= \Delta(j, n),\quad   \xi_R(h_{ij})= \Delta(i,n+i-j),\\
& \xi_L(f_{kl})= \Delta(n-k+1, n)+ \Delta(n-l+1, n),\quad 
\xi_R(f_{kl})= \Delta(n-k-l+1,n), \\
& \xi_L(\fy_{kl})= (n-k-l)\left(\one + \Delta(n,n)\right ) + \Delta(n-k+1,n)+ \Delta(n-l+1,n),\\ 
& \xi_R(\fy_{kl})= (n-k-l+1)\one.
\end{aligned}
\end{equation}
Now the verification of the claim above 
becomes a straightforward calculation 
based on the description of $Q_n$ in Section~\ref{init} 
and the fact that for a monomial in 
homogeneous functions $M=\psi_1^{\alpha_1}\psi_2^{\alpha_2}\cdots$ the 
right and left weights are  $\xi_{R,L}(M) = \alpha_1 \xi_{R,L}(\psi_1) + \alpha_2 \xi_{R,L}(\psi_2) + \cdots$. 

For example, if $v$ is the vertex associated with the function $h_{ii}$,
$i\ne 1$, $j\ne n$, $i\ne j$, then 
\begin{align*}
\xi_R(y_v) &= \xi_R(h_{i-1,i}) + \xi_R(f_{1,n-i}) -\xi_R(h_{i,i+1}) - \xi_R(f_{1,n-i+1})\\ 
&= \Delta(i-1, n-1)+ \Delta(i,n)- \Delta(i, n-1)  - \Delta(i-1,n)= 0
\end{align*}
and
\begin{align*}
\xi_L(y_v) &= \xi_L(h_{i-1,i}) + \xi_L(f_{1,n-i}) -\xi_L(h_{i,i+1}) - \xi_L(f_{1,n-i+1})\\
& = \Delta(i, n) + \Delta(i+1,n) + \Delta(n,n)
- \Delta(i+1, n) -\Delta(n,n) - \Delta(i,n)= 0.
\end{align*}
Other vertices are treated in a similar way.
\end{proof}

\subsection{Compatibility}\label{cmpty}
Let us proceed to the proof of the compatibility statement of 
Theorem~\ref{structure}(i). We have already seen above that $\hat p_{1r}$ are Casimirs of the bracket $\Poi_D$.
Therefore, by Proposition~\ref{compatchar}, it suffices to 
show that for every mutable vertex $v\in Q_n$
\begin{equation}\label{compcond}
\{\log x_u, \log y_v\}_D = \lambda \delta_{u,v}\quad\text{for any  $u\in Q_n$}, 
\end{equation}
where $\lambda$ is some nonzero rational number not depending on $v$.

Let $v$ be a mutable $g$-vertex in $Q_n$ and  
$u$ be a vertex in one of the other three regions of $Q_n$. Then to show that $\{\log x_u, \log y_v\}_D=0$ one can use~\eqref{ourinv_prop}, 
Lemma~\ref{cross-inv} and the proof of Theorem~\ref{gtaind}, which implies that 
\begin{align*}
\pi_0\left (X\nabla_X \log y_v\right )& = \pi_0\left (E_R \log y_v\right ) = \xi_R(y_v)=0,\\ 
\pi_0\left (\nabla_X \log y_v X\right )& = \pi_0\left (E_L \log y_v\right ) = \xi_L(y_v)=0.
\end{align*} 
The same argument works if $u$ and $v$ belong to any two different regions of the quiver $Q_n$.

Thus, to complete the proof it remains to verify~\eqref{compcond} for vertices $u, v$ in the same region of $Q_n$.
In view of~\eqref{canform}, for $g$- and $h$-vertices other than vertices corresponding to $g_{ii}$ and $h_{ii}$, this becomes a particular case of Theorem~4.18
in~\cite{GSVb} which establishes the compatibility of the standard 
Poisson--Lie structure on a simple Lie group $\G$ with the cluster 
structure on double Bruhat cells in $\G$. We just need to set $\G=GL_n$ (a 
transition to $GL_n$ from a simple group $SL_n$ is straightforward), set 
$\lambda$ in~\eqref{compcond} to be equal to $-1$  and apply the theorem to $\G^{id,w_0}, \G^{w_0,id}$ in the case of $g$- and $h$-regions, respectively (here $w_0$ is the longest permutation of length $n-1$). 

Vertices corresponding to $g_{ii}$ and $h_{ii}$ are treated separately, because in quivers for $\G^{id,w_0}, \G^{w_0,id}$ they would have been frozen. For any such vertex $v$ we only need to check that $\{\log x_v, \log y_v\}_D = -1$. Using the description of $Q_n$ in Section~\ref{init},
%the second and 
the third equation in Lemma~\ref{cross-inv}, the second and the third lines in~\eqref{weights}, and equation~\eqref{aux_psi}, we compute
\begin{equation*}
\begin{aligned}
\left \{ \log h_{ii},\log\frac{f_{1, n-i}   h_{i-1,i}} {f_{1, n-i+1}   h_{i,i+1}}\right \}_D 
&= \frac 1 2  \langle  \Delta(i,n) ,2\Delta(i+1,n) - 3\Delta(i,n) +\Delta(i-1,n)\\
&\qquad + \Delta(i-1,n-1) - \Delta(i,n-1)\rangle_0\\
& = \frac 1 2 \left \langle \Delta(i,n) , 2(\Delta(i-1,i-1) -\Delta(i,i)) \right \rangle_0 = -1
\end{aligned}
\end{equation*}
for $1<i<n$. Using in addition the first equation in Lemma~\ref{cross-inv}
and the first line in~\eqref{weights} we get
\begin{equation*}
\begin{aligned}
\left \{ \log h_{nn},\log\frac{ h_{n-1,n}g_{nn}} {f_{11}}\right \}_D 
&= \frac 1 2  \langle  \Delta(n,n) ,\Delta(n-1,n)- 3\Delta(n,n) + \Delta(n-1,n-1)\rangle_0\\
& = \frac 1 2 \left \langle \Delta(n,n) , 2(\Delta(n-1,n-1) -\Delta(n,n)) \right \rangle_0 = -1.
\end{aligned}
\end{equation*}
Vertices corresponding to $g_{ii}$ are treated in a similar way.

Now, let us turn to the $f$-region. Let $v$ be a vertex that corresponds to $f_{kl}$, $k,l\ge 1$, $k+l< n$, then by the last equality in~\eqref{canform},
\begin{equation}\label{yftoyh}
y_v = \frac{f_{k+1,l-1} f_{k,l+1}f_{k-1,l}} {f_{k+1,l} f_{k-1,l+1}f_{k,l-1}}(X,Y) = \frac{h_{\alpha,\beta-1} h_{\alpha-1,\beta}h_{\alpha+1,\beta+1}} 
{h_{\alpha,\beta+1}h_{\alpha+1,\beta}h_{\alpha-1,\beta-1}}(Z),
\end{equation}
where $\alpha=n-k-l+1$, $\beta=n-k+1$, and $Z=(Y_{>0})^{-1}X_{\geq 0}$. Consequently, if $u$ is a vertex that corresponds to $f_{k'l'}$, 
and $\alpha'=n-k'-l'+1$, $\beta'=n-k'+1$, then
$$
\{ \log x_u, \log y_v\}_D = \{\log h_{n-l'+1,n-l'+1}(Y), \log y_v\}_D +  \{\log h_{\alpha'\beta'}, \log y_v\}_r(Z)
$$ 
%&=\{\log h_{i'j'}, \log y_v\}_r(Z). 
by Lemma~\ref{Zmap}. The first term in the right hand side vanishes,
as it was already shown above (this corresponds to the case when one vertex belongs to the $h$-region and the other to the $f$-region), and we are left
with
\begin{equation}\label{Dtor}
\{ \log x_u, \log y_v\}_D = \{\log h_{\alpha'\beta'}, \log y_v\}_r(Z).
\end{equation}

Consider the subquiver of $Q_n$ formed by all $f$-vertices, 
as well as vertices (viewed as frozen) that correspond to 
functions $g_{ii}(X)$ and $\fy_{n-i, i}(X,Y)$. It is isomorphic, up to edges 
between the frozen vertices, 
to the $h$-part of $Q_n$ in which vertices corresponding to $h_{ii}(Y)$ are 
viewed as frozen. The isomorphism consists in sending the vertex occupied by 
$f_{kl}(X,Y)$ to the vertex occupied by $h_{\alpha\beta}(Z)$, including the 
values of $k,l$ subject to identifications of Remark~\ref{identify}.
The latter is possible since $g_{ii}(X)=h_{ii}(Z)$ by the second equation in~\eqref{canform},  
and since the third equation in~\eqref{canform} can be extended to the cases $k=0$ and $l=0$ by setting $h_{i,n+1}\equiv1$ for any $i$. It now follows
from~\eqref{yftoyh}, \eqref{Dtor} that this isomorphism takes~\eqref{compcond} for 
$f$-vertices to the same statement for $h$-vertices, which has
been already proved. 

 We are left with the $\fy$-region.
If $v$ is a vertex that corresponds to 
$\fy_{kl}$, $k > 1$, $l>1$, $k+l < n$, then by~\eqref{tildefy}
\begin{equation}\label{fyviah}
y_v = \frac{\fy_{k,l-1} \fy_{k-1,l+1}\fy_{k+1,l}} {\fy_{k+1,l-1}
\fy_{k,l+1}\fy_{k-1,l}}(X,Y) = \frac{h_{\alpha-1,\gamma} h_{\alpha+1,\gamma+1}h_{\alpha,\gamma-1}}
{h_{\alpha,\gamma+1} h_{\alpha+1,\gamma}h_{\alpha-1,\gamma-1}}(B_+'),
\end{equation}
where $\alpha=n-k-l+2$, $\gamma=n-l+2$, $B_+'=(B_+)_{[2,n]}^{[2,n]}$
(here we use the identity $h_{ij}(B_+)=h_{i-1,j-1}(B_+')$ for $i,j>1$). 
In view of Lemma~\ref{PoissonB} and Remark~\ref{MMJcompat}, we can
establish~\eqref{compcond} for $v$ by applying the same reasoning
as in the case of $f$-vertices.  To include the case $k=1$ it suffices 
to use the same convention $h_{i,n+1}\equiv1$ as above. 

 Therefore, the only vertices left to
consider are the ones corresponding to $\fy_{k,n-k}$, $1\le k\le n-1$, and
$\fy_{k1}$, $1\le k\le n-2$. They are treated by straightforward, albeit lengthy,  
calculations based on Lemma~\ref{PoissonB}, equations~\eqref{tildefy}, \eqref{aux_psi}, and the fourth equation in Lemma~\ref{cross-inv}. 
Note that in all the cases equation~\eqref{compcond} is satisfied with $\lambda=-1$.

\begin{remark}\label{Brank}
It follows from~\eqref{compcond} and Proposition~\ref{compatchar} that the exchange matrix corresponding to the extended seed
$\widetilde{\Sigma}_n=(F_n, Q_n, \P_n)$ 
is of full rank. 
\end{remark}

\subsection{Irreducibility: the proof of Theorem~\ref{irredinbas}}
The claim is clear for the functions $g_{ij}$, $h_{ij}$ and $f_{kl}$, since each one of them is a minor of
a matrix whose entries are independent variables (see e.g.~\cite[Ch.~XIV, Th.~1]{Boch}). 
Functions $c_k$, $1\le k\le n-1$, are sums of such minors. Consequently, each $c_k$
is linear in any of the variables $x_{ij}$, $y_{ij}$. Assume that $c_k=P_1P_2$, where $P_1$ and $P_2$ are nonconstant polynomials, 
and that $P_1$ is linear in $y_{11}$, hence $P_2$ does not depend
on $y_{11}$. If $P_2$ is linear in any of $y_{1j}$, $2\le j\le n$, then $c_k$ includes a multiple of the product $y_{11}y_{1j}$, a contradiction. 
Therefore, $P_1$ is linear in any one of $y_{1j}$, $1\le j\le n$. If $P_2$ is linear in $z_{ij}$ for some values of $i$ and $j$ (here $z$ is either $x$ or $y$) 
then $c_k$ includes a multiple of the product $y_{1j}z_{ij}$, a contradiction. Hence, $P_2$ is trivial, which means that $c_k$ is irreducible.

Our goal is to prove  
\begin{proposition}\label{irredfy}
$\fy_{kl}(X,Y)$ is an irreducible polynomial in the entries of $X$ and $Y$.
\end{proposition}

\begin{proof} We first aim at a weaker statement: 

\begin{proposition}\label{irredfyu}
$\fy_{kl}(I,Y)$ is irreducible in the entries of $Y$. 
\end{proposition}

\begin{proof}
In this case $U=X^{-1}Y=Y$,
so we write $\psi_{kl}(U)$ instead of $\fy_{kl}(I,Y)$. It will be convenient to indicate explicitly the size of $U$ and 
to write $\psi_{kl}^{[n]}(U)$ for the function $\psi_{kl}$ evaluated on an $n \times n$ matrix
$$
U=\begin{pmatrix} u_{11} & u_{12}& \dots & u_{1n}\cr
                 u_{21} & u_{22} & \dots & u_{2n}\cr
								 \vdots &  &  & \vdots \cr
								 u_{n1} & u_{n2} & \dots & u_{nn}
 \end{pmatrix}.
$$								

We start with the following observations.

\begin{lemma}\label{coeff}
{\rm (i)} Consider $\psi_{kl}^{[n]}(U)$ as a polynomial in $u_{jn}$, then its leading coefficient does not depend on the entries
$u_{ji}$, $1\le i\le n-1$, nor on $u_{in}$, $1\le i\le n$, $i\ne j$.

{\rm (ii)} The degree of  $\psi_{kl}^{[n]}(U)$ as a polynomial in $u_{1n}$ equals $n-k-l+1$.
\end{lemma}

\begin{proof} Explicit computation.
\end{proof}

Next, let us find a specialization of variables $u_{ij}$ such that the corresponding $\psi_{kl}(U)$ is irreducible. 

Define a polynomial in two variables $z, t$ by
\[
P_m(z,t)=t^{m-1}z^m+\sum_{i=0}^{m-2}(-1)^{m-i-1}\binom mi \frac{t^{m-1}-t^i}{t-1}z^i.
\]
This is an explicit expression for the determinant of the $m\times m$ matrix
\[
\begin{pmatrix} z & 1 & 1 & 1 &\dots & 1\cr
                1 & tz & t & t &\dots & t \cr
								1 & 1  & tz & t &\dots & t \cr
                \ddots & & &  & \ddots  & \cr
								1 & 1  & 1 & 1 & \dots & tz
\end{pmatrix}.
\]
								
\begin{lemma}\label{polytwo} For any $m\ge 2$,
$P_m(z,t)$ is an irreducible polynomial.
\end{lemma} 

\begin{proof}
It is easy to see that a polynomial in two variables is irreducible if its Newton polygon can not be represented as a Minkowski
sum of two polygons. 
The Newton polygon of $P_m(z,t)$ in the $(z,t)$-plane has the following vertices: $(m,m-1), (0,m-2), (0,m-3),\dots, (0,1), (0,0)$. 
It contains two non-vertical sides
connecting $(m,m-1)$ with $(0,m-2)$ and $(0,0)$, respectively; all the other sides are vertical. Assume that two polygons as above exist.
Consequently, both non-vertical sides should belong to the same
one out of the two. However, it is not possible to complete this polygon without  using all the remaining vertical sides, a contradiction.
\end{proof}

\begin{lemma}\label{expform}
Define 
an $n\times n$ matrix $M_{kl}$ via
\begin{gather*}
M_{kl}=\begin{pmatrix}
0 & 0 & 1 & z \cr
t & 0 & 0 & 1 \cr
0 &  \one_{n-3} & 0 & 1 \cr
0 & 0 & 0 & 1 \end{pmatrix} \qquad \text{if  $l=1$},\\
M_{kl}=\begin{pmatrix}
0 & 0 & 1 & 0 & 0 & 0 & z \cr
t & 0 & 0 & 0 & 0 &  0 & 1 \cr
0 & 0 & 0 & 0 & 0  & \one_{l-1} & 1 \cr
0 & 1 & 0 & 0 & 0  & 0 & 1 \cr
0 & 0 & 0 & 0 & \one_{m-1} & 0 & 1 \cr
0 & 0 & 0 & \one_{l-2} &  0 & 0 & 1 \cr
0 & 0 & 0 & 0 & 0 &  0 & 1 \end{pmatrix}\qquad \text{if $k\ge l$, $l>1$},\\
M_{kl}=\begin{pmatrix}
0 & 0 & 0 & 0 & 1 & 0 & z \cr
t & 0 & 0 & 0 & 0 & 0 & 1 \cr
0 & 0 & 0 & 0 & 0 & \one_{k-1} & 1 \cr
0 & 1 & 0 & 0 & 0 & 0 & 1 \cr
0 & 0 & 0 & \one_{m-2} & 0 & 0 & 1 \cr
0 & 0 & \one_{k-1} & 0 & 0 & 0 & 1 \cr
0 & 0 & 0 & 0 & 0 & 0 & 1 \end{pmatrix}\qquad \text{if $k\le l$, $l>1$}.
\end{gather*} 
where $m=\max\{n-2l, n-2k\}$. 
Then $\det M_{kl}=\pm t$, and
\[
\psi^{[n]}_{kl}(M_{kl})=\pm P_{n-k-l+1}(z,t).
\]
\end{lemma}

\begin{proof}
Explicit computation.
\end{proof}

Let us proceed with the proof of Proposition~\ref{irredfyu}. Assume to the contrary that $\psi_{kl}^{[n]}(U)=P_1P_2$ for some
values of $n$, $k$ and $l$, where both $P_1$ and $P_2$ are nontrivial. It follows from Lemmas~\ref{expform},~\ref{polytwo} and~\ref{coeff}(ii) that only
one of $P_1$ and $P_2$ depends on $u_{1n}$, say, $P_1$. 
Therefore, $P_2(M_{kl})$ is a nonzero constant. Consequently, $P_2$ contains a monomial 
$\prod_{j=3}^{n-1}u_{j\sigma_{kl}(j)}^{r_j}$,
where $r_j\ge 0$ and $\sigma_{kl}\in S_{n-1}$ is the permutation defined by the first $n-1$ rows and columns of $M_{kl}$. 
Assume there exists $j$ such that $r_j>0$, and consider $u_{jn}$. If $P_2$ depends on $u_{jn}$ then the leading coefficient of $\psi_{kl}^{[n]}(U)$ at
$u_{1n}$ depends on $u_{jn}$, in a contradiction to Lemma~\ref{coeff}(i). Therefore, $P_2$ does not depend on $u_{jn}$, and hence the leading coefficient of $\psi_{kl}^{[n]}(U)$ at
$u_{jn}$ depends on $u_{j\sigma_{kl}(j)}$, which again contradicts Lemma~\ref{coeff}(i). We thus obtain that $r_j=0$ for $3\le j\le n-1$, which means that
$P_2$ is a constant, and
Proposition~\ref{irredfyu} holds true.
\end{proof}

The next step of the proof is the following statement: 
\begin{proposition}\label{irredfyuinv}
$\fy_{kl}(X,I)$ is irreducible in the entries of $X$. 
\end{proposition}

\begin{proof} 
Note that
\begin{multline*}
\left[I^{[n-k+1,n]}\;(X^{-1})^{[n-l+1,n]}\;(X^{-2})^{[n]}\dots(X^{k+l-n-1})^{[n]}\right]=\hfill\\
X^{k+l-n-1}\left[(X^{n-k-l+1})^{[n-k+1,n]}\;(X^{n-k-l})^{[n-l+1,n]}\;(X^{n-k-l-1})^{[n]}\;\dots I^{[n]}\right],
\end{multline*}
hence it suffices to prove that the functions
$$
\bar\psi_{kl}=\det\left[(X^{n-k-l+1})^{[n-k+1,n]}\;(X^{n-k-l})^{[n-l+1,n]}\;(X^{n-k-l-1})^{[n]}\;\dots I^{[n]}\right]
$$ 
are irreducible. As in the proof of Proposition~\ref{irredfyu}, we write 
$\bar\psi_{kl}^{[n]}(X)$ for the function $\psi_{kl}$ evaluated on an $n \times n$ matrix
$$
X=\begin{pmatrix} x_{11} & x_{12}& \dots & x_{1n}\cr
                 x_{21} & x_{22} & \dots & x_{2n}\cr
								 \vdots &  &  & \vdots \cr
								 x_{n1} & x_{n2} & \dots & x_{nn}
 \end{pmatrix}.
$$

Similarly to the case of $\psi_{kl}^{[n]}(U)$, we have the following observations:

\begin{lemma}\label{coefpsi}
 {\rm (i)} Consider $\bar\psi_{kl}^{[n]}(X)$ as a polynomial in $x_{jn}$, then its leading coefficient does not depend on the entries
$x_{ji}$, $1\le i\le n-1$, and $x_{in}$, $1\le i\le n$, $i\ne j$.

{\rm (ii)} The degree of  $\bar\psi_{kl}^{[n]}(X)$ as a polynomial in $x_{1n}$ equals $n-k-l+1$.
\end{lemma}

\begin{proof} Explicit computation.
\end{proof}

As in the previous case, we find a specialization of variables $x_{ij}$, such that the corresponding $\bar\psi_{kl}(X)$
is irreducible.

\begin{lemma}\label{expformpsi}
Define 
an $n\times n$ matrix $\bar M_{kl}$ via
\begin{equation*}
\bar M_{kl}=\begin{pmatrix}
0 & 0 & 1 &  0 & 0 & 0 & 0 & z \cr
t & 0 & 0 &  0 & 0 & 0 & 0 & 1 \cr
0 & \one_{m-1} &  0 & 0 & 0 & 0 & 0 & 1 \cr
0 & 0 & 0 &  0 & 0 & 1 & 0 & 1 \cr
0 & 0 & 0 &  \one_{l-2} & 0 & 0 & 0 & 1 \cr
0 & 0 & 0 &  0 & 0 & 0 & \one_{k-2} & 1 \cr
0 & 0 & 0 &  0 & 1 & 0 & 0 & 1 \cr
0 & 0 & 0 &  0 & 0 & 0 & 0 & 1 \end{pmatrix},
\end{equation*}
where $m=n-k-l$.
Then $\det \bar M_{kl}=\pm t$, and 
\[
\bar\psi^{[n]}_{kl}(\bar M_{kl})=\pm P_{n-k-l+1}(z,t).
\]
\end{lemma}

\begin{proof}
Similar to the proof of Lemma~\ref{expform}.
\end{proof}

The rest of the proof of Proposition~\ref{irredfyuinv} is parallel to the proof of Proposition~\ref{irredfyu}.
\end{proof}

Assume now that $\fy_{kl}(X,Y)=P_1P_2$. It follows from Propositions~\ref{irredfyu} and~\ref{irredfyuinv} 
 and the quasihomogeneity of $\fy_{kl}(X,Y)$
 that one
of $P_1, P_2$ depends only on $X$, and the other only on $Y$; moreover, both $P_1(I)$ and $P_2(I)$ are nonzero complex numbers.
Consequently, $\fy_{kl}(I,I)$ is a nonzero complex number, which contradicts the explicit expression for $\fy_{kl}(X,Y)$.
\end{proof}

The proof of Theorem~\ref{irredinbas} is complete.

\subsection{Regularity}
We now proceed with condition (ii) of Proposition~\ref{regcoin}.
By Theorem~\ref{irredinbas}, we have to check that for any mutable vertex of $Q_n$ the new variable in the
adjacent cluster is regular and not divisible by the corresponding variable in the initial cluster. 
Let us consider the exchange relations according to the type of the vertices. 

Regularity of the function $\fy_{11}'$ that replaces $\fy_{11}$ follows from Proposition~\ref{polyrel}. 
 Let us prove that this function is not divisible by $\fy_{11}$. 
Indeed, assume the contrary, and define an $n \times n$ matrix $\Sigma_{11}$ via
\[
\Sigma_{11}=\begin{pmatrix}
0 &  0 &  0 & 1 \cr
\one_{n-3} & 0 & 0 & 0 \cr
0 &   t & 0 & 0  \cr 
0 & 0 & 1 & 0  \end{pmatrix}.
\]
An explicit computation shows that $\det\Sigma_{11}=\pm t$ and 
\[
\fy_{11}(I,\Sigma_{11})=\pm t,\qquad \fy_{21}(I,\Sigma_{11})=\pm 1,\qquad \fy_{12}(I,\Sigma_{11})=0.
\]
Therefore,~\eqref{general} yields 
$$
\fy_{11}(I,\Sigma_{11})\fy'_{11}(I,\Sigma_{11})=\fy_{21}(I,\Sigma_{11})\det\Sigma_{11},
$$ 
and hence
$\fy'_{11}(I,\Sigma_{11})=\pm 1$, which contradicts the divisibility assumption.

Denote by $\Phi_{pl}^*$ the matrix obtained from $\Phi_{pl}$ via replacing the column $U^{[n-l+1]}$ by the column $U^{[n-l]}$,
and put $\tilde\fy^*_{pl}=\det \Phi^*_{pl}$. Clearly, $\fy^*_{pl}=-s_{pl}\tilde\fy^*_{pl}\det X^{t_{p+l}}$ is regular (here and in what follows we set $t_s=n-s+1$).
By the short Pl\"ucker relation for the matrix 
$$
\left[\begin{array}{ccccc}(U^0)^{[n-p+1,n]}& U^{[n-l,n]} & (U^2)^{[n]} & \dots & (U^{n-p-l+2})^{[n]}\end{array}\right]
$$ 
and columns
$I^{[n-p+1]}$, $U^{[n-l]}$, $U^{[n-l+1]}$, $(U^{n-p-l+2})^{[n]}$, one has 
\begin{equation*}
\tilde\fy_{pl}\tilde\fy^*_{p-1,l}=\tilde\fy_{pl}^*\tilde\fy_{p-1,l}+
\tilde\fy_{p,l-1}\tilde\fy_{p-1,l+1},
\end{equation*}
provided $p,l>1$ and $p+l\le n$. Multiplying the above relation by 
$\det X^{t_{p+l}+t_{p+l-1}}$ 
and $s_{p,l-1}s_{p-1,l+1}=-s_{pl}s_{p-1,l}$, one gets
\begin{equation}\label{shpl}
\fy_{pl}\fy^*_{p-1,l}=\fy_{pl}^*\fy_{p-1,l}+\fy_{p,l-1}\fy_{p-1,l+1}.
\end{equation}
A linear combination of~\eqref{shpl}
for $p=k$ and for $p=k+1$ yields
\begin{equation}\label{fyklex}
\fy_{kl} (\fy_{k+1,l}\fy^*_{k-1,l}-\fy_{k-1,l}\fy^*_{k+1,l})
=\fy_{k+1,l} \fy_{k,l-1} \fy_{k-1,l+1}+\fy_{k+1,l-1} \fy_{k,l+1} \psi_{k-1,l},
\end{equation}
which means that the function that replaces $\fy_{kl}$ after the transformation is regular whenever $k,l>1$ and $k+l<n$. 

Let us prove that $\fy_{k+1,l}\fy^*_{k-1,l}-\fy_{k-1,l}\fy^*_{k+1,l}$ is not divisible by $\fy_{kl}$. Indeed, assume the contrary, and define an $n \times n$
matrix $\Sigma_{kl}$ via
\begin{gather*}
\Sigma_{kl}=\begin{pmatrix}
0 &  0 &  0 & 0 & 0 & 0 & 1 \cr
0 &  0 &  0 & 1 & 0 & 0 & 0 \cr
0 &   0 & 0 & 0 & 0 & J_{k-2} & 0 \cr
1 &  0 &  0 & 0 & 0 & 0 & 0 \cr
0 & 0 & \one_{m} & 0 & 0 & 0 & 0 \cr
0 &  0 &  0 & 0 & 1 & 0 & 0 \cr
0 &  J_{k-1} &  0 & 0 & 0 & 0 & 0 \end{pmatrix}, \qquad \text{if  $k=l$},\\
\Sigma_{kl}=\begin{pmatrix}
0 & 0 & 0 & 0 & 0 & 0 & 0 & 0 & 1 \cr
0 & 0 & 0 & 1 & 0 & 0 & 0 &  0 & 0 \cr
0 & 0 & 0 & 0 & 0 & 0 & 0  & J_{l-2} & 0 \cr
1 & 0 & 0 & 0 & 0 & 0 & 0 &  0 & 0 \cr
0 & 0 & \one_{m} & 0 & 0 & 0 & 0  & 0 & 0 \cr
0 & 0 & 0 & 0 & 0 & 0 & 1 & 0 & 0 \cr
0 & 0 & 0 & 0 & 0 & \one_{k-l-1} &  0 & 0 & 0\cr
0 & 0 & 0 & 0 & 1 & 0 & 0 &  0 & 0 \cr
0& J_{l-1} & 0 & 0 & 0 & 0 & 0 &  0 & 0 \end{pmatrix}+e_{2,n-l},\qquad \text{if $k\ge l+1$},\\
\Sigma_{kl}=\begin{pmatrix}
0 & 0 & 0 & 0 &  0 & 1 \cr
\one_{m+1} & 0 & 0 & 0 & 0 & 0  \cr
 0 & 0 & 0 & 0 & \one_{l-2} & 0 \cr
0 & 1 & 0 & 0 & 0 & 0  \cr
0 & 0 & 0 & 1 &  0 & 0  \cr
0 & 0 & \one_{k-1} & 0 & 0 & 0 \end{pmatrix}+e_{m+2,n-l},\qquad \text{if $k< l$},
\end{gather*} 
where $m=n-k-l-1$ and $J_r$ is the $r\times r$ unitary antidiagonal matrix. An explicit computation shows that $\det\Sigma_{kl}=\pm1$, and
\begin{gather*}
\fy_{k-1,l}(I,\Sigma_{kl})=\pm1,\qquad \fy_{k+1,l}^*(I,\Sigma_{kl})=\pm1,\\
\fy_{k+1,l}(I,\Sigma_{kl})=\fy_{k-1,l}^*(I,\Sigma_{kl})=\fy_{kl}(I,\Sigma_{kl})=0,
\end{gather*}
which contradicts the divisibility assumption.

Extend the definition of $\Phi_{pq}$ and $s_{pq}$ to the case $p=0$, denote by $\Phi^{**}_{pq}$ the matrix obtained from $\Phi_{pq}$ by
deleting the last column and inserting $(U^2)^{[n-1]}$ as the $(p+q+1)$-st column, and put $\tilde\fy^{**}_{pq}=\det\Phi^{**}_{pq}$. 
Clearly, $\fy^{**}_{pq}=s_{q,p+2}\tilde\fy^{**}_{pq}\det X^{t_{p+q}-1}$ is regular.
It is easy to see that $U\Phi_{p1}=\Phi_{0p}$ and $U\Phi_{p2}=\Phi^{**}_{0p}$. Therefore, by the short Pl\"ucker relation
for the matrix 
$$
\left[\begin{array}{cccccc} I^{[n]} & U^{[n-k+1,n]} & (U^2)^{[n-1,n]} & (U^3)^{[n]} \dots & (U^{n-k+1})^{[n]}\end{array}\right]
$$ 
and columns
$I^{[n]}$, $U^{[n-k+1]}$, $(U^2)^{[n-1]}$, $(U^{n-k+1})^{[n]}$, one has 
\begin{equation*}
\tilde\fy_{0k}\tilde\fy^{**}_{1,k-1}=\tilde\fy^{**}_{0,k-1}\tilde\fy_{1k}+\tilde\fy^{**}_{0k}\tilde\fy_{1,k-1}.
\end{equation*}
Multiplying the above relation by $\det X^{2t_{k}-1}=\det X^{t_{k-1}+t_{k+1}-1}$ and $s_{k2}s_{1,k-1}=s_{k-1,2}s_{1k}=s_{0k}s_{k-1,3}$,
taking into account that $\fy_{0p}=h_{11}\fy_{p1}$, $\fy^{**}_{0p}=h_{11}\fy_{p2}$, $s_{k1}=(-1)^ns_{0k}$ and dividing the above relation by $h_{11}$ we arrive at
\begin{equation}\label{k1neigh}
(-1)^n\fy_{k1}\fy^{**}_{1,k-1}=\fy_{k-1,2}\fy_{1k}+\fy_{k2}\fy_{1,k-1},
\end{equation}
which means that the function that replaces $\fy_{k1}$ after the transformation is regular whenever $1<k<n-1$.

To prove that $\fy^{**}_{1,k-1}$ is not divisible by $\fy_{k1}$, we assume the contrary and define an $n \times n$ matrix $\Sigma_{k1}$ via
\begin{gather*}
\Sigma_{k1}=\begin{pmatrix}
0 &  0 &  0 & 1 \cr
\one_{n-4} & 0 & 0 & 0 \cr
0 &   0 & 1 & 0  \cr 
0 & \one_{2} & 0 & 0  \end{pmatrix}+e_{n-1,n-l}, \qquad \text{if  $k=2$},\\
\Sigma_{k1}=\begin{pmatrix}
 0 &  0 & 0 & 0 & 0 & 1 \cr
\one_{m}&  0 & 0 & 0 &  0 & 0 \cr
 0 &  0 & 1 & 0 &  0 & 0 \cr
 0 &  0 & 0 & 0 & 1 & 0 \cr
0& 0 & 0 & \one_{k-3} &  0 & 0 \cr
 0 & \one_2 &  0 & 0 &  0 & 0  \end{pmatrix},\qquad \text{if $k\ge 3$},
\end{gather*} 
where $m=n-k-2$. An explicit computation shows that $\det\Sigma_{k1}=\pm1$, and
\[
\fy^{**}_{1,k-1}(I,\Sigma_{k1})=\pm1,\qquad \fy_{kl}(I,\Sigma_{k1})=0,
\]
which contradicts the divisibility assumption.

With the extended definition of $\Phi_{pq}$, relation~\eqref{shpl} becomes valid for $p=1$. Taking a linear combination of~\eqref{shpl} for $p=1$ and $p=2$
 one gets
\begin{equation*}
\fy_{1l}(\fy_{2l}\fy^*_{0l}-\fy_{0l}\fy^*_{2l})=\fy_{2l}\fy_{1,l-1}\fy_{0,l+1}+
\fy_{2,l-1}\fy_{1,l+1}\fy_{0l}.
\end{equation*}
Recall that $\fy_{0p}=h_{11}\fy_{p1}$. Besides, $\Phi^*_{0l}=U\Phi^\circ_{l1}$, where $\Phi_{pq}^\circ$ is the matrix obtained from $\Phi_{pq}$ 
via replacing the column $I^{[n-p+1]}$ by the column $I^{[n-p]}$. Denote $\tilde\fy^\circ_{pq}=\det\Phi^\circ_{pq}$;
clearly, $\fy^\circ_{pq}=-s_{q-1,p}\tilde\fy^\circ_{pq}\det X^{t_{p+q}}$ is regular. We thus have $\tilde\fy^*_{0l}=\det U\tilde\fy^\circ_{l1}$, and hence 
$\fy^*_{0l}=h_{11}\fy^\circ_{l1}$.  Consequently,
\begin{equation}\label{fy1lex}
\fy_{1l}(\fy_{2l}\fy^\circ_{l1}-\fy_{l1}\fy^*_{2l})=\fy_{2l}\fy_{1,l-1}\fy_{l+1,1}+
\fy_{2,l-1}\fy_{1,l+1}\fy_{l1},
\end{equation}
which means that the function that replaces $\fy_{1l}$ after the transformation is regular whenever $1<l<n-1$.

To prove that $\fy_{2l}\fy^\circ_{l1}-\fy_{l1}\fy^*_{2l}$ is not divisible by $\fy_{1l}$, we assume the contrary and define an $n \times n$ matrix $\Sigma_{1l}$ via
\[
\Sigma_{1l}=\begin{pmatrix}
0 &  0 &  0 & 1 \cr
\one_{m} & 0 & 0 & 0 \cr
0 &   0 & \one_{l-2} & 0  \cr 
0 & \one_{2} & 0 & 0  \end{pmatrix}+e_{n-l,n-l},
\]
where $m=n-l-1$.  An explicit computation shows that $\det\Sigma_{1l}=\pm1$, and
\begin{gather*}
\fy_{l1}(I,\Sigma_{1l})=\pm1,\qquad \fy_{2l}^*(I,\Sigma_{1l})=\pm1,\\
\fy_{2l}(I,\Sigma_{1l})=\fy_{l1}^\circ(I,\Sigma_{1l})=\fy_{1l}(I,\Sigma_{1l})=0,
\end{gather*}
which contradicts the divisibility assumption.

Define an $n\times n$ matrix $F^\lozenge_{k,n-k-1}$, $0\le k\le n-2$, by
$$
F^\lozenge_{k,n-k-1}=\left[\begin{array}{ccc}I^{[1]} & X^{[n-k+1,n]} & Y^{[k+2,n]}\end{array}\right].
$$ 
Clearly, $f_{k,n-k-1}=\det F^\lozenge_{k,n-k-1}$. Besides, denote by $\Phi^\lozenge_{k,n-k-1}$ the matrix obtained from $X\Phi_{k,n-k-1}$ via replacing
the column $X^{[n-k+1]}$ by the column $I^{[1]}$. Denote $\tilde\fy^\lozenge_{k,n-k-1}=\det\Phi^\lozenge_{k,n-k-1}$. Clearly, 
$\fy^\lozenge_{k,n-k-1}=s_{k,n-k-1}\tilde\fy^\lozenge_{k,n-k-1}\det X$ is regular.
Therefore, the short Pl\"ucker relation for the matrix 
$$
\left[\begin{array}{cccc}I^{[1]}& X^{[n-k+1,n]} & Y^{[k+1,n]} &  (YX^{-1}Y)^{[n]}\end{array}\right]
$$ 
and columns
$I^{[1]}$, $X^{[n-k+1]}$, $Y^{[k+1]}$, $(YX^{-1}Y)^{[n]}$ involves submatrices 
$X\Phi_{k,n-k}$, $X\Phi_{k,n-k-1}$, $X\Phi_{k-1,n-k}$, $\Phi^\lozenge_{k,n-k-1}$, $F^\lozenge_{k,n-k-1}$, and $F^\lozenge_{k-1,n-k}$ and  gives
\begin{equation*}
\det X\tilde\fy_{k,n-k}\tilde\fy^\lozenge_{k,n-k-1}=-\det X\tilde\fy_{k-1,n-k}f_{k,n-k-1}+\det X\tilde\fy_{k,n-k-1}f_{k-1,n-k}.
\end{equation*}
Multiplying by $\det X$ and $s_{k,n-k-1}=-s_{k-1,n-k}$ and taking into account that $s_{k,n-k}=1$, we arrive at
\begin{equation*}
\fy_{k,n-k}\fy^\lozenge_{k,n-k-1}=\fy_{k-1,n-k}f_{k,n-k-1}+\fy_{k,n-k-1}f_{k-1,n-k},
\end{equation*}
which together with the description of the quiver $Q_n$ means that the function that replaces $\fy_{k,n-k}$ after the transformation is regular whenever $1<k<n-1$. 
For $k=1$ the latter relation is modified taking into account that $f_{0,n-1}=h_{22}$ and $\fy_{0,n-1}=s_{0,n-1}s_{n-1,1}h_{11}\fy_{n-1,1}$,
which together with $s_{0,n-1}=s_{n-1,1}=1$ gives
\begin{equation*}
\fy_{1,n-1}\fy^\lozenge_{1,n-2}=h_{11}\fy_{n-1,1}f_{1,n-2}+\fy_{1,n-2}h_{22}.
\end{equation*}
This means that the function that replaces $\fy_{1,n-1}$ after the transformation is regular.

To prove that $\fy^{\lozenge}_{k,n-k-1}$ is not divisible by $\fy_{k,n-k}$ is enough to note
that $\fy^{\lozenge}_{k,n-k-1}$ is irreducible as a minor of a matrix whose entries are independent variables.

Define two $n \times n$ matrices $\Phi^\square_{2,n-2}$ and $F^\square_{n-2,1}$ via replacing the first column of $\Phi_{2,n-2}$ by $U^{[1]}$ and the first
column of $F^\lozenge_{n-2,1}$ by $X^{[1]}$. Clearly, $f^\square_{n-2,1}=\det F^\square_{n-2,1}$ is regular; besides, denote
$\tilde\fy^\square_{2,n-2}=\det\Phi^\square_{2,n-2}$, then
$\fy^\square_{2,n-2}=-s_{2,n-2}\tilde\fy^\square_{2,n-2}\det X$ is regular. The short Pl\"ucker relation for the matrix 
$$
\left[\begin{array}{cccc}U^{[1]}& I^{[n]} & U^{[2,n]} &  (U^2)^{[n]}\end{array}\right]
$$ 
and columns
$U^{[1]}$, $I^{[n]}$, $U^{[2]}$, $(U^2)^{[n]}$ involves submatrices 
$U\Phi_{n-1,1}$, $\Phi_{1,n-2}$,
$\Phi_{1,n-1}$, $UX^{-1}F^\square_{n-2,1}$, $\Phi^\square_{2,n-2}$, and $U$ and gives
\begin{equation*}
\det U\tilde\fy_{n-1,1}\tilde\fy^\square_{2,n-2}=\det U\tilde\fy_{1,n-2}-\det U\tilde\fy_{1,n-1}f^\square_{n-2,1}.
\end{equation*}
Multiplying the above relation by $\det X^2$ and $s_{1,n-2}=-s_{1,n-1}$, dividing by $\det U$ and taking into account that $s_{k,n-k}=1$ for $0\le k \le n$ one gets
\begin{equation}\label{reg6}
\fy_{n-1,1}\fy^\square_{2,n-2}=\fy_{1,n-2}+\fy_{1,n-1}f^\square_{n-2,1},
\end{equation}
which means that the function that replaces $\fy_{n-1,1}$ after the transformation is regular.

To prove that $\fy^{\square}_{2,n-2}$ is not divisible by $\fy_{n-1,1}$ is enough to note that $\fy_{n-1,1}(I,Y)=y_{1n}$ and
$\fy^{\square}_{2,n-2}(I,Y)$ is a $(n-1)\times(n-1)$ minor of $Y$.

Define a $(p+q+1)\times (p+q+1)$ matrix $F^*_{pq}$   by
$$
F^*_{pq}=\left[\begin{array}{ccc}I^{[n-p-q+1]} & X^{[n-p+1,n]} & Y^{[n-q+1,n]}\end{array}\right]_{[n-p-q,n]},
$$ 
and put $f^*_{pq}=\det F^*_{pq}$.
 The short Pl\"ucker relation for the matrix 
$$
\left[\begin{array}{ccc}I^{[n-p-q,n-p-q+1]}& X^{[n-p+1,n]} & Y^{[n-q,n]}\end{array}\right]_{[n-p-q,n]}
$$ 
and columns
$I^{[n-p-q]}$, $I^{[n-p-q+1]}$, $X^{[n-p+1]}$, $Y^{[n-q]}$
gives
\begin{equation*}
f_{pq}f^*_{p-1,q+1}=f_{p-1,q+1}f^*_{pq}+f_{p-1,q}f_{p,q+1}
\end{equation*}
for $p+q<n-1$, $p>0$, $q>0$.
Applying this relation for $(p,q)=(i,j)$ and $(p,q)=(i+1,j-1)$ we get
\begin{equation*}
f_{ij}(f_{i+1,j-1}f^*_{i-1,j+1}-f_{i-1,j+1}f^*_{i+1,j-1})=f_{i+1,j-1}f_{i-1,j}f_{i,j+1}+f_{i,j-1}f_{i+1,j}f_{i-1,j+1},
\end{equation*}
which 
means that the function that replaces $f_{ij}$ after the transformation is regular whenever $i+j<n-1$, $i>0$, $j>0$.
To extend the above relation to the case $i+j=n-1$ is enough to recall that  
$f_{k,n-k}=\fy_{k,n-k}$ by the identification of Remark~\ref{identify}.

To prove that $f'_{ij}=f_{i+1,j-1}f^*_{i-1,j+1}-f_{i-1,j+1}f^*_{i+1,j-1}$ is not divisible by $f_{ij}$ is enough to show
that $f'_{ij}$ is irreducible. Indeed, $f'_{ij}$ can be written as
\[
f'_{ij}=y_{n-i-j,n-j}f_{i+1,j-1}f_{i-1,j}-x_{n-i-j,n-i}f_{i-1,j+1}f_{i,j-1}+R(X,Y),
\]
where $R$ does not depend on $y_{n-i-j,n-j}$ and $x_{n-i-j,n-i}$. Moreover, $f_{i+1,j-1}$, $f_{i-1,j}$, $f_{i-1,j+1}$, and $f_{i,j-1}$
do not depend on these two variables as well. Consequently, reducibility of $f'_{ij}$ would imply
\[
f'_{ij}=\left(y_{n-i-j,n-j}P(X,Y)+x_{n-i-j,n-i}Q(X,Y)+R'(X,Y)\right)R''(X,Y),
\]
which contradicts the irreducibility of $f_{i+1,j-1}$, $f_{i-1,j}$, $f_{i-1,j+1}$, $f_{i,j-1}$.

Define an $(n-i+2)\times (n-i+2)$ matrix $H^\star_{ii}$ by
$$
H^\star_{ii}=\left[\begin{array}{ccc} X^{[n]} & Y^{[i+1,n]} & I^{[n]}\end{array}\right]_{[i-1,n]}
$$ 
and denote $h^\star_{ii}=\det H^\star_{ii}$. Clearly,
\begin{equation}\label{hii}
h_{ii}^\star = \det \left [\begin{array}{cc} X^{[n]} & Y^{[i+1,n]}\end{array} \right ]_{[i-1,n-1]};
\end{equation}
in particular, $h_{nn}^\star=x_{n-1,n}$. The short Pl\"ucker relation for the matrix
$$
\left[\begin{array}{cccc} I^{[i-1]} & X^{[n]} & Y^{[i,n]} & I^{[n]}\end{array}\right]_{[i-1,n]}
$$
and columns $I^{[i-1]}$, $X^{[n]}$, $Y^{[i]}$ and $I^{[n]}$ gives
\begin{equation*}
h_{ii}h^\star_{ii}=f_{1,n-i}h_{i-1,i}+f_{1,n-i+1}h_{i,i+1},
\end{equation*}
which means that the function that replaces $h_{ii}$ after the transformation is regular whenever $2\le i\le n$. Note that for $i=n$ we use the convention $h_{n,n+1}=1$, which was already used above.

To prove that $h^\star_{ii}$ is not divisible by $h_{ii}$ is enough to note
that $h^\star_{ii}$ is irreducible as a minor of a matrix whose entries are independent variables.

Define an $(n-i+2)\times (n-i+2)$ matrix $F^\circ_{n-i+1,1}$ and an $(n-i+1)\times (n-i+1)$ matrix $G^\circ_{ii}$
via replacing the column $X^{[i]}$ with the column $X^{[i-1]}$ in $F_{n-i+1,1}$ and $G_{ii}$, respectively, and denote
$f^\circ_{n-i+1,1}=\det F^\circ_{n-i+1,1}$, $g^\circ_{ii}=\det G^\circ_{ii}$. The short Pl\"ucker relation for the matrix
$$
\left[\begin{array}{ccc} I^{[i-1]} & X^{[i-1,n]} & Y^{[n]}\end{array}\right]_{[i-1,n]}
$$
and columns $I^{[i-1]}$, $X^{[i-1]}$, $X^{[i]}$, and $Y^{[n]}$ gives
\begin{equation}\label{pluone}
g_{ii}f^\circ_{n-i+1,1}=f_{n-i,1}g_{i-1,i-1}+f_{n-i+1,1}g^\circ_{ii}
\end{equation}
for $2\le i\le n$. Taking into account that $g^\circ_{nn}=g_{n,n-1}$ and the description of the quiver $Q_n$, we see that 
the function that replaces $g_{nn}$ after the transformation is regular.

To prove that $f^\circ_{n-i+1,1}$ is not divisible by $g_{ii}$ is enough to note
that $f^\circ_{n-i+1,1}$ is irreducible as a minor of a matrix whose entries are independent variables.

Next, define an $(n-i)\times (n-i)$ matrix $G^\circ_{i+1,i}$ via replacing the column $X^{[i]}$ with the column $X^{[i-1]}$ in
$G_{i+1,i}$, and denote $g^\circ_{i+1,i}=\det G^\circ_{i+1,i}$. The short Pl\"ucker relation for the matrix
$$
\left[\begin{array}{cc} I^{[i-1,i]} & X^{[i-1,n]}\end{array}\right]_{[i-1,n]}
$$
and columns $I^{[i]}$, $X^{[i-1]}$, $X^{[i]}$, and $X^{[n]}$ gives
\begin{equation}\label{plutwo}
g_{i+1,i}g^\circ_{ii}=g_{i,i-1}g_{i+1,i+1}+g_{ii}g^\circ_{i+1,i}
\end{equation}
for $2\le i\le n-1$.
Combining relations~\eqref{pluone} and~\eqref{plutwo} one gets
\begin{equation*}
g_{ii}(g_{i+1,i}f^\circ_{n-i+1,1}-g^\circ_{i+1,i}f_{n-i+1,1})=f_{n-i,1}g_{i-1,i-1}g_{i+1,i}+f_{n-i+1,1}g_{i,i-1}g_{i+1,i+1},
\end{equation*}
which 
means that the function that replaces $g_{ii}$ after the transformation is regular whenever $2\le i\le n-1$.

To prove that $g_{ii}'=g_{i+1,i}f^\circ_{n-i+1,1}-g^\circ_{i+1,i}f_{n-i+1,1}$ is not divisible by $g_{ii}$, define $X^\circ$ to be the lower bidiagonal matrix with
$x^\circ_{ii}=t$ and all other entries in the two diagonals equal to one. Then  $g_{i+1,i}(X^\circ)=1$, $f^\circ_{n-i+1,1}(X^\circ,Y)=\pm(y_{in}-y_{i-1,n})$,
$g^\circ_{i+1,i}(X^\circ)=0$, and hence $g_{ii}'(X^\circ,Y)=\pm(y_{in}-y_{i-1,n})$ is not divisible by $g_{ii}(X^\circ)=t$.

The rest of the vertices $g_{ij}$ and $h_{ij}$ do not need a separate treatment since the corresponding relations coincide with those for the standard cluster 
structure in $GL_n$.

\section{Proof of Theorem~\ref{structure}(ii)}\label{upeqreg}

\subsection{Proof of Theorem~\ref{allxcluster}}
Functions $x_{ni}$, $1\le i\le n$, are in the initial cluster. Our goal is to explicitly construct a sequence of cluster
transformations that will allow us to recover all $x_{ij}$ as cluster variables.
For this, we will only need to work with a subquiver $\Gamma_n^n$ of $Q_n$ whose vertices belong to lower $n$ levels of 
$Q_n$ and in which we view the vertices in the top row
as frozen (see Fig.~\ref{Gamma55} for the quiver $\Gamma_5^5$).

\begin{figure}[ht]
\begin{center}
\includegraphics[width=12cm]{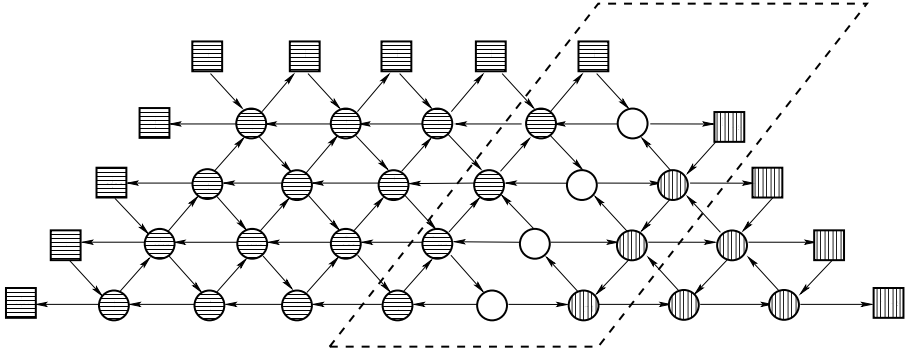}
\caption{Quiver $\Gamma_5^5$}
\label{Gamma55}
\end{center}
\end{figure}

One can distinguish two oriented triangular grid subquivers of $\Gamma_n^n$: a ``square'' one on $n^2$ vertices with the clockwise 
orientation of triangles in the lowest row (dashed horizontally on Fig.~\ref{Gamma55}), and 
a ``triangular'' one  on $n(n-1)/2$ vertices with the counterclockwise orientation of triangles in the lowest row (dashed vertically). 
They are glued together with the help of the quiver $\Gamma_{n}$ on $3n-2$ vertices placed in three columns of size $n$, $n-1$, and $n-1$. 
The left column of $\Gamma_{n}$ is identified with the rightmost side of the square subquiver, and the right column, with the leftmost side of the triangular subquiver, see Fig.~\ref{sub2}a) for the case $n=5$. More generally, we can define a quiver $\Gamma_m^n$, $m\le n$, by using $\Gamma_{m}$
to glue an oriented triangular grid quiver on $m n$ vertices forming a parallelogram with a base of length $n$ and the side of length $m$  with an 
oriented triangular grid quiver on  $m(m-1)/2$ vertices forming a triangle with sides $m-1$ (the orientations of the triangles in the lowest rows of both
grids obeys the same rule as above).

\begin{figure}[ht]
\begin{center}
\includegraphics[height=5cm]{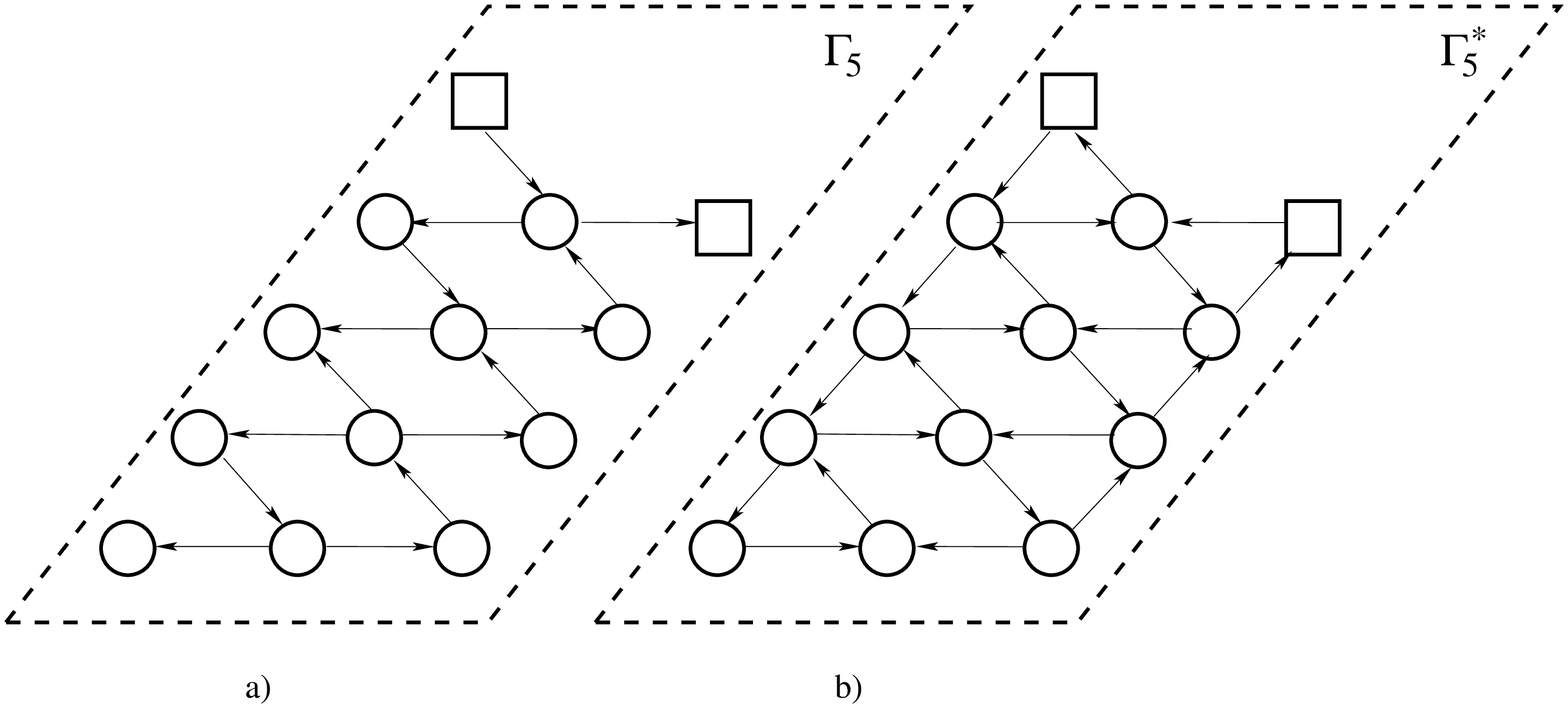}
\caption{Quivers $\Gamma_5$ and $\Gamma^*_5$}
\label{sub2}
\end{center}
\end{figure}

Let $A$ be an $m\times m$ matrix; for $1\le i\le m-1$, denote
$$
a_{m-i}=\det A_{[i+1,m]}^{[i+1,m]}, \qquad b_{m-i}=\det A_{[i,m-1]}^{[i+1,m]}, \qquad c_{m-i}=\det A_{[i+1,m]}^{[1]\cup [i+2,m]}.
$$
It is easy to check that $a_{m-i}$, $b_{m-i}$ and $c_{m-i}$ are maximal minors in the $(m-i+1)\times (m-i+3)$ matrix 
$$
\left[\begin{array}{ccc} I^{[i]} & A^{[1]\cup [i+1,m]} & I^{[m]} \end{array}\right ]_{[i,m]},
$$
obtained by removing columns $A^{[1]}$ and $I^{[m]}$, $I^{[i]}$ and $A^{[1]}$, $A^{[i+1]}$ and $I^{[m]}$, respectively. Moreover, the
maximal minor obtained by removing columns $A^{[1]}$ and $A^{[i+1]}$ equals $b_{m-i-1}$ (assuming $b_0=1$), and the one obtained
by removing columns $I^{[i]}$ and $I^{[m]}$ equals $c_{m-i+1}$ (assuming $c_m=\det A$).
Consequently, the short Pl\"ucker relation for the matrix above and columns $ I^{[i]}$, $A^{[1]}$, $A^{[i+1]}$, $I^{[m]}$ reads
$$
a_i a^*_i = b_i c_i + b_{i-1} c_{i+1},
$$
where $a^*_{m-i} =\det A^{[1]\cup [i+2,m]}_{[i,m-1]}$. Let us assign variables $a_i$ to the vertices in the central column of $\Gamma_m$ bottom up, 
variables $b_i$ to the vertices in the right column, and variables $c_i$ to the vertices in the left column. It follows from the above discussion that
applying commuting cluster transformations and the corresponding quiver mutations to vertices in the central column of $\Gamma_m$  
results in the quiver $\Gamma^*_m$ shown in Fig.~\ref{sub2}b) for $m=5$. The variables attached to the vertices of the central column bottom up are
$a_i^*$ defined above.

Denote  $X_m=X_{[1,m]}$ and $Y_m=Y_{[1,m]}^{[n-m+1,n]}$. Extending the definition of functions
$g_{ij}(X)$ in Section~\ref{logcan} to rectangular matrices, we write
$$
g_{ij}^m=g_{ij}(X_m) = \det X_{[i,m]}^{[j,j+m-i]}
$$
for $1\leq j \leq i + n -m$. Functions $f_{kl}^m=f_{kl}(X_m,Y_m)$ and $h_{ij}^m=h_{ij}(Y_m)$ are defined exactly 
as in Section~\ref{logcan}. Consequently,
\begin{equation}\label{hshift}
h_{ij}^m=h_{i+1,j}^{m-1} \qquad \text{for $i<j$}.
\end{equation}

Going back to the quiver $\Gamma_m^n$, let us
denote by $\widetilde\Sigma_m^n$ the extended seed that consists of $\Gamma_m^n$ and 
the following family of functions attached to its vertices: the functions attached to the $s$th row listed
left to right are $g_{m-s+1,1}^m,\dots, g_{m-s+1,n-s+1}^m$ followed by $f_{s-1,1}^m,\dots,f_{1,s-1}^m$
followed by $h_{m-s+1,m-s+1}^m,\dots$, $h_{1,m-s+1}^m$ (except for the top row that does not contain $h_{11}^m$).
Now, consider the sequence 
of $m-1$ commuting  mutations of $\Gamma_m^n$ at vertices $h_{ii}^m$, $2\le i\le m$. As  in~\eqref{hii}, one sees that the new variables associated
with these vertices are $h_{ii}^\star = \det \left [ X^{[n]} \;\; Y^{[i+1,m]} \right ]_{[i-1,m-1]}$. In particular, $h^\star_{mm}=x_{m-1,n}$, and we thus have 
obtained $x_{m-1,n}$ as a cluster variable. Moreover, $h^\star_{mm}=g_{m-1,n}^{m-1}$, and $h^\star_{ii}=f_{1,m-i}^{m-1}$ for $2\le i\le m-1$.

The cumulative effect of this sequence of transformations can be summarized as follows: 

(i) detach from the quiver $\Gamma_m^n$ a quiver isomorphic to $\Gamma_{m}$ with $h_{ii}^m$, $i=m, \dots, 2$, playing the role of $a_{i}$, $i=1, \dots, m-1$, 
and note that functions assigned to its vertices are of the form described above if one selects $A= \left [ X_m^{[n]}\;\; Y_m^{[2,m]} \right ]$; 

(ii) apply cluster transformations to vertices $a_{i}$, $1\le i\le m-1$, of $\Gamma_{m}$; 

(iii) glue the resulting quiver $\Gamma^*_{m}$ back into $\Gamma_m^n$ and erase any two-cycles that may have been created in this process;

(iv) note that the new variables attached to the mutated vertices are $g_{m-1,n}^{m-1}$, $f_{1,m-i}^{m-1}$, $i=m-1,\dots, 2$.

The resulting quiver contains another copy of $\Gamma_{m}$, shifted leftwards by~1, with vertices 
$g_{mn}^m$, $f_{11}^m,\ldots, f_{1,m-2}^m$ 
playing the role of $a_{i}$, $i=1, \dots, m-1$, and the matrix $\left [ X_m^{[n-1,n]}\;\; Y_m^{[3,m]} \right ]$ playing the role of $A$. 
Therefore, we can repeat the procedure used on the previous step
to obtain a new quiver in which $g_{mn}^m$, $f_{11}^m,\ldots, f_{1,m-2}^m$ are replaced
by $x_{m-1,n-1}$,  $\det \left [ X^{[n-1,n]}\;\; Y^{[i+2,m]} \right ]_{[i-1,m-1]}$, $i= m-1, \ldots, 2$, respectively. Thus, we have obtained
$x_{m-1,n-1}$ as a cluster variable, and, moreover, the new variables attached to the mutated vertices are $g_{m-1,n-1}^{m-1}$, $g_{m-2,n-1}^{m-1}$,
$f_{2,m-i}^{m-1}$ for $i=m-1,\dots, 3$.

We proceed in the same way $n-2$ more times. At the $j$th stage, $1\le j\le n-1$, the copy of $\Gamma_m$ is shifted leftwards by $j$, 
$A= \left [ X_m^{[n-j,n]}\;\; Y_m^{[j+2,m]} \right ]$, the role of $a_{i}$, $i=1, \dots, m-1$, 
is played by vertices associated with the functions $g_{m,n-j+1}^m,\ldots{}, g_{m-j+1, n-j+1}^m$, $f_{j1}^m,\ldots, f_{j,m-j}^m$, which are being replaced with  
\begin{equation*}
\begin{split}
&\det X_{[m-i,m-1]}^{[n-j,n-j+i-1]},\quad i=1,\dots, j+1,\\  
&\det \left [ X^{[n-j,n]}\;\; Y^{[j+i+1,m]} \right ]_{[i-1,m-1]}, \quad  i= m-j, \ldots, 2. 
\end{split}
\end{equation*}
Note that
the first of the functions listed above is $x_{m-1,n-j}$, so in the end of this process, we have restored all the entries of the $(m-1)$-st row of $X$. Moreover,
the new variables are $g_{m-i,n-j}^{m-1}$, $i=1,\dots, j+1$, $f_{j+1,m-i}^{m-1}$, $i=m-1,\dots, j+2$.

Let us freeze in the resulting  quiver 
$\tilde\Gamma_m^n$ 
all $g$-vertices and $f$-vertices adjacent to the frozen vertices. It is easy to check that the quiver obtained in this way is isomorphic to $\Gamma_{m-1}^n$.
Moreover, the above discussion together with the identity~\eqref{hshift} shows that the functions assigned to its vertices are exactly those stipulated by
the definition of the extended seed $\widetilde\Sigma_{m-1}^n$. Thus we establish the claim of the theorem by applying the procedure described above 
consecutively to  $\widetilde\Sigma_{n}^n$, $\widetilde\Sigma_{n-1}^n, \ldots, \widetilde\Sigma_{2}^n$.

\subsection{Sequence $\TE$: the proof of Theorem~\ref{clusterU}}

Consider the subquiver $\widehat\Gamma_n^n$ of $\Gamma_n^n$ obtained by freezing the vertices corresponding to functions 
$h_{ii}(Y)$,  $2\le i\le n$, and ignoring vertices to the right of them. In other words, $\widehat\Gamma_n^n$ is the subquiver of $Q_n$
induced by all $g$-vertices, all $f$-vertices, $\fy$-vertices with $k+l=n$, and $h$-vertices with $i=j\ge 2$.
The quiver $\widehat\Gamma_5^5$
is shown in Fig.~\ref{sub4}; the vertices that are frozen in $\widehat\Gamma_5^5$, but are mutable in $Q_5$ are shown by rounded squares. 
Note the special edge shown by the dashed line. It does not exist in $\widehat\Gamma_n^n$ (since it connects frozen vertices), 
but it exists in $Q_n$.

Within this proof we label the vertices of $\widehat\Gamma_n^n$ by pair of indices $(i,j)$, $1\le i\le n$, $1\le j\le n+1$, $(i,j)\ne (1, n+1)$, where $i$ 
increases from top to bottom and  $j$ increases from left to right; thus, the special edge is $(1,n)\to (2,n+1)$.
The set of vertices with $j-i=l$ forms the $l$th diagonal in $\widehat\Gamma_n^n$, 
$1-n\le l \le n-1$. The function attached to the vertex $(i,j)$ is
$$
\chi_{ij} = \begin{cases} g_{ij}(X) = \det X_{[i,n]}^{ [j, n+j-i]}& \mbox{if}\quad  i \geq j, \\
f_{n-j+1,j-i}(X,Y)= \det \left [ X^{ [j,n]}\;\; Y^{[n+i-j+1,n]} \right ]_{[i,n]} & \mbox{if} \quad i < j.  \end{cases}
$$
We denote the extended seed thus obtained from $\widetilde{\Sigma}_n^n$ by $\widehat{\Sigma}_n^n$. Note that it is a seed 
of an ordinary cluster structure, since no generalized exchange relations are involved.

\begin{figure}[ht]
\begin{center}
\includegraphics[width=9cm]{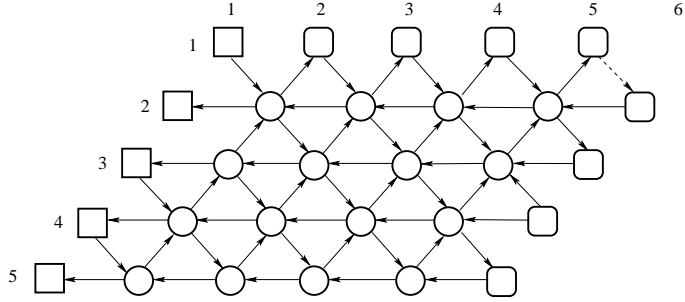}
\caption{Quiver $\widehat\Gamma_5^5$}
\label{sub4}
\end{center}
\end{figure}

Consider a sequence of mutations $\TE_n$ which involves mutating once at every non-frozen vertex of $\widehat\Gamma_n^n$ starting with $(n,2)$ then
using vertices of the $(3-n)$-th, $(4-n)$-th, ..., $(n-3)$-rd, $(n-2)$-nd diagonals. Note that a similar sequence of transformations was used in the proof 
of Proposition 4.15 in \cite{GSVb}
in the study of the natural cluster structure on Grassmannians. The order in which vertices of each diagonal are mutated 
is not important, since at the moment a diagonal is reached in this sequence of transformations, 
there are no edges between its vertices. In fact, functions $\chi_{ij}^1$  obtained as a result of applying $\TE_n$ are subject to relations
$$
\chi_{ij}^1 \chi_{ij} = \chi^1_{i,j-1} \chi_{i,j+1} + \chi^1_{i+1,j} \chi_{i-1,j},\qquad 2\le i,j \le n,
$$
where we adopt a convention $\chi^1_{n+1,j} = x_{n1}$, $\chi^1_{i1} = \chi_{i-1,1}=g_{i-1,1}(X)$. 
These relations imply 
\begin{equation}
\label{chistar}
\chi^1_{ij}= \begin{cases}  \det X_{[i-1,n]}^{[1] \cup [j+1, n+j-i+1]} & \mbox{if}\quad  i > j, \\
\det \left [ X^{[1]\cup [j+1,n]}\;\; Y^{[n+i-j,n]} \right ]_{[i-1,n]} & \mbox{if} \quad i \leq j.  \end{cases}
\end{equation}
To verify~\eqref{chistar} for $i>j$, one has to apply the short Pl\"ucker relation to 
$$
\left[\begin{array}{cc} I^{[i-1]} & X^{[1]\cup [j,n+j-i+1]} \end{array}\right ]_{[i-1,n]}
$$
using columns 
$I^{[i-1]}$, $X^{[1]}$, $X^{[j]}$, $X^{[n+j-i+1]}$. In the case $i\leq j$, the short Pl\"ucker relation is applied to 
$$
\left[\begin{array}{ccc} I^{[i-1]} & X^{[1]\cup [j,n]} & Y^{[n+i-j,n]} \end{array}\right ]_{[i-1,n]}
$$
using columns 
$I^{[i-1]}$, $X^{[1]}$, $X^{[j]}$, $Y^{[n+i-j]} $.
Note that 
\begin{equation*}
\begin{aligned}
\chi^1_{2j} &= \det \left [ X^{[1]\cup [j+1,n]}\;\; Y^{[n-j+2,n]} \right ]_{[1,n]} 
= \det X \cdot(-1)^{(n-j)(j-1)}\det U_{[2,j]}^{[n-j+2,n]}\\ &= \det X\cdot(-1)^{(n-j)(n-2)}h_{2,n-j+2}(U), \qquad 2\le j\le n.
\end{aligned}
\end{equation*}

The subquiver of $\TE_n(\widehat\Gamma_n^n)$ formed by non-frozen vertices is isomorphic to the corresponding 
subquiver of $\widehat\Gamma_n^n$. However, the frozen vertices are connected to non-frozen ones in a different way now:
there are arrows $(i,1)\to(i+2, 2)$ and $(i+1,n+1)\to(i+2, n)$ for $1\le i\le n-2$, $(i,2)\to(i-1, 1)$ and $(i,n)\to(i, n+1)$ 
for $2\le i\le n$, $(1,j)\to(2, j)$ for $2\le j\le n$, $(n,1)\to (n,n)$, and 
  $(2,j)\to(1, j+1)$ for $2\le j\le n-1$. 	After moving frozen vertices we can make 
 $\TE_n(\widehat\Gamma_n^n)$ look as shown in Fig.~\ref{sub5}. 

\begin{figure}[ht]
\begin{center}
\includegraphics[width=8cm]{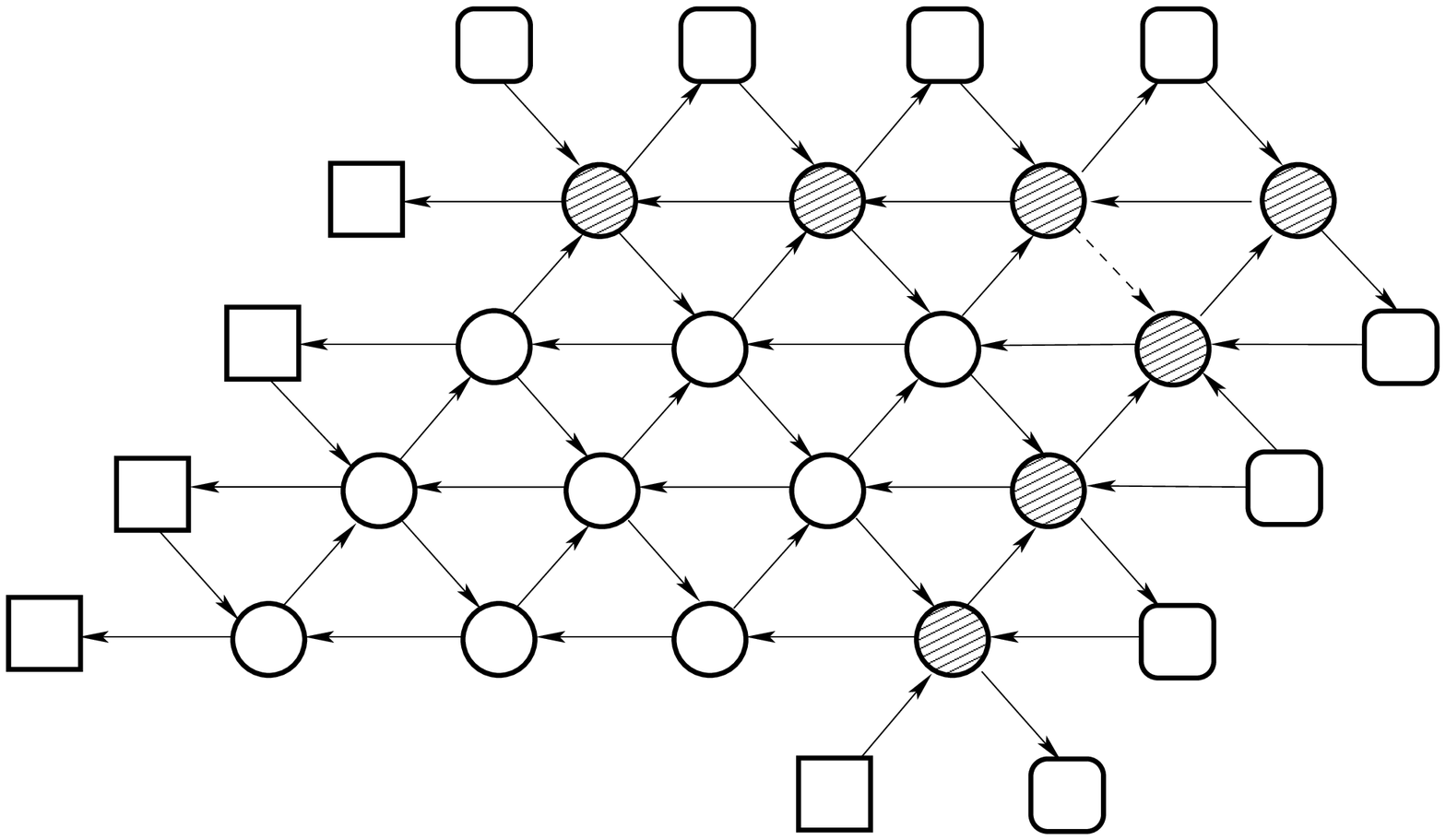}
\caption{Quiver $\TE_n(\widehat\Gamma_5^5)$}
\label{sub5}
\end{center}
\end{figure}

Note that if we freeze the vertices $(2,2), \ldots, (2,n), (3,n), \ldots (n,n)$ in $\TE_n(\widehat\Gamma_n^n)$ 
(marked gray in Fig.~\ref{sub5}) and ignore the isolated frozen vertices thus obtained, we will be left
with a quiver isomorphic to $\widehat\Gamma_{n-1}^{n-1}$ whose vertices are labeled by $(i,j)$, $2\le i\le n$, 
$1\le j\le  n$, $(i,j)\ne (2, n)$, and have functions $\chi^1_{ij}$ attached to them. The new special edge is $(2,n-1)\to (3,n)$.
Let us call the resulting 
extended seed $\widehat{\Sigma}_{n-1}^{n-1}$. 

We can now repeat the procedure described above $n-2$ more times  by applying, on the $k$th step, the sequence of mutations 
$\TE_{n-k+1}$ to the extended seed
$$
\widehat\Sigma_{n-k+1}^{n-k+1}= \left((\chi^{k-1}_{ij})_{k\le i\le n, 1\le j\le n-k+2, 
 (i,j)\ne (k, n-k+2)}, \widehat\Gamma_{n-k+1}^{n-k+1}\right ).
$$
 The functions  $\chi^k_{ij}$ are subject to relations
$$
\chi^{k}_{ij} \chi^{k-1}_{ij} = \chi^{k}_{i,j-1} \chi^{k-1}_{i,j+1} + \chi^{k}_{i+1,j} \chi^{k-1}_{i-1,j},\qquad
 k+1\le i\le n, \quad 2\le j\le n-k+1,
$$
where we adopt the convention $\chi^{k}_{n+1,j} = \chi^{k-1}_{n1}$, $\chi^{k}_{i1} = \chi^{k-1}_{i-1,1}$. 
Arguing as above, we conclude that
\begin{equation}
\label{chistark}
\chi^k_{ij}= \begin{cases}  \det X_{[i-k,n]}^{[1,k] \cup [j+k, n+j-i+k]} & \mbox{if}\quad  i - k+1 > j, \\
\det \left [ X^{[1,k]\cup [j+k,n]}\;\; Y^{[n+i-j+1-k,n]} \right ]_{[i-k,n]} & \mbox{if} \quad i - k+1 \leq j.  \end{cases}
\end{equation}
To verify \eqref{chistark} for $i-k+1>j$, one has to apply the short Pl\"ucker relation to 
$$
\left [ \begin{array}{cc}I^{[i-k]} & X^{[1,k] \cup [j+k-1, n+j-i+k]} \end{array} \right ]_{[i-k,n]}
$$
using columns 
$I^{[i-k]}$, $X^{[k]}$, $X^{[j+k-1]}$, $X^{[n+j-i+k]}$. In the case $i-k+1\leq j$, the short Pl\"ucker relation is applied to 
$$
\left [ \begin{array}{ccc} I^{[i-k]} & X^{[1,k]\cup [j+k-1,n]} & Y^{[n+i-j-k+1,n]}\end{array} \right ]_{[i-k,n]}
$$
using columns 
$I^{[i-k]}$, $X^{[k]}$, $X^{[j+k-1]}$, $Y^{[n+i-j-k+1]}$.
Note that 
\begin{align*}
\chi^k_{k+1,j} &= \det \left [ X^{[1,k]\cup [j+k,n]}\;\; Y^{[n-j+2,n]} \right ]_{[1,n]} \\
&= \det X \cdot(-1)^{(n-j-k+1)(j-1)}\det U_{[k+1,j+k-1]}^{[n-j+2,n]}\\
&= \det X \cdot (-1)^{(n-j-k+1)(n-k-1)}h_{k+1, n-j+2}(U),\qquad 2\le j\le n-k+1.
\end{align*}

Define the sequence of transformations $\TE$ as the composition $\TE= \TE_2 \circ \cdots \circ \TE_n$. Assertion (i) of Theorem~\ref{clusterU} follows from the fact
that $\fy$-vertices of $Q_n$ are not involved in any of $\TE_i$. This fact also implies that the subquiver of $Q_n$ induced by 
$\fy$-vertices remains intact in $\TE(Q_n)$. As it was shown above, the function $\det X\cdot h_{ij}(U)$ 
is attached to the vertex $(i,n-j+2)$. It is easy to prove by induction that the last mutation at $(i,j)$ (which occurs 
at the $(i-1)$-st step) creates edges $(i,j)\to (i,j-1)$, $(i-1,j)\to (i,j)$ and $(i,j-1)\to (i-1,j)$. 
Comparing this with the description of $Q_n$ in Section~\ref{init} and the definition of $Q_n^\dag$ in Section~\ref{main}
yields assertions (ii) and (iii) of Theorem~\ref{clusterU}. Finally, assertion (iv) follows from the fact that the special
edge $(i,n-i+1)\to (i+1,n-i+2)$ disappears after the last mutation at $(i+1,n-i+1)$.

The quiver $\TE(Q_5)$ and the subquiver $Q_5'$ are shown in Fig.~\ref{sub6}. The vertices of $Q_5'$ are shadowed in dark gray. The  area shadowed in light gray
represents the remaining part of $\TE(Q_5)$. The only vertices in this part shown in the figure are those connected to vertices of $Q_5'$.

\begin{figure}[ht]
\begin{center}
\includegraphics[width=8cm]{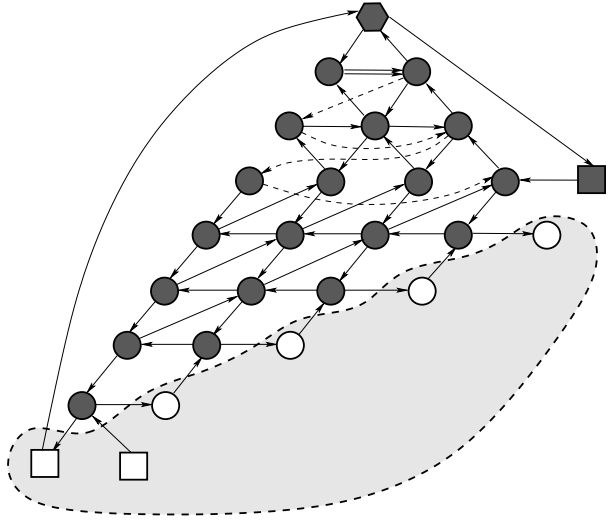}
\caption{Quiver $\TE(Q_5)$ and the subquiver $Q_5'$}
\label{sub6}
\end{center}
\end{figure}

\subsection{Proof of Theorem~\ref{Urestore}}
\subsubsection{The nerve $\N_0$} \label{thenerve} The nerve $\N_0$ is obtained as follows: it contains the seed $\widetilde\Sigma'_n$ built in the proof
of Theorem~\ref{clusterU}, the seed $\widetilde\Sigma''_n$ adjacent to $\widetilde\Sigma'_n$ in direction $\psi_{n-1,1}$, and the seed 
$\widetilde\Sigma'''_n$ adjacent to $\widetilde\Sigma''_n$ in direction $\psi_{1,n-1}$. 
Besides, it contains $n-3$ seeds adjacent to $\widetilde\Sigma'''_n$ in directions $\psi_{n-i,i}$, $2\le i\le n-2$, and $2(n-1)^2$ seeds adjacent to $\widetilde\Sigma'_n$ in all the remaining directions.

We subdivide $\N_0$ into several disjoint components. Component I contains the seed $\widetilde\Sigma'_n$ and its $(n-1)^2$ neighbors in directions that are not in $Q_n'$.
Component II contains $n-3$ neighbors of $\widetilde\Sigma'_n$ in directions
$\psi_{i1}$, $2\le i\le n-2$. 
Component III contains only the neighbor of $\widetilde\Sigma'_n$ in direction $\psi_{11}$. 
Component IV contains $(n-2)(n-3)/2$ seeds adjacent to $\widetilde\Sigma'_n$ in directions $\psi_{kl}$, $k+l<n$, $l>1$.
Component V contains $n(n-1)/2$ seeds adjacent to $\widetilde\Sigma'_n$ in directions
$h_{ij}$, $2\le i\le j\le n$. 
Component VI contains the seeds $\widetilde\Sigma''_n$ and 
$\widetilde\Sigma'''_n$ together with all other seeds adjacent to the latter. 

In each of the components we consider several normal forms for $U$ with respect to actions of different subgroups of $GL_n$. We 
 show how to restore entries of these normal forms and,
consequently, the entries of $U$ as Laurent expressions in corresponding clusters. Recall that $\det X$ and $\det X^{-1}$ belong to the ground ring, 
so it suffices to obtain Laurent expressions in variables $\psi_{kl}$, $h_{ij}$ (and their neighbors), and $c_i$ instead of actual cluster variables.

\subsubsection{Component I} To restore $U$ in component I, we use two normal forms for $U$: one under right multiplication by unipotent lower triangular matrices,
and the other under conjugation by unipotent lower triangular matrices, so 
$U=B_+N_-=\bar N_- \bar B_+ C \bar N_-^{-1}$,  where 
$B_+, \bar B_+$ are upper triangular, $N_-, \bar N_-$ are unipotent lower triangular, and $C=e_{21} + \cdots + e_{n, n-1} + e_{1n}$ is the 
cyclic permutation matrix (cf.~with~\eqref{normalU}). Note that by~\eqref{ourinv_prop}, 
$h_{ij}(U)=h_{ij}(B_+)$ and $\psi_{kl}(U)=\psi_{kl}(\bar B_+C)$.  
Our goal is to restore $B_+$ and $\bar B_+C$.

Once this is done, the matrix $U$ 
itself is restored as follows. We multiply the equality $B_+N_-=\bar N_-\bar B_+ C \bar N_-^{-1}$ by $W_0$ on the left and by $\bar N_-$ on the right,
where $W_0$ is the matrix corresponding to the longest permutation of size $n$, and
consider the Gauss factorizations~\eqref{gauss} for $W_0B_+$ in the left side and for  $W_0\bar B_+C$ in the right hand side. This gives
\[
(W_0B_+)_{>0}(W_0B_+)_{\le 0}N_-\bar N_-=W_0\bar N_-W_0\cdot(W_0\bar B_+C)_{>0}(W_0\bar B_+C)_{\le 0},
\]
where $W_0\bar N_-W_0$ is unipotent upper triangular. Consequently, 
\begin{equation}\label{gluecompI}
\bar N_-=W_0(W_0B_+)_{>0}(W_0\bar B_+C)_{>0}^{-1}W_0. 
\end{equation}
Recall
that matrix entries in the Gauss factorization are given by Laurent expressions in the entries of the initial matrix with denominators equal to its trailing
principal minors (see, e.g., \cite[Ch.~2.4]{Gant}). Clearly, the trailing principal minors of $W_0B_+$ and $W_0 \bar B_+C$ are just $\psi_{n-i,i}$, 
which allows to restore $U$ in any cluster of component I.

Restoration of $B_+=(\beta_{ij})$ is standard: an explicit computation shows that $\beta_{ii}=\pm h_{ii}/h_{i+1,i+1}$ with $h_{n+1,n+1}=1$, 
and $\beta_{ij}$ for $i<j$ is a Laurent
polynomial in $h_{kl}$, $k\le l$,  with denominators in the range $i+1\le k\le n$, $j+1\le l\le n$ (here $h_{1l}$ is identified up to a sign with 
$\psi_{l-1, n-l+1}$ for $l>1$).
Since all $h_{ij}$ are cluster variables in the clusters of component I, we are done.

In order to restore $\bar M=\bar B_+C$ we proceed as follows. Let $\bar B_+=(\bar\beta_{ij})$. Clearly, $\psi_{n-k,1}=\pm \prod_{i=1}^k\bar\beta_{ii}^{k-i+1}$ for
$1\le k\le n-1$, which yields
\begin{equation}\label{barbeta}
\bar\beta_{ii}=\pm\frac{\psi_{n-i,1}\psi_{n-i+2,1}}{\psi_{n-i+1,1}^2}, \qquad 1\le i\le n-1,
\end{equation}
where we assume $\psi_{n1}=\psi_{n+1,1}=1$. The remaining diagonal entry $\bar\beta_{nn}$ is given by 
\[
\bar\beta_{nn}=h_{11}\prod_{i=1}^{n-1}\bar\beta_{ii}^{-1}=\pm\frac{h_{11}\psi_{21}}{\psi_{11}}.
\]

\begin{remark}\label{forfy11}
Note that the only diagonal entries depending on $\psi_{11}$ are $\bar\beta_{nn}$ and $\bar\beta_{n-1,n-1}=\pm\psi_{11}\psi_{31}/\psi_{21}^2$. This fact will be important 
for the restoration process in component III below.
\end{remark}

We proceed with the restoration process and use~\eqref{tildefy} to find
\[
\det (\bar M)_{[n-k-l+2,n-k]}^{[n-l+1,n-1]}=\pm \psi_{kl}\prod_{i=1}^{n-k-l+1}\bar\beta_{ii}^{k+l+i-n-2},
\]
which together with~\eqref{barbeta} gives $\det (\bar M)_{[n-k-l+2,n-k]}^{[n-l+1,n-1]}=\pm\psi_{kl}\psi_{k+l,1}^{k+l-3}/\psi_{k+l-1,1}^{k+l-2}$. By
Remark~\ref{tribfz}, this means that all entries $\bar \beta_{ij}$ with $i>1$ are restored  as Laurent polynomials in any cluster in component I.
Note that non-diagonal entries do not depend on $\psi_{11}$.

\begin{remark}\label{moreforfy11} Using Remark~\ref{tildefysign}, we can find signs in the above relations. Specifically,  
$\bar\beta_{n-1,n}=(-1)^n\psi_{12}/\psi_{21}$. This fact will be used in the restoration
process in component III below.
\end{remark}

To restore the entries in the first row of $\bar M$, we first conjugate it by a diagonal matrix $\Delta=\diag(\delta_1,\dots\delta_{n-1},h_{11}^{-1})$ so that
 $\Delta\bar M\Delta^{-1}$ has ones on the subdiagonal. This condition implies $\delta_i=\pm \psi_{n-i+1,1}/\psi_{n-i,1}$. Consequently, the entries of the rows
$2,\dots,n$ of $\Delta\bar M\Delta^{-1}$ remain Laurent polynomials. Next, we further conjugate
the obtained matrix with a unipotent upper triangular matrix $N_+$ so that $\bar M^*=N_+\Delta\bar M\Delta^{-1}N_+^{-1}$ has the companion form
\begin{equation}\label{compan}
\bar M^*=\left[\begin{array}{c} \gamma\\ \one_{n-1}\;\; 0\end{array}\right]
\end{equation}
with $\gamma=(\gamma_1,\dots,\gamma_n)$.
If we set all non-diagonal entries in the last column of $N_+$ equal to zero, all other entries of
$N_+$ (and hence of $N_+^{-1}$) can be restored uniquely as polynomials in the entries in the rows
$2,\dots,n$ of $\Delta\bar M\Delta^{-1}$. Recall that $\bar M^*$ is obtained from $U$ by conjugations, and hence $U$ and $\bar M^*$ are
isospectral. Therefore, $\gamma_i=(-1)^{i-1} c_{n-i}$ for $1\le i\le n$. This allows to restore the entries in the first row of 
\begin{equation}\label{betaviac}
\bar M=\Delta^{-1}N_+^{-1}\bar M^* N_+\Delta
\end{equation}
as Laurent polynomials in any cluster in component I.

\begin{remark}\label{nohii}
Note that although diagonal entries of $B_+$ are Laurent monomials in stable variables $h_{ii}$, Laurent expressions for entries of $(W_0B_+)_{>0}$ depend
on $h_{ii}$ polynomially. This follows from the fact that these entries are Laurent polynomials in dense minors of $W_0B_+$ containing the last column; recall that
such minors are cluster variables in  any cluster of component I. Moreover,
dense minors containing the upper right corner enter these expressions polynomially, see Remark~\ref{tribfz}. Consequently, stable variables $h_{ii}$ do not enter denominators of Laurent expressions for  entries of $U$ by~\eqref{gluecompI},  since restoration of $\bar B_+C$ does not involve division by $h_{ii}$. 
\end{remark}

\subsubsection{Component II} The two normal forms used in this component are given by  $U=B_+N_-=\check N_-\check B_+ W_0 \check N_-^{-1}$, where $B_+, \check B_+$ are upper
triangular, $N_-, \check N_-$ are unipotent lower triangular, and $W_0$ is the matrix corresponding to the longest permutation, see Lemma~\ref{BW}.  
Note that by~\eqref{ourinv_prop}, 
$h_{ij}(U)=h_{ij}(B_+)$ and $\psi_{kl}(U)=\psi_{kl}(\check B_+W_0)$.  
Our goal is to restore $B_+$ and $\check B_+W_0$.

Once this is done, the matrix $U$ 
itself is restored as follows. We multiply the equality $B_+N_-=\check N_-\check B_+ W_0 N_-^{-1}$ by $W_0$ on the left and by $\check N_-$ on the right and
consider the Gauss factorization~\eqref{gauss} for $W_0B_+$ in the left hand side. This gives
\[
(W_0B_+)_{>0}(W_0B_+)_{\le 0}N_-\check N_-=W_0\check N_-W_0\cdot W_0\check B_+W_0
\]
where $W_0\check N_-W_0$ is unipotent upper triangular and $W_0\check B_+W_0$ is lower triangular. Consequently, $\check N_-=W_0(W_0B_+)_{>0}W_0$.  
Clearly, the trailing principal minors of $W_0B_+$ are just $\psi_{n-i,i}$, which allows to restore $U$
in $\widetilde\Sigma'_n$ and in any cluster of component II.

Restoration of $B_+$ is exactly the same as before. In order to restore $\check M=\check B_+W_0$ we proceed as follows. 
Let $\check B_+=(\check\beta_{ij})_{1\le i\le j\le n}$. 
We start with the diagonal entries. 
An explicit computation immediately gives
\begin{equation}\label{diagrI}
\check\beta_{ii}=\pm\frac{\psi_{n-i,i}}{\psi_{n-i+1,i-1}}, \qquad 1\le i\le n,
\end{equation}
with $\psi_{n0}=1$ and $\psi_{0n}=h_{11}$. Next, we recover the entries in the last column of $\check B_+$. We find
\[
\psi_{n-l,l-1}=\pm\check \beta_{11}\check \beta_{ln}\prod_{i=1}^{l-1}\check \beta_{ii},
\]
which together with~\eqref{diagrI} gives
\begin{equation}\label{lastrI}
\check \beta_{ln}=\pm\frac{\psi_{n-l,l-1}}{\psi_{n-1,1}\psi_{n-l+1,l-1}}, \quad 2\le l\le n-1.
\end{equation}

Note that we have already restored the last two rows of $\check M$. We restore the other rows consecutively, starting from row $n-2$
and moving upwards. To this end,  
define an $n\times n$ matrix $\Psi$ via
\[
\Psi=\left[\begin{array}{ccccc} e_1 & \check Me_1 & \check M^{2}e_1 & \dots & \check M^{n-1}e_1\end{array}\right].
\]
Clearly, for $2\le l\le n-2$, $2\le t\le n-l+1$ one has
\[
\psi_{n-t-l+2,l-1}=\pm \check\beta_{11}^{t-1}\prod_{i=1}^{l-1}\check\beta_{ii}\det \Psi_{[l,t+l-2]}^{[2,t]},
\]
which together with~\eqref{diagrI} yields
\begin{equation}\label{detPsi}
\det \Psi_{[l,t+l-2]}^{[2,t]}=\pm \frac{\psi_{n-t-l+2,l-1}}{\psi_{n-l+1,l-1}\psi_{n-1,1}^{t-1}}.
\end{equation}
Therefore, each $\det \Psi_{[l,t+l-2]}^{[2,t]}$ is a Laurent polynomial in any cluster in component II. Moreover, the minors $\det\Psi_I^{[2,t]}$ 
possess the same property for any index set $I\subset [2,n]$, $|I|=t-1$, since they can be expressed as Laurent polynomials 
in $\det \Psi_{[l,t+l-2]}^{[2,t]}$ for $l>2$  that are polynomials in $\det\Psi_{[2,t]}^{[2,t]}$, see Remark~\ref{tribfz}.

On the other hand, $\Psi_{[l,t+l-2]}^{[2,t]}=\check M_{[l,t+l-2]}\Psi^{[1,t-1]}$, which yields a system of linear 
equations on the entries $\check\beta_{lj}$:
\begin{equation}\label{rowl}
\begin{split}
\sum_{j=2}^{n-l}\check\beta_{lj}&\det \left(\left[\begin{array}{c} e_j^T\\ \check M_{[l+1,t+l-2]}\end{array}\right]\Psi^{[1,t-1]}\right)\\
& =\det \Psi_{[l,t+l-2]}^{[2,t]}-\det \left(\left[\begin{array}{c} \hat \beta_l\\ \check M_{[l+1,t+l-2]}\end{array}\right]\Psi^{[1,t-1]}\right),
\quad 3\le t\le n-l+1,
\end{split}
\end{equation}
where $\hat \beta_{l1}=\check\beta_{1l}$, $\hat \beta_{l,n-l+1}=\check\beta_{ll}$, and $\hat \beta_{lj}=0$ for $j\ne 1, n-l+1$. 
Rewrite the second determinant in the right hand side 
of~\eqref{rowl} via the Binet--Cauchy formula; it involves minors $\det \Psi^{[2,t]}_I$ and minors of $\check M_{[l+1,t+l-2]}$.
Assuming that the entries in rows $l+1,\dots,n$  have been already restored and taking into account~\eqref{detPsi}, we ascertain that the 
right hand side can be expressed as a Laurent polynomial 
in any cluster in component II. Clearly, the same is also true for the coefficients in the left hand side of~\eqref{rowl}.

It remains to calculate the determinant of the linear system~\eqref{rowl}. Denote the coefficient at $\check\beta_{lj}$ in the $t$-th equation by $a_{j,t-2}$.
Then
\[
a_{jk}=\sum_{i=1}^k\left(\check M^i\right)_{l1}\det\Psi_{[l+1,k+l]}^{[2,i+1]\cup[i+3,k+2]}
=\sum_{i=1}^k\left(\check M^i\right)_{l1}\gamma_{ik}.
\]
Let $A=(a_{jk})$, $2\le j\le n-l$, $1\le k\le n-l-1$, be the matrix of the system~\eqref{rowl}. By the above formula we get
\[
A=\left[\begin{array}{ccc} \check Me_1 & \dots &\check M^{n-l-1}e_1\end{array}\right]_{[2,n-l]}\left[\begin{array}{cccc}
\gamma_{11} & \gamma_{12} & \dots &\gamma _{1,n-l-1}\\
0 & \gamma_{22} & \dots & \gamma_{2,n-l-1}\\
\vdots & \vdots & \ddots & \vdots \\
0 & 0 & \dots & \gamma_{n-l-1,n-l-1}\end{array}\right],
\]
and hence 
\[
\begin{split}
\det A=\det \Psi_{[2,n-l]}^{[2,n-l]}\prod_{i=1}^{n-l-1}\gamma_{ii}&=\det \Psi_{[2,n-l]}^{[2,n-l]}\prod_{i=1}^{n-l-1}\det\Psi_{[l+1,l+i]}^{[2,i+1]}\\
&=\pm \frac{\psi_{l1}}{\psi_{n-l,l}^{n-l-1}\psi_{n-1,1}^{(n-l)(n-l+1)/2}}\prod_{i=1}^{n-l-1}\psi_{n-l-i,l}.
\end{split}
\]

We thus restored the entries in the $l$-th row of $\check M$ as Laurent polynomials in all clusters in component II except for the neighbor 
of $\widetilde\Sigma'_n$ in direction $\psi_{l1}$, which we denote $\widetilde\Sigma'_n(l)$. In the latter cluster $\psi_{l1}$, which enters the expression for $\det A$, is no longer available.
It is easy to see that the factor $\psi_{l1}$ in $\det A$ comes from the $(n-l+1)$-st equation in~\eqref{rowl}: its left hand side is defined by the expression
\begin{equation}\label{oldeq}
\det \check M_{[l,n-1]}^{[1,n-l]}\det\Psi_{[2,n-l]}^{[2,n-l]}=\pm\det \check M_{[l,n-1]}^{[1,n-l]}
\frac{\psi_{l1}}{\psi_{n-1,1}^{n-l}}.
\end{equation}
In other words, each coefficient in the left hand side of this equation is proportional to $\psi_{l1}$. To avoid this problem, we replace this 
equation by a different one. Recall that by~\eqref{k1neigh},  in the cluster under consideration 
$\psi_{l1}$ is replaced by $\psi_{1,l-1}^{**}$
given by
\[
\begin{split}
&\psi_{1,l-1}^{**}=\pm \det \left[
e_n \; \check M^{[n-l+2,n]} \; \check M^2e_{n-1} \; \check M^2e_n \; \check M^3e_n \dots  \check M^{n-l}e_n
\right]\\
&\quad=\pm\check\beta_{11}^{n-l-1}\check\beta_{22} \det \left[
\check M^{[n-l+2,n]} \; \check Me_{2} \; \check Me_1 \; \check M^2e_1 \dots  \check M^{n-l-1}e_1
\right]_{[1,n-1]}\\
&\quad=\pm\check\beta_{11}^{n-l}\check\beta_{22}\det \left(\check M_{[2,n-1]}^{[1,n-1]}\left[
\one^{[1,2]\cup[n-l+2,n-1]} \;  \check Me_1 \; \check M^2e_1 \dots  \check M^{n-l-2}e_1
\right]_{[1,n-1]}\right);
\end{split}
\]
here we used relations $\check Me_n=\check\beta_{11}e_1$ and $\check Me_{n-1}=\check\beta_{12}e_1+\check\beta_{22}e_2$. By the Binet--Cauchy formula, the
latter determinant can be written as
\[
\sum_{j=1}^{n-1}\det \check M_{[n-j+1,n-1]}^{[1,j-1]}\prod_{i=2}^{n-j}\check\beta_{ii}\det\left[
\one^{[1,2]\cup[n-l+2,n-1]} \;  \check Me_1 \; \check M^2e_1 \dots  \check M^{n-l-2}e_1
\right]_{[1,n-1]\setminus\{j\}}.
\]
 Clearly, the second factor in each summand vanishes for $j=1,2$ and $n-l+2\le j\le n-1$. For $3\le j\le n-l+1$, the second factor equals
$\det\Psi_{[3,n-l+1]\setminus\{j\}}^{[2,n-l-1]}$. As it was explained above, for $3\le j\le n-l$ these determinants are Laurent polynomials in
all clusters of component II, whereas the corresponding first factors contain only entries in rows $l+1,\dots,n$ of $\check M$, which are already restored
as Laurent polynomials. Consequently, the left hand side of the new equation for the entries in the 
$l$-th row is defined by the $(n-l+1)$-st summand
\[
\det \check M_{[l,n-1]}^{[1,n-l]}\prod_{i=2}^{l-1}\check\beta_{ii}\det\Psi_{[3,n-l]}^{[2,n-l-1]}=\pm\det \check M_{[l,n-1]}^{[1,n-l]}
\frac{\psi_{l-1,n-l+1}\psi_{l2}}{\psi_{n-2,2}\psi_{n-1,1}^{n-l-1}};
\]
note that the inverse of the factor in the right hand side above is a Laurent monomial in the cluster under consideration.
Comparing with~\eqref{oldeq}, we infer that the left hand side of the new equation is proportional to the left hand side of the initial one. Therefore, the determinant 
of the new system is Laurent in the cluster $\widetilde\Sigma'_n(l)$, and the entries of the $l$-th row of $\check M$ are restored as Laurent polynomials.

Therefore, the entries in all rows of $\check M$ except for the first one are restored as Laurent polynomials in component II. The entries of the first row
are restored via Lemma~\ref{row1viac} as polynomials in the entries of other rows and variables $c_i$ divided by the $\det K=\det\Psi=
\pm h_{11}\psi_{11}/\psi_{n-1,1}^n$; the latter equality follows from $\check\beta_{11}e_1=\check M e_n$ and~\eqref{diagrI}.

\subsubsection{Component III} In this component we use all three normal forms that have been used in components I and II. 
Restoration of $B_+$ is exactly the same as before. Restoration of $\check B_+W_0$ goes through for all entries except for the entries in the first row, 
since the determinant $\det K$ involves a factor $\psi_{11}$, which is no longer available. On the other hand,
restoration of $\bar B_+C$ also fails at two instances: firstly, at the entry $\bar\beta_{nn}$, see Remark~\ref{forfy11}, and secondly, at the entry
$\bar\beta_{1n}$, which gets $\psi_{11}$ in the denominator after the conjugation by $\Delta^{-1}$. However, we will be able to use partial results obtained 
during restoration of $\bar B_+C$ in order to restore the first row of $\check B_+W_0$.

Indeed, using Lemma~\ref{BW} we can write
$\check M = W_0\left (W_0 U\right)_{\leq 0} W_0 \left (W_0 \bar M\right)_{>0}W_0$.
Clearly, $W_0\bar M$ is block-triangular
with diagonal blocks $1$ and $W_0'\bar B_+'$, where $\bar B_+'=(\bar B_+)_{[2,n]}^{[2,n]}$ and $W_0'$ is the matrix of the longest permutation on size $n-1$. Therefore,
\[
(\check\beta_{11}, \check\beta_{12}, \dots, \check\beta_{1n})=(\bar\beta_{11},\bar\beta_{1n},\bar\beta_{1,n-1},\dots,\bar\beta_{12})
\left[\begin{array}{cc} 0 & (W_0'\bar B_+')_{>0}W_0'\\
                        1  & 0\end{array}\right],
\]
which can be more conveniently written as
\begin{equation}\label{row1}
\begin{split}
(\check\beta_{11}, \check\beta_{12}, \dots, \check\beta_{nn})&
\left[\begin{array}{cc} 1  & 0\\
0 & W_0'(W_0'\bar B_+')_{\le 0}W_0'
      \end{array}\right]\\
			&=
(\bar\beta_{11},\bar\beta_{1n},\bar\beta_{1,n-1},\dots,\bar\beta_{12})
\left[\begin{array}{cc} 0 & W_0'\bar B_+'W_0'\\
                        1  & 0\end{array}\right].												
\end{split}
\end{equation}

It follows from the description of the restoration process for $\bar M$, that all entries  of the matrix in the left hand side 
of the system~\eqref{row1} are Laurent polynomials in component III, and the determinant of this matrix equals $\pm h_{11}/\psi_{n-1,1}$.
In the right hand side, $W_0'\bar B_+'W_0'$ is a lower triangular matrix with $\bar\beta_{nn}$ in the upper left corner and $\bar\beta_{n-1,n-1}$ next to it along the 
diagonal. Recall that other entries of $W_0'\bar B_+'W_0'$ do not involve $\psi_{11}$; besides, the entries $\bar\beta_{1j}$ for $j\ne n$ may involve $\psi_{11}$ only in the 
numerator, which makes them Laurent polynomials in component III. Therefore, the right hand side of~\eqref{row1} involves two expressions that should be investigated:
 $\bar\beta_{11}\bar\beta_{nn}+\bar\beta_{1n}\bar\beta_{n-1,n}$ and $\bar\beta_{1n}\bar\beta_{n-1,n-1}$. Recall that $\bar\beta_{1n}$ is the product of a Laurent 
polynomial in component III by $\delta_{n-1}/\delta_1=\pm\psi_{21}\psi_{n-1,1}/\psi_{11}$, whereas $\bar\beta_{n-1,n-1}=\pm\psi_{11}\psi_{31}/\psi_{21}^2$ by
Remark~\ref{forfy11}, hence the second expression above is a Laurent polynomial in component III.  

Recall further that $\bar M$ is transformed to the companion form~\eqref{compan} by a conjugation first with $\Delta$ and second with $N_+$. 
The latter can be written in two forms:
\[
N_+=\left[\begin{array}{cc} N_+' & 0\\
                            0 & 1\end{array}\right] \qquad\text{and}\qquad 
N_+=\left[\begin{array}{cc} 1 & \nu\\
                            0 & N_+''\end{array}\right]
\]
where $N_+=(\nu_{ij})$ and  $N_+''$  are $(n-1)\times (n-1)$ unipotent upper triangular matrices, and $\nu=(\nu_{12},\dots,\nu_{1,n-1})$. Consequently,~\eqref{compan}
yields 
\begin{equation}\label{forn}
N_+''(\Delta\bar M\Delta^{-1})_{[2,n]}^{[1,n-1]}=N_+'. 
\end{equation}
Note that~\eqref{betaviac} implies
\[
\left(\frac{\bar\beta_{12}}{\delta_1}, \frac{\bar\beta_{13}}{\delta_2},\dots,\frac{\bar\beta_{1n}}{\delta_{n-1}}\right)=
\frac1{\delta_1}\left(\gamma_1+\bar\nu_{12},\gamma_2+\bar\nu_{13},\dots, \gamma_{n-2}+\bar\nu_{1,n-1},\gamma_{n-1}\right)N_+',
\]
where $\bar\nu_{ij}$ are the entries of $(N_+')^{-1}$. Taking into account that 
 \[
\delta_1=\prod_{i=2}^n\bar\beta_{ii},\qquad \delta_{n-1}=\bar\beta_{nn},\qquad \prod_{i=1}^n\bar\beta_{ii}=(-1)^{n-1}h_{11},
\]
we infer that 
\[
\bar\beta_{1n}=(-1)^{n-1}\frac{\bar\beta_{11}\bar\beta_{nn}}{h_{11}}\left(\sum_{i=1}^{n-1}\gamma_i\nu_{i,n-1}+\sum_{i=2}^{n-1}\bar\nu_{1i}\nu_{i-1,n-1}\right).
\]
Besides, $\bar\beta_{n-1,n}=(-1)^n\psi_{12}/\psi_{21}$ by Remark~\ref{moreforfy11}. Denote $\zeta=(-1)^{n-1}\psi_{12}/\psi_{21}$,
Then one can write
\begin{equation}\label{binom}
\bar\beta_{11}\bar\beta_{nn}+\bar\beta_{1n}\bar\beta_{n-1,n}
=\frac{\bar\beta_{11}\bar\beta_{nn}}{h_{11}}
\left(h_{11}+(-1)^{n}\zeta\left( \sum_{i=1}^{n-1}\gamma_i\nu_{i,n-1}+\sum_{i=2}^{n-1}\bar\nu_{1i}\nu_{i-1,n-1}\right)\right).
\end{equation}

To treat the latter expression, we consider first the last column of $(\Delta\bar M\Delta^{-1})_{[2,n]}^{[1,n-1]}$. It is easy to see that the first $n-3$ entries
in this column equal zero  modulo $\psi_{11}$, whereas the $(n-2)$-nd entry equals $-\zeta$, and the last entry equals $1$. Consequently,~\eqref{forn}
implies that $\nu_{i,n-1}=(-1)^{n-i-1}\zeta^{n-i-1}\mod \psi_{11}$. Therefore, the first sum in the right hand side of~\eqref{binom} equals
\[
\sum_{i=1}^{n-1}\gamma_i(-1)^{n-i-1}\zeta^{n-i-1} \mod \psi_{11}.
 \]

By Lemma~\ref{Hessenberg},
 $-\bar\nu_{12},\dots,-\bar\nu_{1,n-1},0$ form the first row of the companion form for $(\Delta\bar M\Delta^{-1})_{[2,n]}^{[2,n]}$. Consequently,
\[
\begin{split}
\sum_{i=2}^{n-1}\bar\nu_{1i}\nu_{i-1,n-1}&=\sum_{i=2}^{n-1}(-1)^{n-i}\bar\nu_{1i}\zeta^{n-i} \mod \psi_{11}\\
&=(-1)^{n-1}\left(\det\left((\Delta\bar M\Delta^{-1})_{[2,n]}^{[2,n]}+\zeta\one_{n-1}\right)-\zeta^{n-1}\right) \mod \psi_{11}.
\end{split}
\]
Note that $\det((\Delta\bar M\Delta^{-1})_{[2,n]}^{[2,n]}+\zeta\one_{n-1})=0 \mod \psi_{11}$, since the last column of $(\Delta\bar M\Delta^{-1})_{[2,n]}^{[2,n]}$
is zero, and the second to last column equals $e_{n-1}-\zeta e_{n-2} \mod\psi_{11}$; therefore, the second sum in the right hand side of~\eqref{binom} 
equals $(-1)^{n}\zeta^{n-1} \mod \psi_{11}$. Combining this with the previous result and taking into account that $h_{11}=(-1)^{n-1}\gamma_n$,  we get
\[
\begin{split}
\bar\beta_{11}\bar\beta_{nn}&+\bar\beta_{1n}\bar\beta_{n-1,n}\\
&=\frac{\bar\beta_{11}\bar\beta_{nn}}{h_{11}}\left(h_{11}+(-1)^{n}\left(\sum_{i=1}^{n-1}\gamma_i(-1)^{n-i-1}\zeta^{n-i}+(-1)^{n}\zeta^{n}\right)\right)\mod \psi_{11}\\
&=\frac{\bar\beta_{11}\bar\beta_{nn}}{h_{11}}\left(\zeta^n+\sum_{i=1}^n\gamma_i(-1)^{i+1}\zeta^{n-i}\right) \mod \psi_{11}\\
&=\pm\frac{\psi_{n-1,1}}{\psi_{21}^{n-1}\psi_{11}}\det(\psi_{21}\bar M+(-1)^{n-1}\psi_{12}\one_n) \mod \psi_{11}.
\end{split}
\]
Recall that $\det(\psi_{21}\bar M+(-1)^{n-1}\psi_{12}\one_n)/\psi_{11}$ multiplied by an appropriate power of $\det X$ is the new variable that replaces $\psi_{11}$
in component III, 
and hence~\eqref{row1} defines the entries of the first row of $\check M$ as Laurent polynomials.

\subsubsection{Component IV} In this component we use the same two normal forms as in component I. Restoration of $B_+$  and of diagonal entries of $\bar B_+$ 
is exactly the same as before. Next,
we apply the second line of~\eqref{tildefy} and~\eqref{barbeta} to~\eqref{fyklex} (or~\eqref{fy1lex}) and
observe that the right hand side of the exchange relation for $\psi_{kl}$, $k>1$, 
can be written as
\[
\begin{split}
&\psi_{k+l-3,1}\psi_{k+l-2,1}\psi_{k+l-1,1}\\
&\times\left(h_{\alpha-1,\gamma+1}(\bar B_+)h_{\alpha,\gamma-1}(\bar B_+)h_{\alpha+1,\gamma}(\bar B_+)
+h_{\alpha-1,\gamma}(\bar B_+)h_{\alpha,\gamma+1}(\bar B_+)h_{\alpha+1,\gamma-1}(\bar B_+)\right),
\end{split}
\]
where $\alpha=n-k-l+2$, $\gamma=n-l+2$ (cf.~with~\eqref{fyviah}). Note that for $l=2$ we have $h_{\alpha-1,\gamma+1}(\bar B_+)=h_{\alpha,\gamma+1}(\bar B_+)=1$,
and hence the expression above can be rewritten as
\[
{\psi_{k-1,1}\psi_{k1}\psi_{k+1,1}}
\left(h_{\alpha,n-1}(\bar B_+)h_{\alpha+1,n}(\bar B_+)
+h_{\alpha-1,n}(\bar B_+)h_{\alpha+1,n-1}(\bar B_+)\right).
\]
We conclude that the map 
\[
(k,l)\mapsto \begin{cases} (\alpha,\gamma)\quad\text{for $l>1$},\\
                           (\alpha+1,\alpha+1)\quad\text{for $l=1$ }
\end{cases}
\] 
transforms exchange relations
for $\psi_{kl}$, $l>1$, $k+l<n$, to exchange relations for the standard cluster structure on triangular matrices of size $(n-2)\times(n-2)$, 
up to a monomial factor in variables that 
are fixed in component IV, see Fig.~\ref{Q6}. Consequently, the entries $\bar\beta_{ij}$, $2\le i <j\le n$, can be restored as Laurent polynomials in component IV via 
Remark~\ref{tribfz}. 

\begin{figure}[ht]
\begin{center}
\includegraphics[width=10cm]{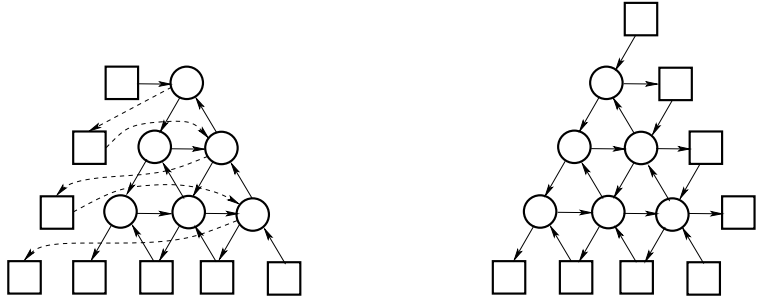}
\caption{Modification of the relevant part of $Q_6^\dag$ in component IV}
\label{Q6}
\end{center}
\end{figure}

The entries in the first row of $\bar B_+$ are restored exactly as in component I.

\subsubsection{Component V} In this component we once again use the same two normal forms as in component I. Restoration of $\bar B_+$ is exactly the same as before. 
To restore $B_+$, we note that the cluster structure in component V coincides with the standard cluster structure on triangular $n\times n$ matrices, and hence
the entries of $B_+$ can be restored as Laurent polynomials via Remark~\ref{tribfz}. 

\subsubsection{Component VI} The two normal forms used in this component are given by $U=B_+N_-=\widehat N_+\widehat N_-\widehat M S_{12}\widehat N_-^{-1}\widehat N_+^{-1}$, 
where $B_+$ is upper triangular, $N_-, \widehat N_-=(\hat\nu_{ij})$ are unipotent lower triangular with $\hat \nu_{j1}=0$ for $2\le j\le n$, 
 $\widehat N_+=\one_n+\hat\nu e_{12}$, and $\widehat M=(\hat \mu_{ij})$ satisfies conditions $\hat \mu_{1n}=0$ and $\hat \mu_{i,n+2-j}=0$ for
$2\le j<i\le n$, see Lemma~\ref{NMN}.  
Note that 
$h_{2j}(U)=h_{2j}(\widehat M)$ and $\psi_{kl}(U)=\psi_{kl}(\widehat M)$.  
Our goal is to restore $B_+$ and $\widehat M$.

Once this is done, the matrix $U$ itself is restored as follows. First, by the proof of Lemma~\ref{NMN}, 
$\hat \nu=\beta_{1n}/\hat \mu_{2n}$, which is a Laurent polynomial in 
component VI, since $\hat \mu_{2n}= h_{2n}$ is a Laurent monomial. Next, we write 
$B_+N_-\widehat N_+\widehat N_-=\widehat N_+\widehat N_-\widehat M$. Taking into account that 
 $\widehat N_+\widehat N_-=\widehat N_-\widehat N_+$, we arrive at
\begin{equation*}
\bar W_0' B_+\cdot N_-\widehat N_-=\bar W_0'\widehat N_- \bar W_0'\cdot\bar W_0' \widehat N_+\widehat M\widehat N_+^{-1}
\end{equation*}
with $\bar W_0'=\begin{pmatrix} 1 & 0\\ 0 & W_0'\end{pmatrix}$. Note that the second factor on the left is unipotent lower triangular, whereas the first factor
on the right is unipotent upper triangular. Taking the Gauss factorizations of the remaining two factors, we restore
$\bar W_0'\widehat N_- \bar W_0'=\left( \bar W_0' B_+\right)_{> 0}\left(\bar W_0' \widehat N_+\widehat M\widehat N_+^{-1}\right)^{-1}_{> 0}$. Note that trailing minors
needed for Gauss factorizations in the right hand side above equal $\det\widehat M_{[2,i]}^{[n-i+2,n]}= h_{2i}$, and hence are Laurent
monomials in component VI.

We describe the restoration process at $\widetilde\Sigma'''_n$, and indicate the changes that occur at its neighbors.
Restoration of $B_+$ is almost the same as before. The difference is that $h_{1n}$ and $h_{12}$, which coincide up to a sign with $\psi_{n-1,1}$ and $\psi_{1,n-1}$,
are no longer available (at $\widetilde\Sigma''_n$ only $h_{1n}$ is not available). However, since they are cluster variables in other clusters, say, in any cluster in component I, they both can be written as Laurent polynomials
at  $\widetilde\Sigma'''_n$. Moreover, they never enter denominators in expressions for $\beta_{ij}$, and hence $B_+$ is restored. At the neighbor of 
$\widetilde\Sigma'''_n$ in direction $\psi_{n-i,i}$ we apply the same reasoning to $h_{1,n-i+1}=\pm \psi_{n-i,i}$; at $\widetilde\Sigma''_n$ we apply it only to $h_{1n}$.

In order to restore $\widehat M$ we proceed as follows. First, we note that 
\[
h_{2j}=\pm\prod_{i=j}^n \hat \mu_{n+2-i,i}, 
\]
and hence $\hat \mu_{n+2-i,i}=\pm h_{2j}/h_{2,j+1}$ for $2\le j\le n$ with $h_{2,n+1}=1$. 
 Next, note that by~\eqref{reg6}, the function $\psi_{n-1,1}$ is replaced in component VI  by
$\psi'_{n-1,1}=\pm\det \widehat M_{[1,n-1]}^{1\cup[3,n]}$. Besides, it is easy to see that in all clusters in component VI except for $\widetilde\Sigma''_n$,
$\psi_{1,n-1}$ is replaced by $\psi_{1,n-1}'=\pm\det \widehat M_{[2,n]}^{1\cup[3,n]}$. 
Consequently, $h_{11}=\pm\left(\hat \mu_{n1}\psi_{1,n-1}-\hat \mu_{n2}\psi'_{n-1,1}\right)$
yields $\hat \mu_{n1}$ as a Laurent monomial at $\widetilde\Sigma''_n$, and $\psi_{1,n-1}'=\pm\hat\mu_{n1}h_{23}$ yields it as a Laurent monomial in all
other clusters in component VI. 

To proceed further, we introduce an $n \times n$ matrix $\widehat\Psi$ similar to the matrix $\Psi$ used in component II:
\[
\widehat\Psi=\left[\begin{array}{ccccc} e_2 & \widehat Me_2 & \widehat M^{2}e_2 & \dots & \widehat M^{n-1}e_2\end{array}\right].
\]

\begin{lemma}\label{psiminors}
The minors $\det\widehat\Psi_I^{[2,t]}$, $2\le t\le n-1$, are Laurent polynomials in component VI for any index set $I$ such that $2\notin I$, $|I|=t-1$.
\end{lemma}

\begin{proof}
An easy computation shows that for $k+l<n$ one has 
\[
\psi_{kl}=\pm h_{2n}^{n-k-l+1}\sum_{j\in [1,l+1]\setminus2}(-1)^{j+\chi_j}\det\widehat M_{[1,l+1]\setminus\{2, j\}}^{[n-l+1,n-1]} 
\det\widehat\Psi_{j\cup[l+2,n-k]}^{[2,n-k-l+1]},
\]
where $\chi_1=1$ and $\chi_j=0$ for $j\ne 1$. It follows from the discussion above that $\det\widehat M_{[3,l+1]}^{[n-l+1,n-1]}=(-1)^{l-1}h_{2,n-l+1}/h_{2n}$.
Besides, 
\begin{equation*}
\begin{split}
\det \widehat M_{[1,l+1]\setminus\{2, j\}}^{[n-l+1,n-1]}&=(-1)^{(j-1)(l-j+1)}\det\widehat M_{[1,j-1]\setminus2}^{[n-j+2,n-1]}\det\widehat M_{[j+1,l+1]}^{[n-l+1,n-j+1]}\\
&=(-1)^{l}\frac{h_{1,n-j+2}}{h_{2n}}\frac{h_{2,n-l+1}}{h_{2,n-j+2}},
\end{split}
\end{equation*}
and hence
\begin{equation}\label{psiviaPsi}
\psi_{kl}=\pm h_{2n}^{n-k-l}h_{2,n-l+1}\sum_{j\in [1,l+1]\setminus2}\eta_j\det\widehat\Psi_{j\cup[l+2,n-k]}^{[2,n-k-l+1]},
\end{equation}
where $\eta_j=(-1)^{(n-j)(j-1)}{\psi_{n-j+1,j-1}}/{h_{2,n-j+2}}$ with $\psi_{n0}=h_{2,n+1}=1$. 

Now we can prove the claim of the lemma by induction on the maximal index in $I$. The minimum value of this index is $t$. In this case we use~\eqref{psiviaPsi} 
to see that
\begin{equation}\label{solidpsi}
\det\widehat\Psi_{[1,t]\setminus2}^{[2,t]}=\pm\psi_{n-t,1}/h_{2n}^{t}
\end{equation}
 is a Laurent monomial in component VI. 
Assume that the value of the maximal index in $I$ equals $r>t$. We multiply the
$(r-1)\times(t-1)$ matrix $\widehat\Psi^{[2,t]}_{[1,r]\setminus2}$  
by a $(t-1)\times (t-1)$ block upper triangular matrix with unimodular blocks of size $t-2$ and $1$, so that
the upper $(t-2)\times(t-1)$ submatrix is diagonalized with $1$'s on the diagonal except for the first row. Clearly, this transformation does not change the 
values of any minors in the first $t-2$ and $t-1$ columns. Consequently, each matrix entry (except for the one in the lower right corner, which we denote $z$) is a Laurent polynomial in component VI. We then consider~\eqref{psiviaPsi} with $k=n-r$ and $l=r-t+1$ and expand each minor in the right hand side by the last row. 
Each entry in the last row
distinct from $z$ enters this expansion with a coefficient that is a Laurent polynomial in component VI. 
The entry $z$ enters the expansion with the coefficient
\[
\pm h_{2n}^{t-1}h_{2,n-r+t}\sum_{j\in [1,r-t+2]\setminus2}\eta_j\det\widehat\Psi^{[2,t-1]}_{j\cup[r-t+3,r-1]}=\pm\psi_{n-r+1,r-t+1}h_{2n}.
\]
Thus, $z$ is a Laurent monomial in component VI, and hence any minor of $\widehat\Psi^{[2,t]}_{[1,r]\setminus2}$ that involves the $r$-th row is a Laurent polynomial.
\end{proof}

We can now proceed with restoring the entries of $\widehat M$. Equation~\eqref{psiviaPsi} for $k=n-l-1$ gives 
\[
\psi_{n-l-1,l}=
\pm h_{2n}h_{2,n-l+1}\sum_{j\in [1,l+1]\setminus2}\eta_j\widehat\Psi_j^2=\pm h_{2n}^2h_{2,n-l+1}\sum_{j\in [1,l+1]\setminus2}\eta_j\hat\mu_{j2},
\]
 and we consecutively
restore $\hat\mu_{j2}$, $1\le j\le n-1$, $j\ne2$, as Laurent polynomials at all clusters in component VI 
except for the neighbor of $\widetilde\Sigma'''_n$ in direction $\psi_{n-j+1,j-1}$. An easy computation shows that in this cluster $\psi_{n-j+1,j-1}$
is replaced by $\psi_{n-j+1,j-1}'=\pm h_{2n}\det\widehat M_{[2,j]}^{2\cup[n-j+3,n]}=\pm h_{2n}h_{2,n-j+3}\hat\mu_{j2}$, and hence $\hat\mu_{j2}$
is restored there as a Laurent monomial. In particular, it follows from above that
\begin{equation}\label{mu12}
\hat\mu_{12}=\pm\frac{\psi_{n-2,1}}{h_{2n}^3}
\end{equation}
in any cluster in component VI.

Recall that we have already restored the last row of $\widehat M$. We restore the other rows consecutively,
starting from row $n-1$ and moving upwards. 
Matrix entries in the $l$-th row, $3\le l\le n-1$, 
are restored in two stages, together with the minors $\det \widehat M^{i\cup[n-l+3,n]}_{[1,l-1]}$ for $1\le i\le n-l+1$, $i\ne 2$. 
First, we recover minors $\det \widehat M^{2\cup i\cup[n-l+3,n]}_{[1,l]}$ for $1\le i\le n-l+1$, $i\ne 2$, as Laurent polynomials in component VI. 
Once they are recovered,
we find $\hat\mu_{li}$ and $\det \widehat M^{i\cup[n-l+3,n]}_{[1,l-1]}$ via expanding $\det \widehat M^{2\cup i\cup[n-l+3,n]}_{[1,l]}$ and 
$\det \widehat M^{i\cup[n-l+2,n]}_{[1,l]}$ by the last row. This gives a system of two linear equations for a fixed $i$, 
and its determinant equals $\pm\psi_{n-l,l-1}$. Note that minors $\det \widehat M^{i\cup[n-l+2,n]}_{[1,l]}$ for $1\le i\le n-l$, $i\ne 2$, 
were recovered together with the entries of the $(l+1)$-st row, and $\det \widehat M^{[n-l+1,n]}_{[1,l]}=\pm\psi_{n-l,l}$ 
(recall that $\psi_{n-l,l}$ are Laurent polynomials in component VI). 
For $l=3$, we have  $\det \widehat M^{i\cup n}_{[1,2]}=\hat\mu_{1i}h_{2n}$, 
and hence the entries of the first row are recovered together with the entries of the third row.

The minors $\det \widehat M^{2\cup i\cup[n-l+3,n]}_{[1,l]}$ for $1\le i\le n-l+1$, $i\ne 2$, are recovered together with all other minors 
$\det \widehat M^{i\cup j\cup[n-l+3,n]}_{[1,l]}$ for $1\le i<j\le n-l+1$, $i,j\ne 2$,  altogether $(n-l)(n-l+1)/2$ minors. 
We first note that the Binet--Cauchy formula gives
\begin{equation}\label{psiviabinet}
\psi_{kl}=\sum_{\substack{J\subseteq [1,n-l]\setminus2\\ |J|=n-k-l-1}}(-1)^{\chi_J}\det\widehat M^{2\cup J\cup[n-l+1,n]}_{[1,n-k]}\det\widehat\Psi^{[2,n-k-l]}_J
\end{equation}
for $k+l<n-1$, where $\chi_J=\sum_{j\in J}\chi_j$. Recall that by Lemma~\ref{psiminors}, $\det\widehat\Psi^{[2,n-k-l]}_J$ are Laurent
polynomials in component VI. Writing down $\psi_{k,l-2}$ for $1\le k\le n-l$ via the above formula, and expanding the minors $\det\widehat M^{2\cup J\cup[n-l+3,n]}_{[1,n-k]}$ for $J\not\ni n-l+2$ by the last $n-k-l$ rows, we get $n-l$ linear equations; note that the corresponding minors for $J\ni n-l+2$ have been already restored as Laurent
polynomials in component VI when we dealt with the previous rows. For $l=n-1$ we get a single equation 
\[
\psi_{1,n-3}=-\det\widehat M_{[1,n-1]}^{[1,n]\setminus3}\hat\mu_{12}+\det\widehat M_{[1,n-1]}^{[2,n]}\hat\mu_{32},
\]
and hence $\det\widehat M_{[1,n-1]}^{[1,n]\setminus3}$ is a Laurent polynomial in component VI, since $\hat\mu_{12}$ is a Laurent monomial by~\eqref{mu12}.

In what follows we assume that $3\le l\le n-2$. 
The remaining $(n-l)(n-l-1)/2$ linear equations are provided by short Pl\"ucker relations
\begin{equation*}
\begin{split}
\det \widehat M^{i\cup j\cup[n-l+3,n]}_{[1,l]}\det \widehat M^{2\cup[n-l+2,n]}_{[1,l]}&=
\det \widehat M^{2\cup j\cup[n-l+3,n]}_{[1,l]}\det \widehat M^{i\cup[n-l+2,n]}_{[1,l]}\\
&+(-1)^{\chi_i+1}\det \widehat M^{2\cup i\cup[n-l+3,n]}_{[1,l]}\det \widehat M^{j\cup[n-l+2,n]}_{[1,l]},
\end{split}
\end{equation*}
where the second factor in each term is a Laurent polynomial in component VI, and, moreover, $\det \widehat M^{2\cup[n-l+2,n]}_{[1,l]}=\pm \psi_{n-l-1,l}/h_{2n}$.
We can arrange the variables in such a way that the matrix of the linear system takes the form $A=\begin{pmatrix} A_1 & A_2\\ A_3 & A_4 \end{pmatrix}$, where
$A_4$ is an $(n-l)(n-l-1)/2\times (n-l)(n-l-1)/2$ diagonal matrix with $\det \widehat M^{2\cup[n-l+2,n]}_{[1,l]}$ on the diagonal. 
The column $(i,j)$ of $A_2$ corresponding
to the variable $\det \widehat M^{i\cup j\cup[n-l+3,n]}_{[1,l]}$, $i,j\ne 2$, contains 
\begin{equation}\label{a2element}
\sum_{\substack{I\subseteq [1,n-l+1]\setminus\{2,i,j\}\\ |I|=t-1}}(-1)^{\theta(i,j,I)}\det\widehat M^{2\cup I}_{[l+1,l+t]}\det\widehat\Psi_{i\cup j\cup I}^{[2, t+2]}
\end{equation}
in row $n-l-t$ for $1\le t\le n-l-1$, where $\theta(i,j,I)=\chi_{i\cup I}+\#\{p\in 2\cup I: i<p<j\}$, and zero in row $n-l$. 
Similarly, the  column $(2,i)$ of $A_1$ corresponding to the variable $\det \widehat M^{2\cup i\cup[n-l+3,n]}_{[1,l]}$, $i\ne 2$, contains 
\[
\sum_{\substack{J\subseteq [1,n-l+1]\setminus\{2,i\}\\ |J|=t}}(-1)^{\theta(i,J)}\det\widehat M^{ J}_{[l+1,l+t]}\det\widehat\Psi_{i\cup J}^{[2, t+2]}
\]
in row $n-l-t$ for $1\le t\le n-l-1$, where 
\[
\theta(i,J)=\begin{cases} 1 \quad\text{for $i=1$},\\ 
                          \chi_J+\#\{p\in J: 2<p<i\} \quad\text{for $i\ne1$}, 
            \end{cases}
\]						
and $(-1)^{\chi_i}\hat\mu_{i2}$ in row $n-l$. 

To find $\det A$, we multiply $A$ by a square lower triangular matrix whose column $(2,i)$ contains $\hat\mu_{l+1,2}$ in row $(2,i)$, $\hat\mu_{l+1,j}$
in row $(j,i)$, $1\le j\le i-1$, $j\ne 2$, and $(-1)^{\chi_i+1}\hat\mu_{l+1,j}$ in row $(i,j)$, $i+1\le j\le n-l+1$, $j\ne 2$; column $(i,j)$ of this matrix
contains a single $1$ on the main diagonal, and all its other entries are equal to zero.
Let $A'=\begin{pmatrix} A'_1 & A_2\\ A'_3 & A_4 \end{pmatrix}$ be the obtained product. Clearly,
$\det A=\hat\mu_{l+1,2}^{l-n}\det A'$. We claim that $(A'_1)_{[1,n-l-1]}$ is a zero matrix, and $(A'_1)_{[n-l]}=\hat\mu_{l+1,2}(A_1)_{[n-l]}$.

The second claim follows immediately from the fact that $(A_2)_{[n-l]}$ is a zero vector. To prove the first one, we fix  arbitrary $i$, $t$, and $J=\{j_1<j_2<\cdots<j_t\}$,
and find the coefficient at $\det\widehat\Psi_{i\cup J}^{[2, t+2]}$ in the entry of $A_1'$ in row $n-l-t$ and column $(2,i)$. For $i=1$ this coefficient equals
\[
(-1)^{\theta(1,J)}\hat\mu_{l+1,2}\det\widehat M^J_{[l+1,l+t]}+\sum_{r=1}^t(-1)^{2+\theta(1,j_r,J\setminus j_r)}
\hat\mu_{l+1,j_r}\det\widehat M^{2\cup J\setminus j_r}_{[l+1,l+t]}.
\]
Taking into account that $\theta(1,J)=1$ and $\theta(1,j_r,J\setminus j_r)=2+\#\{p\in J:2<p<j_r\}=1+r$, we conclude that the signs in the above expression alternate.
For $i\ne 1$, the coefficient in question equals
\begin{equation*}
\begin{split}
 (-1)^{\theta(i,J)}\hat\mu_{l+1,2}\det\widehat M^J_{[l+1,l+t]}
&+\sum_{j_r<i}(-1)^{\theta(j_r,i,J\setminus j_r)}\hat\mu_{l+1,j_r}\det\widehat M^{2\cup J\setminus j_r}_{[l+1,l+t]}\\
&+\sum_{j_r>i}(-1)^{1+\theta(i,j_r,J\setminus j_r)}\hat\mu_{l+1,j_r}\det\widehat M^{2\cup J\setminus j_r}_{[l+1,l+t]}.
\end{split}
\end{equation*}
If $j_1=1$, then $\theta(i,J)=r^*$, where $r^*=r^*(i,J)=\max\{r: j_r<i\}$. Further, $\theta(1,i,J\setminus 1)=1+r^*$, $\theta(j_r,i,J\setminus j_r)=1+r^*-r$ for $1<j_r<i$
and $\theta(i,j_r,J\setminus j_r)=2+r-r^*$ for $i<j_r$, hence the signs  in the above expression taken in the order $1, 2, j_2,\dots, j_t$ alternate once again.
Finally, if $j_1>1$, then $\theta(i,J)=r^*$, $\theta(j_r,i,J\setminus j_r)=r^*-r$ for $j_r<i$ and $\theta(i,j_r,J\setminus j_r)=1+r-r^*$ for $i<j_r$, and again the signs
taken in the corresponding order alternate. Therefore, in all three cases the coefficient at $\det\widehat\Psi_{i\cup J}^{[2, t+2]}$ equals 
$\det\widehat M^{2\cup J}_{{l+1}\cup[l+1,l+t]}$, and hence 
vanishes as the determinant of a matrix with two identical rows.

The entry of $A_3'$ in column $(2,i)$ and row
$(i,j)$  is 
\begin{equation}\label{a3elementi}
\begin{split}
(-1)^{\chi_i+1}\hat\mu_{l+1,j}\det \widehat M^{2\cup[n-l+2,n]}_{[1l]}&+(-1)^{\chi_i}\hat\mu_{l+1,2}\det \widehat M^{j\cup[n-l+2,n]}_{[1l]}\\
&=(-1)^{\chi_i}\widehat M^{2\cup j\cup[n-l+2,n]}_{[1,l+1]},
\end{split}
\end{equation}
and the entry in column $(2,j)$ and row $(i,j)$ is
\begin{equation}\label{a3elementj}
\hat\mu_{l+1,i}\det \widehat M^{2\cup[n-l+2,n]}_{[1l]}-\hat\mu_{l+1,2}\det \widehat M^{i\cup[n-l+2,n]}_{[1l]}=
(-1)^{\chi_i+1}\widehat M^{2\cup i\cup[n-l+2,n]}_{[1,l+1]}.
\end{equation} 

By the Schur determinant lemma,
\[
\det A'=\det A_4\cdot\det\left(A_1'-A_2A_4^{-1}A_3'\right)=\pm \hat\mu_{l+1,2}\left(\frac{\psi_{n-l-1,l}}{h_{2n}}\right)^{(n-l)(n-l-3)/2+1}\det A'',
\]
where $A''$ is the $(n-l)\times (n-l)$ matrix satisfying $A''_{[1,n-l-1]}=(A_2A_3')_{[1,n-l-1]}$ and $A''_{[n-l]}=(A_1')_{[n-l]}$. 
Using~\eqref{a2element},~\eqref{a3elementi}, and~\eqref{a3elementj}, we compute
the entry $A''_{n-l-t,i}$ for $1\le t\le n-l-1$, $1\le i\le n-l+1$, $i\ne 2$, as
\begin{equation*}
\begin{split}
&\sum_{j<i}\sum_{\substack{I\subseteq [1,n-l+1]\setminus\{2,i,j\}\\ |I|=t-1}}(-1)^{\theta(j,i,I)+\chi_j+1}\det\widehat M^{2\cup I}_{[l+1,l+t]}
\det\widehat\Psi_{i\cup j\cup I}^{[2, t+2]}\det\widehat M_{[1,l+1]}^{2\cup j\cup [n-l+2,n]}\\
&+\sum_{j>i}\sum_{\substack{I\subseteq [1,n-l+1]\setminus\{2,i,j\}\\ |I|=t-1}}(-1)^{\theta(i,j,I)+\chi_i}\det\widehat M^{2\cup I}_{[l+1,l+t]}
\det\widehat\Psi_{i\cup j\cup I}^{[2, t+2]}\det\widehat M_{[1,l+1]}^{2\cup j\cup [n-l+2,n]}\\
&=\sum_{\substack{J\subseteq [1,n-l+1]\setminus\{2,i\}\\ |J|=t}}\det\widehat\Psi_{i\cup J}^{[2, t+2]}\\
&\qquad\times\left(\sum_{j_r<i}
(-1)^{\theta(j_r,i,J\setminus j_r)+\chi_{j_r}+1}\det\widehat M^{2\cup J\setminus j_r}_{[l+1,l+t]}\det\widehat M_{[1,l+1]}^{2\cup j_r\cup [n-l+2,n]}\right.\\
&\qquad\qquad+\left.\sum_{j_r>i} (-1)^{\theta(i,j_r,J\setminus j_r)+\chi_i+1}\det\widehat M^{2\cup J\setminus j_r}_{[l+1,l+t]}\det\widehat M_{[1,l+1]}^{2\cup j_r\cup [n-l+2,n]}
\right).
\end{split}
\end{equation*}
Analyzing the same three cases as in the computation of the entries of $A_1'$, we conclude that the second factor in the above expression can be rewritten as
\[
(-1)^{\chi_J+\chi_i+r^*+1}\sum_{r=1}^t(-1)^{r+1}\det\widehat M^{2\cup J\setminus j_r}_{[l+1,l+t]}\det\widehat M_{[1,l+1]}^{2\cup j_r\cup [n-l+2,n]}.
\]
To evaluate the latter sum, we multiply $\widehat M_{[1,l+t]}^{2\cup J\cup [n-l+2,n]}$ by a matrix $\begin{pmatrix} \one_{l+1} & 0\\ 0 & N \end{pmatrix}$, where
$N$ is unipotent lower triangular, so that the entries of column~2 in rows $l+2,\dots,l+t$ vanish. This operation preserves the minors involved in the above sum, and hence
the latter is equal to
\[
\sum_{r=1}^t(-1)^{r+1+\chi_J}\hat\mu_{l+1,2}\det\widehat M^{J\setminus j_r}_{[l+2,l+t]}\det\widehat M_{[1,l+1]}^{2\cup j_r\cup [n-l+2,n]}=
\hat\mu_{l+1,2}\det\widehat M_{[1,l+t]}^{2\cup J\cup [n-l+2,n]}.
\]
Thus, finally, 
\[
A''_{n-l-t,i}=\hat\mu_{l+1,2}\sum_{\substack{J\subseteq [1,n-l+1]\setminus\{2,i\}\\ |J|=t}}(-1)^{\chi_J+\chi_i+r^*+1}
\det\widehat\Psi_{i\cup J}^{[2, t+2]}\det\widehat M_{[1,l+t]}^{2\cup J\cup [n-l+2,n]}.
\]

To evaluate $\det A''$, we multiply $A''$ by the $(n-l)\times (n-l)$ upper triangular matrix having $(-1)^i\det\widehat\Psi_{[1,k]\setminus\{2,i\}}^{[2,k-1]}$ in row $i$ and column $k$, $i,k\in [1,n-l+1]\setminus 2$, $i\le k$; we assume that for $i=k=1$ this expression equals~1. Denote the obtained product by $A'''$ and note
that $\det A'''=\pm\det A''\prod_{k=3}^{n-l+1}\det\widehat\Psi^{[2,k-1]}_{[1,k-1]\setminus 2}$.
The entry $A'''_{n-l-t,k}$ equals
\begin{equation*}
\begin{split}
&\hat\mu_{l+1,2}\sum_{i\in [1,k]\setminus2}(-1)^i\det\widehat\Psi_{[1,k]\setminus\{2,i\}}^{[2,k-1]}\\
&\qquad\qquad\times\sum_{\substack{J\subseteq [1,n-l+1]\setminus\{2,i\}\\ |J|=t}}(-1)^{\chi_J+\chi_i+r^*+1}
\det\widehat\Psi_{i\cup J}^{[2, t+2]}\det\widehat M_{[1,l+t]}^{2\cup J\cup [n-l+2,n]}\\
&=\hat\mu_{l+1,2}\sum_{\substack{J\subseteq [1,n-l+1]\setminus2\\ |J|=t}}\det\widehat M_{[1,l+t]}^{2\cup J\cup [n-l+2,n]}\\
&\qquad\qquad\times\sum_{i\in [1,k]\setminus2}
(-1)^{\chi_J+\chi_i+i+r^*+1}\det\widehat\Psi_{[1,k]\setminus\{2,i\}}^{[2,k-1]}\det\widehat\Psi_{i\cup J}^{[2, t+2]};
\end{split}
\end{equation*}
the equality follows from the fact that $\det\widehat\Psi_{i\cup J}^{[2, t+2]}=0$ for $i\in J$. Expanding the latter determinant by the $i$-th row, we see that the entry 
in question equals
\begin{equation*}
\begin{split}
\hat\mu_{l+1,2}\sum_{\substack{J\subseteq [1,n-l+1]\setminus2\\ |J|=t}}&(-1)^{\chi_J+1}\det\widehat M_{[1,l+t]}^{2\cup J\cup [n-l+2,n]}\\
&\times\sum_{s=2}^{t+2}(-1)^s\det\widehat\Psi_{J}^{[2, t+2]\setminus s}\sum_{i\in [1,k]\setminus2}(-1)^{\chi_i+i}\hat\mu_{is}
\det\widehat\Psi_{[1,k]\setminus\{2,i\}}^{[2,k-1]}.
\end{split}
\end{equation*}
Note that the inner sum above equals 
\[
\det\widehat \Psi_{[1,k]\setminus 2}^{s\cup [2,k-1]}=
\begin{cases} 
0 , \qquad\qquad\quad 2\le s \le k-1,\\
\det\widehat \Psi_{[1,k]\setminus 2}^{[2,k]}, \quad s=k.
\end{cases}
\]
We conclude that $A'''_{n-l-t,k}=0$ if $t+2<k$, and hence $\det A'''$ equals the product of the entries on the main antidiagonal. The latter correspond to
$t+2=k$ and are equal to
\[
 (-1)^{t+1}\hat\mu_{l+1,2}\det\widehat\Psi_{[1,t+2]\setminus2}^{[2,t+2]}\sum_{\substack{J\subseteq [1,n-l+1]\setminus2\\ |J|=t}}(-1)^{\chi_J}
\det\widehat M_{[1,l+t]}^{2\cup J\cup [n-l+2,n]}\det\widehat\Psi_J^{[2,t+1]},
\]
which by~\eqref{psiviabinet} coincides with $\pm\hat\mu_{l+1,2}\det\widehat\Psi_{[1,t+2]\setminus2}^{[2,t+2]}\psi_{n-l-t,l-1}$. Therefore,
\[
\det A'''=\pm\hat\mu_{l+1,2}^{n-l-1}\prod_{t=1}^{n-l-1}\psi_{n-l-t,l-1}\prod_{t=0}^{n-l-1}\det\widehat\Psi_{[1,t+2]\setminus 2}^{[2,t+2]},
\]
and hence, taking into account~\eqref{solidpsi} for $t=n-l+1$, we get
\[
\det A=\pm \frac{\psi_{n-l-1,l}^{(n-l)(n-l-3)/2+1}\psi_{l-1,1}}{h_{2n}^{(n-l)(n-l-1)/2+2}}\prod_{t=1}^{n-l-1}\psi_{n-l-t,l-1}.
\]
So, all minors $\det \widehat M^{2\cup i\cup[n-l+3,n]}_{[1,l]}$ for $1\le i\le n-l+1$, $i\ne 2$ are Laurent polynomials in component VI, and hence so are all 
entries in the rows $1,3,4,\dots,n$ of $\widetilde M$. 
The entries of the second row
are restored via Lemma~\ref{row1viac} as polynomials in the entries of other rows and variables $c_i$ divided by the $\det K=\det\widehat\Psi=
\pm h_{11}\psi_{11}/h_{2n}^n$.

\begin{remark}\label{ultinohii}
Once we have established that stable variables $h_{ii}$ do not enter denominators of Laurent expressions for entries of $U$
at $\widetilde\Sigma'_n$ (see Remark~\ref{nohii}), this remains valid for any other cluster in $\N_0$. This is guaranteed by the fact that  entries of
$U$ are Laurent polynomials in any cluster in $\N_0$ and the 
condition that exchange polynomials are not divisible by any of the stable variables.
\end{remark}

\section{Auxiliary results from matrix theory}\label{aux}

\subsection{Proof of Proposition~\ref{polyrel}}
\label{long}

This proposition is an easy corollary of the following more general result.

\begin{proposition}
\label{long_identity}
Let $A$ be a complex $n\times n$ matrix. For  $u,v\in \C^n$, define matrices
\begin{gather*}\nonumber
 K(u)=\left[ u \; A u \; A^2 u \dots  A^{n-1} u\right],\\ 
 K_1(u,v)=\left[ v \; u  \; A u \dots A^{n-2} u\right],\;\; 
 K_2(u,v)=\left[A v \; u  \;A u  \dots  A^{n-2} u\right].
\end{gather*}
In addition, let $w$ be the last row of the classical adjoint of $ K_1(u,v)$, i.e.
$w  K_1(u,v) = \left (\det  K_1(u,v) \right )e_n^T$. Define
$ K^*(u,v)$ to be the matrix with rows $w, w A, \ldots, w A^{n-1}$. Then
\begin{equation*}
\label{longid}
\det\Big(\det K_1(u,v) A -  \det K_2(u,v) \one \Big ) = (-1)^{\frac{n(n-1)}{2}} \det K(u) \det K^*(u,v).
\end{equation*}
\end{proposition}

\begin{proof} It suffices to prove the identity for generic $A, u$ and $v$. In particular, we can assume that $A$ has distinct eigenvalues.
Then, after a change of basis, we can reduce the proof to the case of a diagonal $A=\diag(a_1,\ldots, a_n)$ and vectors $u=(u_i)_{i=1}^n$ and
$v=(v_i)_{i=1}^n$ with all entries non-zero.  In this case,
$$
\det  K(u) = \Van A \prod_{j=1}^n u_j, \qquad \det  K^*(u,v) = \Van A \prod_{j=1}^n w_j,
$$
where $\Van A$ is the $n\times n$ 
Vandermonde determinant based on $a_1,\ldots, a_n$ and $w_j$ is the $j$th component of the row vector $w$.
The the left-hand side of the identity can be rewritten as
$$
\prod_{j=1}^n \left (a_j \det  K_1(u,v) - \det  K_2(u,v) \right ) =
\prod_{j=1}^n \det \left[ (a_j\one - A)v\  u \ Au \ldots\ A^{n-2} u  \right ].
$$

We compute the $j$th factor in the product above as
\begin{equation}\label{Pj}
\begin{split}
P_j&=\sum_{\alpha\ne j} (-1)^{\alpha + 1} (a_j - a_\alpha) {v_\alpha} \Van A_{\{\alpha\}}
\prod_{i\ne\alpha} u_i \\
&=\left(u_j \prod_{\beta\ne j} (a_j -a_\beta)\right) \sum_{\alpha\ne j} (-1)^{\alpha + 1} 
(-1)^{n-j -\theta(\alpha - j)}  {v_\alpha}\Van A_{\{\alpha,j\}}\prod_{i\ne\alpha,j} u_i,
\end{split}
\end{equation}
where  $\Van A_I$ is the $(n-|I|)\times (n-|I|)$ 
Vandermonde determinant based on $a_i$, $i\notin I$, and $\theta(\alpha - j)$ is 1 if $\alpha > j$ and $0$ otherwise.

Note that
$$
\prod_{j=1}^n  u_j \prod_{\beta\ne j} (a_j -a_\beta) = (-1)^{\frac{n(n-1)}{2}}(\Van A)^2\prod_{j=1}^n u_j = (-1)^{\frac{n(n-1)}{2}}\Van A \det  K(u),
$$
and that the minor of $ K(u)$ obtained by deleting the last two columns and rows $\alpha$ and $j$ is
$$
\det  K(u)^{[1,n-2]}_{[1,n]\setminus\{\alpha,j\}} =  \Van A_{\{\alpha,j\}}\prod_{i\ne\alpha,j} u_i.
$$
Therefore
\begin{equation*}
\begin{aligned}
w_j =& (-1)^{n+j} \det  K_1(u,v)^{[1,n-1]}_{[1,n]\setminus\{j\}}\\ 
=&(-1)^{n+j} \sum_{\alpha\ne j} 
(-1)^{ \alpha + 1-\theta(\alpha - j)}v_\alpha \det  K(u)^{[1,n-2]}_{[1,n]\setminus\{\alpha,j\}}.
\end{aligned}
\end{equation*}

Comparing with~\eqref{Pj}, we obtain 
\begin{align*}
P_1&\cdots P_n \\
&= (-1)^{\frac{n(n-1)}{2}} \det  K(u) \Van A \prod_{j=1}^n \sum_{\alpha\ne j} 
(-1)^{n + \alpha - j + 1-\theta(\alpha - j)}v_\alpha \det  K(u)^{[1,n-2]}_{[1,n]\setminus\{\alpha,j\}}\\
& = (-1)^{\frac{n(n-1)}{2}} \det  K(u) \Van A \prod_{j=1}^n w_j = (-1)^{\frac{n(n-1)}{2}} \det  K(u) \det  K^*(u,v),
\end{align*}
as needed.
\end{proof}

Specializing Proposition~\ref{long_identity} to the case $A=X^{-1} Y$, $u=e_n$, $v= e_{n-1}$, one obtains Proposition~\ref{polyrel}.

\subsection{Normal forms}

In this section we derive several normal forms that were used in the main body of the paper.

\begin{lemma} 
\label{BC}
For a generic $n\times n$ matrix $U$ there exist a unique unipotent lower triangular matrix $N_-$ and an upper triangular matrix $B_+$ such that
\begin{equation*}
U = N_- B_+ C N_-^{-1},
\end{equation*}
where $C=e_{21} + \cdots + e_{n, n-1} + e_{1n}$ is the cyclic permutation matrix.
\end{lemma}
\begin{proof} 
Let $N_1= \one - \sum_{i=2}^n u_{in}/u_{1n}$. Then $U_1= N_1 U N_1^{-1}$ is a matrix whose  last column has the first entry equal to $u_{1n}$ and all other
entries equal to zero. Next, let $N_2$  be  a unipotent lower triangular matrix  such that 

(i) off-diagonal entries in the first column of $N_2$ are zero, and 

(ii) $U_2=N_2 U_1 N_2^{-1}$ is in the upper Hessenberg form, that, is has zeroes below the first subdiagonal. 

\noindent
Then $U_2$ has a required form $B_+C$, where $B_+$ is upper triangular.
To establish uniqueness, it is enough to show that if $B_1, B_2$ are invertible upper triangular matrices and $N$ is lower unipotent matrix, such that
$N B_1 C = B_2 C N$, then $N=\one$. Comparing the last columns on both sides of the equality, it is easy to see that off-diagonal elements in the first column of $N$ are zero. Then, comparison
of first columns implies that the same is true for off-diagonal elements in the second column of $N$, etc.
\end{proof}

\begin{lemma} 
\label{BW}
For a generic $n\times n$ matrix $U$ there exist a unique unipotent lower triangular matrix $N_-$ and an upper triangular matrix $B_+$ such that
\begin{equation*}
U = N_- B_+ W_0 N_-^{-1},
\end{equation*}
where $W_0$ is the matrix of the longest permutation.
\end{lemma}
\begin{proof} 
Equality $U N_- = N_- B_+ W_0$ implies $W_0 U\cdot N_- = W_0 N_- W_0\cdot W_0 B_+ W_0$. Using the uniqueness of the Gauss factorization, we obtain
$W_0 N_- W_0 =\left (W_0 U\right )_{> 0} $ and $W_0 B_+ W_0 = \left (W_0 U\right )_{\leq 0} N_-$, and thus recover 
$N_-= W_0 \left (W_0 U\right)_{>0}W_0 $  and $B_+= W_0\left (W_0 U\right)_{\leq 0} W_0 \left (W_0 U\right)_{>0}$.
\end{proof}

\begin{lemma}
\label{NMN}
For a generic $n\times n$ matrix $U$ there exists a unique representation 
\begin{equation}
\label{normalNMN}
U = \left (\one_n + \nu e_{12}\right ) N_- M   N_-^{-1} \left (\one_n - \nu e_{12}\right ),
\end{equation}
where $N_-=(\nu_{ij})$ is a unipotent lower triangular matrix with $\nu_{j1}=0$ for $2\le j\le n$ and $M=(\mu_{ij})$ has $\mu_{1n}=0$
and  $\mu_{i,n+2-j}=0$ for $2\leq j < i\leq n$.
 \end{lemma}

\begin{proof}  First, observe that if~\eqref{normalNMN} is valid, then $\mu_{2n}=u_{2n}$, and $\nu ={u_{1n}} /{u_{2n}}$. Denote 
$U'= \left (\one_n - \nu e_{12}\right ) U\left (\one_n + \nu e_{12}\right )$. Then~\eqref{normalNMN} implies that the first row and column 
of $M$ coincide with those of $U'$, and
that $M_{[2,n]}^{[2,n]}$ and $N$ are uniquely determined by applying Lemma~\ref{BW} to ${U'}_{[2,n]}^{[2,n]}$.
\end{proof}

\subsection{Matrix entries via eigenvalues}

For an $n\times n$ matrix $A$, denote $K= K(e_1)$, where $ K(u)$ is defined in Proposition~\ref{long_identity}.

\begin{lemma} \label{row1viac}
Every matrix entry in the first row of $A$ can be expressed as ${P}/{\det K}$, where $P$ is a polynomial in matrix entries
of the last $n-1$ rows of $A$ and coefficients of the characteristic polynomial of $A$.
\end{lemma}

\begin{proof} Let $\det (\lambda \one - A) = \lambda^n + c_1 \lambda^{n-1} + \cdots +c_n$. 
Compare two expressions for the first row $R(\lambda)$ of the resolvent $(\lambda \one - A)^{-1}$ of $A$:
$$
R(\lambda)= \sum_{j\geq 0} \frac{1}{\lambda^{j+1}} A^j e_1= \frac{1}{\det (\lambda \one - A) } \left ( \lambda^{n-1} e_1 + \lambda^{n-2} Q_2 +\cdots + Q_{n-1} \right ),
$$
where $Q_i$, $2\le i\le n-1$, are vectors polynomial in matrix entries of the last $n-1$ rows of $A$. Cross-multiplying by $\det (\lambda \one - A)$ and comparing coefficients at positive degrees of $\lambda$ on both sides, we obtain
$$
c_i e_1+ c_{i-1} A e_1 + \cdots + c_{1} A^{i-1} e_1 + A^{i} e_1 = Q_{i}, \qquad  2\le i\le n-1.
$$
Let $T$ be an $n\times n$ upper triangular unipotent Toeplitz matrix with entries of the $i$th superdiagonal equal to $c_i$, and let $Q$ be the matrix
with columns $e_1, Q_2,\ldots, Q_{n-1}$. Then the equations above can be re-written as a single matrix equation $K T = Q$. Note also that $A K = K C_A$, where $C_A$ is a companion matrix of $A$ with $1$s on the first subdiagonal, $-c_n,\ldots, -c_{1}$ in the last column, and $0$s everywhere else.
Therefore,  $A K T = K T\cdot T^{-1} C_A T = Q T^{-1} C_A T$ (note that $C'_A=T^{-1} C_A T= W_0 C_A^T W_0$ is an alternative companion form of $A$). 
Consequently, the first row  $A_{[1]}$ of $A$ satisfies the linear equation $A_{[1]} K T = Q_{[1]} C_A'$, and the claim follows.
\end{proof}

\begin{lemma}
\label{Hessenberg} 
Let $H$ be an $n\times n$ upper Hessenberg matrix, $C^H=\left [ \begin{array}{cc} \star & \star\\ \one_{n-1}& 0\end{array} \right ]$ be its upper Hessenberg companion  and $N=(\nu_{ij})$ be a unipotent upper triangular matrix such that $N^{-1} H N = C^H$. Then the row vector $\nu=(-\nu_{1i})_{i=2}^n$ coincides with the first row of the companion form of $\tilde H=H_{[2,n]}^{[2,n]}$.
\end{lemma}

\begin{proof} Factor $N$ as $N=\left [ \begin{array}{cc} 1 & 0\\ 0 & \tilde N\end{array} \right ] \left [ \begin{array}{cc} 1 & \nu\\ 0 & \one_{n-1}\end{array} \right ] $. Then
\[
{\tilde N}^{-1}  \tilde H \tilde N = \left ( \left [ \begin{array}{cc}  \one_{n-1} & 0 \end{array} \right ] \left [ \begin{array}{cc} 1 & -\nu\\ 0 & \one_{n-1}\end{array} \right ]  \right )^{[2,n]},
\]
 and the claim follows.
\end{proof}

\section*{Acknowledgments}

M.~G.~was supported in part by NSF Grant DMS \#1362801. 
M.~S.~was supported in part by NSF Grants DMS \#1362352.  
A.~V.~was supported in part by ISF Grant \#162/12. 
The authors would like to thank the following institutions for support and excellent working conditions: 
 Max-Planck-Institut f\"ur Mathematik, Bonn (M.~G., Summer 2014), 
Institut des Hautes \'Etudes Scientifiques (A.~V., Fall 2015), Stockholm University and Higher School of Economics, Moscow (M.~S., Fall 2015), Universit\'e  Claude Bernard Lyon~1 and Universit\'e Paris Diderot (M.~S., Spring 2016).
This paper was completed during the joint visit of the authors to the University of Notre Dame London Global Gateway in May 2016. The authors are grateful to this 
institution for warm hospitality. Special thanks are due to A.~Berenstein, A.~Braverman, Y.~Greenstein, D.~Rupel, G.~Schrader and A.~Shapiro for valuable discussions and to an anonymous referee for helpful comments.


\begin{thebibliography}{00}

\bibitem{BD} A.~Belavin and V.~Drinfeld,
\textit{Solutions of the classical Yang-Baxter equation for simple Lie algebras}.
Funktsional. Anal. i Prilozhen. {\bf16} (1982), 1--29.


\bibitem {CAIII}  A.~Berenstein, S.~Fomin, and A.~Zelevinsky,
\textit{Cluster algebras. III. Upper bounds and double Bruhat cells}. 
Duke Math. J. \textbf{126} (2005), 1--52.

\bibitem{Boch} M.~Bocher,
\textit{Introduction to higher algebra}. Dover Publications, 2004.

\bibitem{Bra} R.~Brahami,
\textit{Cluster $\chi$-varieties for dual Poisson-Lie groups. I.}
Algebra i Analiz \textbf{22} (2010), 14--104.

\bibitem{CP} V.~Chari and A.~Pressley, \textit{A guide to quantum groups}.
Cambridge University Press, 1994.

\bibitem{CheSha} L.~Chekhov and M.~Shapiro, 
\textit{ Teichm\"uller spaces of Riemann surfaces with orbifold points of arbitrary order and cluster variables}. IMRN (2014), no.~10, 2746--2772.

\bibitem{EvLu} S.~Evens and J.-H.~Lu, \textit{Poisson geometry of the Grothendieck resolution of a complex semisimple group}. Mosc. Math. J. \textbf{7} (2007), 613--642. 

\bibitem{FoPy}
S.~Fomin and P.~Pylyavskyy, {\it Tensor diagrams and cluster algebras}.
Adv. Math. \textbf{300} (2017), 717--787.

\bibitem{FoRe}
S.~Fomin and N.~Reading, {\it Root systems and generalized associahedra}.
In: Geometric Combinatorics, AMS, Providence, RI, 2007, 63--131.

\bibitem{FoZe}
S.~Fomin and A.~Zelevinsky, {\it The Laurent phenomenon}.
Adv. Appl. Math. \textbf{28} (2002), 119--144. 

\bibitem{Fra} 
C.~Fraser, {\it Quasi-homomorphisms of cluster algebras}. 
Adv. Appl. Math. \textbf{81} (2016), 40--77.

\bibitem{Gant} 
F.~Gantmacher,
\textit{The theory of matrices, vol.1}. American Mathematical Society, Providence, RI, 1998. 

\bibitem{LMP} M.~Gekhtman, M.~Shapiro, A.~Stolin, and A.~Vainshtein,
\textit{Poisson structures compatible with the cluster algebra structures in Grassmannians}.
Lett. Math. Phys. \textbf{100} (2012), 139--150.

\bibitem{GSV1}  M.~Gekhtman, M.~Shapiro, and A.~Vainshtein,
\textit{Cluster algebras and Poisson geometry}.  
Mosc. Math. J. \textbf{3} (2003), 899--934.

\bibitem{GSVb}  M.~Gekhtman, M.~Shapiro, and A.~Vainshtein,
\textit{Cluster algebras and Poisson geometry}.
Mathematical Surveys and Monographs, 167. American Mathematical Society, Providence, RI, 2010.

\bibitem{GSVMMJ}  M.~Gekhtman, M.~Shapiro, and A.~Vainshtein,
\textit{Cluster structures on simple complex Lie groups and Belavin--Drinfeld classification}.
Mosc. Math. J. \textbf{12} (2012), 293--312.

\bibitem{GSVPNAS}  M.~Gekhtman, M.~Shapiro, and A.~Vainshtein,  {\it Cremmer-Gervais cluster structure on $SL_n$}.   
Proc. Natl. Acad. Sci. USA {\bf 111} (2014), no.~27, 9688--9695. 

\bibitem{GSVcr} M.~Gekhtman, M.~Shapiro, and A.~Vainshtein, 
\textit{Generalized cluster structure on the Drinfeld double of $GL_n$}.
C.~R.~Math. Acad. Sci. Paris \textbf{354} (2016), 345--349.

\bibitem{GSVMem}  M.~Gekhtman, M.~Shapiro, and A.~Vainshtein,  {\it Exotic cluster structures on $SL_n$: the Cremmer--Gervais case}.  
Memoirs of the AMS \textbf{246} (2017), no.~1165, 94pp.

\bibitem{r-sts}  A.~Reyman and M.~Semenov-Tian-Shansky,
\textit{Group-theoretical methods in the theory of
finite-dimensional integrable systems}. Encyclopaedia of
Mathematical Sciences, vol.16, Springer--Verlag, Berlin, 1994 pp.116--225.

\bibitem{Ya} M.~Yakimov, 
{\it Symplectic leaves of complex reductive Poisson-Lie groups}. Duke Math. J. \textbf{112} (2002),
453--509.


\end{thebibliography}
\end{document}